\documentclass[11pt,reqno]{amsart}
\usepackage{amsmath,amssymb,ifthen,mathrsfs}
\usepackage[colorlinks]{hyperref}
\usepackage{fullpage,verbatim}
\usepackage[normalem]{ulem}
\newcommand\tsout{\bgroup\markoverwith{\textcolor{red}{\rule[0.5ex]{2pt}{1.4pt}}}\ULon}
\newcommand{\stkout}[1]{\ifmmode\text{\tsout{\ensuremath{#1}}}\else\tsout{#1}\fi}
\usepackage{graphicx,psfrag,caption,subcaption,color}
\usepackage{setspace}
\usepackage{multicol}
\usepackage[margin=1in]{geometry}
\usepackage{pgfplots}
\usepackage{booktabs}
\pgfplotsset{compat=1.18}

\theoremstyle{definition}
\newtheorem{theorem}{Theorem}[section]
\newtheorem{lemma}[theorem]{Lemma}

\newtheorem{algorithm}[theorem]{Algorithm}
\newtheorem{definition}[theorem]{Definition}

\newtheorem{remark}[theorem]{Remark}
\numberwithin{equation}{section}

\newcommand{\bff}{\boldsymbol}
\newcommand{\bb}{\mathbb}
\newcommand{\leqs}{\lesssim}

\newcommand{\ddt}{\frac{\mathrm{d}}{\mathrm{d}t}}
\newcommand{\dx}{\mathrm{d}\bff{x}}
\newcommand{\ds}{\mathrm{d}s}
\newcommand{\dy}{\mathrm{d}\bff{y}}

\newcommand{\pa}{\partial}
\newcommand{\nn}{\nonumber}
\newcommand{\norm}[2]{\left\|{#1}\right\|_{#2}}
\newcommand{\inpro}[2]{\left\langle#1,#2\right\rangle}
\newcommand{\abs}[1]{\left|{#1}\right|}

\allowdisplaybreaks

\usepackage[authoryear,round,semicolon]{natbib}

\begin{document}
\title{
The Landau--Lifshitz--Bloch equation on polytopal domains: Unique existence and
finite element approximation
}
\author{Kim-Ngan Le}
\address{School of Mathematics,
         Monash University,
         Clayton 3800, Australia}
\email{\textcolor[rgb]{0.00,0.00,0.84}{ngan.le@monash.edu}}

\author{Agus L. Soenjaya}
\address{School of Mathematics and Statistics, The University of New South Wales, Sydney 2052, Australia}
\email{\textcolor[rgb]{0.00,0.00,0.84}{a.soenjaya@unsw.edu.au}}

\author{Thanh Tran}
\address{School of Mathematics and Statistics,
         The University of New South Wales,
         Sydney 2052, Australia}
\email{\textcolor[rgb]{0.00,0.00,0.84}{thanh.tran@unsw.edu.au}}

\date{\today}

\begin{abstract}
The Landau--Lifshitz--Bloch equation (LLBE) describes the evolution of the
magnetic spin field in ferromagnets at high temperatures. In this paper, we study the numerical approximation of the LLBE {in the regime above the Curie temperature} on bounded polytopal domains in $\bb{R}^d$, with {$d\in \{1,2,3\}$}, allowing for non-convex domains when $d=2$. We first establish the existence and uniqueness of strong solutions to the LLBE and propose a linear, fully discrete, conforming finite element scheme for its approximation. While this scheme is shown to converge, the obtained rate is suboptimal. To address this shortcoming, we introduce a viscous
(pseudo-parabolic) regularisation of the LLBE, which we call the $\epsilon$-LLBE. For this regularised problem, we prove the unique existence of strong solutions and establish a rate of convergence of the solution
$\bff{u}^\epsilon$ of the $\epsilon$-LLBE to the
solution $\bff{u}$ of the LLBE as $\epsilon\to 0^+$. Furthermore, we propose a
linear, fully discrete, conforming finite element scheme to approximate the solution of the~$\epsilon$-LLBE. Given sufficiently smooth initial data, error analysis is performed to show stability and uniform-in-time convergence of the scheme. Finally, several numerical simulations are presented to corroborate our theoretical results.
\end{abstract}
\maketitle
\tableofcontents

\section{Introduction}
Micromagnetic modeling has proved itself to be an effective theoretical framework in describing magnetisation dynamics at sub-micrometre length scales.
The framework of Landau--Lifshitz--Bloch equation
(LLBE)~\citep{Gar91,Gar97} overcomes
a problem of the standard model of magnetisation which carries a significant
limitation of being valid only at very low temperatures,
far below the Curie temperature~$T_{\text{c}}$~\citep{ChuNie20}. 
The LLBE essentially interpolates between the standard model at low temperatures and 
the Ginzburg-Landau theory of phase transitions.
It is valid not only below, but also above the temperature $T_{\text{c}}$. The LLBE has been successfully used to analyse magnetisation dynamics at high temperature~\citep{AtxHinNow16, ChuNie20, ChuNowChaGar06, NieChu16}. It is also widely used in the physics literature to model the heat-assisted magnetic recording (HAMR), a modern hard drive technology recently made available commercially~\citep{MeoPan20, ZhuLi13}.

In this paper, we consider a deterministic form of the LLBE 
in which the temperature is {raised higher than $T_{\text{c}}$}. We assume that the effective field is dominated by exchange interactions, although it is straightforward to include the lower-order uniaxial anisotropy field and a time-varying external field in the equation.
The macroscopic magnetisation $\bff{u}(t,\bff{x})\in\bb{R}^3$ at time $t\geq 0$ and
$\bff{x}\in \mathscr{D}\subset \bb{R}^d$, where $d=1,2,3$, satisfies the
problem~\citep{ChuNowChaGar06, Le16}:
\begin{subequations}\label{equ:LLB pro}
\begin{alignat}{2}
&\partial_t \bff{u} 
= 
\kappa_1\Delta\bff{u}
+
\gamma\bff{u}\times  \Delta\bff{u}
-
\kappa_2(1+\mu|\bff{u}|^2)\bff{u}
\qquad&
&\text{in } (0,T)\times \mathscr{D},
\label{equ:LLB}
\\
&\bff{u}(0,\cdot)=\bff{u}_0
\qquad&
&\text{in } \mathscr{D},
\label{equ:u ini con}
\\
&\frac{\pa\bff{u}}{\pa\bff{\nu}} = 0
\qquad&
&\text{on } (0,T)\times\pa \mathscr{D}.
\label{equ:u bou con}
\end{alignat}
\end{subequations}
Here, $\partial\mathscr{D}$ denotes the boundary of a bounded polytopal domain~$\mathscr{D}$, which is assumed to be convex when $d=3$, but can be non-convex when $d=2$. Moreover, $\bff{\nu}$ is the outward normal unit vector.
The positive coefficients~$\kappa_1$ and $\kappa_2$ are phenomenological {temperature-dependent} damping
parameters, while $\gamma$ is the gyromagnetic ratio (which can be positive or
negative depending on the precessional direction of the particle). The positive constant
$\mu$ is proportional to $T_\mathrm{c}/(\Theta-T_\mathrm{c})$, {where $\Theta$ is the ferromagnet temperature}. {Although the LLBE can be coupled to a thermal transport equation to incorporate temperature dynamics~\citep{AtxChuWalMan10}, we consider here an isothermal regime and treat the temperature as a fixed parameter}.
At this juncture, we refer the reader to Section~\ref{sec:notation} for further notations that are used in this paper.

Some available mathematical results related to the theoretical or numerical analysis for the LLBE will be reviewed next. The existence of a weak solution to the LLBE in smooth bounded domains is shown in~\citep{Le16}, while the unique existence of a strong solution to the LLBE coupled with the Maxwell equation is studied in~\citep{LiGuoLiu21} for 2-D periodic domains. To the
best of our knowledge, existence and uniqueness of~\emph{strong} solutions to the LLBE on \emph{polytopal} domains have not been studied before.

In the context of numerical analysis for the LLBE, several finite element schemes have recently been proposed to solve~\eqref{equ:LLB pro}. A nonlinear finite element scheme is introduced in~\citep{BenEssAyo24}, where conditional weak convergence is established, albeit without a convergence rate and under much stronger assumptions on the mesh than those considered in this work. Another scheme is presented in~\citep{Soe25}, which employs a different regularisation strategy (specifically, a fourth-order Baryakhtar-type~\citep{SoeTra23} regularisation) to develop an energy-stable mixed finite element method. This method achieves unconditional convergence with optimal rates. However, both schemes are nonlinear in the sense that a system of nonlinear equations must be solved at each time step. Moreover, the method in~\citep{Soe25} requires very strong regularity assumptions on the exact solution. 

Our first contribution in this paper is to address these gaps by first establishing the well-posedness of the LLBE on polytopal domains, extending the results of~\citep{Le16} in several directions: we allow non-smooth domains, prove the unique existence of strong (and more regular) solutions, and derive a decay estimate as the time variable $t\to \infty$. Building on this analytical foundation, we propose a linear finite element scheme for approximating the solution to \eqref{equ:LLB pro} directly. In contrast to the aforementioned schemes in~\citep{BenEssAyo24, Soe25}, our method is linear and requires solving only a linear system at each time step, resulting in significantly reduced computational cost. Moreover, under an adequate regularity assumption on the initial data, we prove the convergence of the scheme. We note that unlike the standard Landau--Lifshitz equation, whose solutions are constrained to lie on the unit sphere~\citep{Gil55,LL35}, a distinguishing feature of the LLBE is that the magnetisation magnitude is not conserved. 
Consequently, numerical methods developed for the standard Landau--Lifshitz equation, which crucially rely on the sphere constraint, are not applicable to the LLBE.

The linear finite element scheme mentioned above (see Section~\ref{sec:fem LLB}) requires a restriction on the ratio of time-step size to mesh size, assumes strong regularity of the exact solution, and yields only suboptimal convergence rates due to a lack of stability of the discrete solution in $\bb{L}^\infty$ and higher-order Sobolev norms. We note that a more stringent condition is imposed in~\citep{BenEssAyo24}.
To address these shortcomings, we introduce a viscous regularisation of \eqref{equ:LLB pro} by adding the term $-\epsilon
\Delta\partial_t \bff{u}$, where $\epsilon \in (0,1)$ is a small
parameter. The resulting equation, satisfied by the vector field
$\bff{u}^\epsilon:(0,T)\times \mathscr{D} \to \bb{R}^3$, is called the $\epsilon$-LLBE. This regularised problem is formulated as:
\begin{subequations}\label{equ:LLB eps pro}
	\begin{alignat}{2}
		&\partial_t \bff{u}^\epsilon
		-
		\epsilon \Delta \partial_t \bff{u}^\epsilon
		= 
		\kappa_1\Delta\bff{u}^\epsilon
		+
		\gamma\bff{u}^\epsilon\times  \Delta\bff{u}^\epsilon
		-
		\kappa_2(1+\mu|\bff{u}^\epsilon|^2)\bff{u}^\epsilon
		\qquad&
		&\text{in } (0,T)\times \mathscr{D},
		\label{equ:LLB eps}
		\\
		&\bff{u}^\epsilon(0,\cdot)=\bff{u}_0^\epsilon
		\qquad&
		&\text{in } \mathscr{D},
		\label{equ:ue ini con}
		\\
		&\frac{\pa\bff{u}^{\epsilon}}{\pa\bff{\nu}} = 0
		\qquad&
		&\text{on } (0,T)\times \pa \mathscr{D}.
		\label{equ:ue bou con}
	\end{alignat}
\end{subequations}
Our second contribution in this paper is to show a rate of convergence of the solution
$\bff{u}^\epsilon$ of~\eqref{equ:LLB eps pro} to the
solution $\bff{u}$ of~\eqref{equ:LLB pro} as $\epsilon\to 0^+$, provided the initial data $\bff{u}_0^\epsilon$ is appropriately chosen (for instance $\bff{u}_0^\epsilon= \bff{u}_0$). Similar ideas to
approximate the solution of a PDE by its viscous (pseudo-parabolic)
regularisation have been applied before, for instance in~\citep{NovPeg91, ShoTin70} for a
class of nonlinear parabolic equations, as well as in~\citep{BBM72} and~\citep{AngTra90} for the BBM
equation.

Our third contribution in this paper is the development of a stable and optimally convergent linear finite element scheme for the regularised $\epsilon$-LLBE. In contrast to previous works, it is worth emphasising that we consider exclusively polytopal domains, allowing for re-entrant corners in the two-dimensional case. 

To support the numerical analysis, we prove the unique existence of global strong solutions to the $\epsilon$-LLBE by deriving key energy estimates, which are obtained by combining elliptic regularity theorems and embedding properties of fractional Sobolev spaces in Lipschitz domains. It is known that in domains with re-entrant corners, {the} singularity at the corner pollutes the finite element solution everywhere if globally quasi-uniform triangulations are used, leading to a lower order of convergence~\citep{ChaLazThoWah06, LiZha17}. We confirm this behaviour for the problem at hand, and further derive an error estimate that holds uniformly in time, analogous to the result developed in~\citep{Lar89} for semilinear parabolic equations. Notably, the extension to our highly nonlinear setting is nontrivial.

We remark that another regularisation (by adding the term $\epsilon
\Delta^2 \bff{u}$) is considered in~\citep{GolJiaLe22} for the stochastic LLBE, turning the equation into a
fourth-order nonlinear stochastic PDE. A numerical method is then proposed to
solve the equation and the rate of convergence to the exact solution is shown
for $d\in\{1,2\}$. Since we intend to work with a
conforming finite element method, such regularisation is not utilised here to avoid using a $C^1$-conforming finite element for a
fourth-order problem, which can be computationally costly and difficult to
implement, especially for $d=3$.

In summary, the main contributions of this paper are the following results:
\begin{enumerate}
	\item existence and uniqueness of strong and regular solutions $\bff{u}$ to~\eqref{equ:LLB
		pro} in polytopal domains (Theorem~\ref{the:con u eps} and~\ref{the:unique llb});
    \item long-term behaviour of~$\bff{u}$,
		more precisely, $\norm{\bff{u}(t)}{\bb{L}^\infty(\mathscr{D})} \leq e^{-\kappa_2 t} \norm{\bff{u}_0}{\bb{L}^\infty(\mathscr{D})}$ (Theorem~\ref{the:ut inf});
    \item suboptimal error estimates of the finite element approximation (Algorithm~\ref{alg:fem llbe}) to the solution of \eqref{equ:LLB pro} with a realistic assumption on the regularity of the solution (Theorem~\ref{the:err llb});
	\item existence and uniqueness of strong and regular solutions $\bff{u}^\epsilon$
		to the regularised problem~\eqref{equ:LLB eps pro} in polytopal domains (Theorem~\ref{the:ue sol} and~\ref{the:unique llbe});
	\item rate of convergence in $L^\infty(0,T;\bb{H}^1(\mathscr{D}))$ of
		$\bff{u}^\epsilon$ to the solution $\bff{u}$ of \eqref{equ:LLB pro} as
		$\epsilon\to 0^+$ (Theorem~\ref{the:u eps con u}); 
	\item optimal error estimates of the finite element approximation (Algorithm~\ref{alg:fem eps llbe}) to the solution of the $\epsilon$-LLBE problem \eqref{equ:LLB eps pro} with a realistic assumption on the solution regularity (Theorem~\ref{the:err} and~\ref{rem:error});
	\item uniform-in-time error estimates of
		the proposed linear finite element scheme for
		the $\epsilon$-LLBE problem (Theorem~\ref{the:err long}),
		which are verified by a series of numerical experiments in Section~\ref{sec:num sim}.
\end{enumerate}
It is worth noting that physical theory indicates a phase transition from the ferromagnetic to the paramagnetic state above the Curie temperature $T_\mathrm{c}$, where the equilibrium net magnetisation vanishes~\citep{ChuNowChaGar06, KazEtal08, VogAbe14}. The long-time behaviour of the macroscopic magnetisation $\bff{u}$, as established in Theorem~\ref{the:ut inf}, is consistent with this result.

{We remark that if $\mathscr{D}$ has a smooth boundary, then we may approximate the boundary $\partial\mathscr{D}$ with a polytopal boundary $\partial \mathscr{D}_h$. Since $\partial\mathscr{D}_h$ may not coincide with $\partial\mathscr{D}$, additional domain approximation assumptions and more careful analysis are required to deal with the resulting discrepancy; see, e.g.~\citep{Tho73, LinThoWah91}. As discussed in~\citep[Section~2]{SchThoWah98}, one approach is to take the mesh domain slightly larger than $\mathscr{D}$ and restrict the computed solution back to $\mathscr{D}$. For natural boundary conditions, the associated error is comparable to that of a suitable quadrature rule. A detailed treatment of this issue is beyond the scope of this paper.}

The paper is organised as follows. In the next section, we introduce the notation, assumptions, and auxiliary results used throughout. Section~\ref{sec:exist LLBE} establishes the existence of a weak solution and a unique strong solution to the LLBE. In Section~\ref{sec:fem LLB}, we propose a linear finite element discretisation of the LLBE and prove a suboptimal convergence rate. To improve this, we consider the $\epsilon$-LLBE in Section~\ref{sec:reg tec}. Assuming attainable regularity of the exact solution, we develop a finite element method for the regularised problem in Section~\ref{sec:fem} and establish a uniform-in-time error estimate in Section~\ref{sec:longtime}. Section~\ref{sec:num sim} presents numerical simulations that support the theoretical results. Additional auxiliary results and technical estimates are provided in Appendices~\ref{sec:aux} and~\ref{sec:further regular proof}.


\section{Preliminaries}

We first introduce some notations in the following section. Further assumptions and auxiliary results used in this paper are collected in Section~\ref{sec:assum} and Appendix~\ref{sec:aux}, respectively.

\subsection{Notations}\label{sec:notation}

Notations used throughout this paper are clarified in this section. For brevity, we denote by~$\pa_j$ the spatial partial derivative~$\pa/\pa
x_j$ for~$j=1,\ldots,d$, and by $\partial_t$ the temporal partial derivative $\partial /\partial t$. We also denote by $\Delta$ the Neumann Laplacian operator on a bounded domain $\mathscr{D}\subset \bb{R}^d$, where $d=1,2,3$.

For $p\in [1,\infty]$ and $k=1,2,\ldots$, the spaces $\bb{L}^p(\mathscr{D})$ and $\bb{W}^{k,p}(\mathscr{D})$ denote the usual Lebesgue and Sobolev spaces of $\bb{R}^3$-valued functions on $\mathscr{D}$. In particular, we use the convention $\bb{W}^{0,p}(\mathscr{D})\equiv \bb{L}^p(\mathscr{D})$ and $\bb{H}^k(\mathscr{D}):= \bb{W}^{k,2}(\mathscr{D})$. The norm in a Banach space $X$ is denoted by $\norm{\cdot}{X}$. The inner product of $\bb{L}^2$ vector-valued functions and the inner product of $\bb{L}^2$ matrix-valued functions will not be distinguished, and are denoted by $\inpro{\cdot}{\cdot}_{\bb{L}^2}$. Also, for $k\in \{1,2\}$, we define the space
\begin{align}\label{equ:Hk Delta}
	\bb{H}^k_{\Delta}(\mathscr{D}) := \big\{\bff{v}\in \bb{H}^k(\mathscr{D}): \Delta \bff{v}\in \bb{H}^{k-1}(\mathscr{D}) \big\},
\end{align}
which is a Banach space for the obvious norm 
\begin{equation}\label{equ:norm Hk Delta}
	\norm{\bff{v}}{\bb{H}^{k}_{\Delta}(\mathscr{D})} := \norm{\bff{v}}{\bb{H}^{k}(\mathscr{D})} + \norm{\Delta \bff{v}}{\bb{H}^{k-1}(\mathscr{D})}.
\end{equation}
Note that in general $\bb{H}^{k+1}(\mathscr{D}) \subset \bb{H}^k_\Delta(\mathscr{D})$, but $\bb{H}^k_\Delta(\mathscr{D}) \not\subset \bb{H}^{k+1}(\mathscr{D})$.

We shall need to use fractional Sobolev spaces in the analysis. For any $s=m+\sigma$, where $m$ is a non-negative integer and $\sigma\in (0,1)$, we define the {Sobolev--Slobodeckij} norm
\begin{equation*}
	\norm{\bff{v}}{\bb{H}^s(\mathscr{D})}
	:=
	\left( \norm{\bff{v}}{\bb{H}^m}^2 + \sum_{|\alpha|=m} \iint_{\mathscr{D}\times \mathscr{D}} \frac{|\partial^\alpha \bff{v}(\bff{x})- \partial^\alpha \bff{v}(\bff{y})|^2}{|\bff{x}-\bff{y}|^{d+2\sigma}}\,\dx\, \dy \right)^{\frac12},
\end{equation*}
and let $\bb{H}^s(\mathscr{D})$ be the completion of $C^\infty(\overline{\mathscr{D}})$ with respect to this norm, see for instance~\citep{DiPalVal12, Gri11, SteTra21} for a more detailed discussion. We denote by $\widetilde{\bb{H}}^{-s}(\mathscr{D})$ the dual space of $\bb{H}^s(\mathscr{D})$ with respect to the dual pairing which is the extension of the {$\bb{L}^2(\mathscr{D})$} inner product.

If $\bb{X}$ is a Banach space, the spaces $L^p(0,T;\bb{X})$ and $W^{k,p}(0,T;\bb{X})$ denote
respectively the Lebesgue and Sobolev spaces of functions on $(0,T)$ taking
values in $\bb{X}$. The space $C([0,T];\bb{X})$ denotes the space of continuous functions
on $[0,T]$ taking values in $\bb{X}$.  For~$1\le p,q\le \infty$ and~$k\geq 0$,
we write~$\bb{L}^p$, $\bb{H}^k$, $\widetilde{\bb{H}}^{-k}$, $L^q_T(\bb{L}^p)$, $L^q_T(\bb{H}^k)$, and $W^{k,q}_T(\bb{L}^p)$ for~$\bb{L}^p(\mathscr{D})$, $\bb{H}^k(\mathscr{D})$, $\widetilde{\bb{H}}^{-k}(\mathscr{D})$,
$L^q(0,T;\bb{L}^p(\mathscr{D}))$, $L^q(0,T;\bb{H}^k(\mathscr{D}))$, and~$W^{k,q}(0,T;\bb{L}^p(\mathscr{D}))$, respectively. 

Since the functions involved are vector-valued functions, we need to clarify
the meaning of the dot and cross products of different mathematical objects. For any vector-valued function~$\bff{u}=(u_1,u_2,u_3):\mathscr{D}\to\bb{R}^3$, we define
\begin{equation*}\label{equ:dot cro}
	\nabla\bff{u}
	:=
	\begin{pmatrix}
		\pa_1 \bff{u} & \cdots & \pa_d \bff{u}
	\end{pmatrix}
	=
	\begin{pmatrix}
		\pa_1 u_1 & \cdots & \pa_d u_1
		\\
		\pa_1 u_2 & \cdots & \pa_d u_2
		\\
		\pa_1 u_3 & \cdots & \pa_d u_3
	\end{pmatrix}
	\quad\text{and}\quad
	\Delta\bff{u}
	:=
	\begin{pmatrix}
		\Delta u_1 \\ \Delta u_2 \\ \Delta u_3
	\end{pmatrix}.
\end{equation*}
For any vector~$\bff{z}\in\bb{R}^3$ and matrices~$\bff{A}, \ \bff{B}\in\bb{R}^{3\times d}$, we
define
\begin{equation*}\label{equ:dot cro 2}
	\begin{alignedat}{2}
		\bff{z}\cdot\bff{A} &:=
		\begin{bmatrix}
			\bff{z}\cdot\bff{A}^{(1)} & \cdots & \bff{z}\cdot\bff{A}^{(d)}
		\end{bmatrix}
		\in\bb{R}^{1\times d},
		\quad&
		\bff{A} \cdot \bff{B}
		&:=
		\sum_{j=1}^d \bff{A}^{(j)} \cdot \bff{B}^{(j)}
		\in\bb{R},
		\\
		\bff{z}\times\bff{A}
		&:=
		\begin{bmatrix}
			\bff{z}\times\bff{A}^{(1)} & \cdots & \bff{z}\times\bff{A}^{(d)}
		\end{bmatrix}
		\in\bb{R}^{3\times d},
		\qquad&
		\bff{A} \times \bff{B}
		&:=
		\sum_{j=1}^d \bff{A}^{(j)} \times \bff{B}^{(j)}
		\in\bb{R}^3,
	\end{alignedat}
\end{equation*}
where~$\bff{A}^{(j)}$ and~$\bff{B}^{(j)}$ denote the $j^{\rm th}$- column
of~$\bff{A}$ and~$\bff{B}$, respectively.
Hence, if~$\bff{u}$, $\bff{v}$, and~$\bff{w}$ are
vector-valued functions defined on~$\mathscr{D}$ and taking values in~$\bb{R}^3$, then
\begin{equation}\label{equ:nab nab}
	\begin{aligned}
		&2\bff{u}\cdot\nabla\bff{u}
		=
		\nabla\left(|\bff{u}|^2\right)
		=
		\begin{pmatrix}
			\pa_1(|\bff{u}|^2) & \cdots & \pa_d(|\bff{u}|^2)
		\end{pmatrix},
		\\
		&\nabla(\bff{u}\times\bff{v})\cdot\nabla\bff{w}
		=
		\big(\nabla\bff{u}\times\bff{v}\big)\cdot\nabla\bff{w}
		+
		\big(\bff{u}\times\nabla\bff{v}\big)\cdot\nabla\bff{w},
		\\
		&(\bff{u}\times\nabla\bff{v})\cdot\nabla\bff{w}
		=
		\sum_{i=1}^d
		\big(
		\bff{u}\times\pa_i\bff{v}
		\big)
		\cdot
		\pa_i\bff{w},
		\\
		&\inpro{\nabla\bff{u}}{\nabla\bff{v}}_{\bb{L}^2}
		=
		\sum_{j=1}^d 
		\int_\mathscr{D} \pa_j\bff{u} \cdot \pa_j\bff{v} \,\dx
		=
		\sum_{i=1}^3
		\int_\mathscr{D} \nabla u_i \cdot \nabla v_i \,\dx,
		\\
		&\inpro{\bff{u}\times\nabla\bff{v}}{\nabla\bff{w}}_{\bb{L}^2}
		=
		\int_\mathscr{D} (\bff{u}\times\nabla\bff{v})\cdot\nabla\bff{w} \,\dx
		=
		\sum_{j=1}^d
		\int_\mathscr{D} \big(\bff{u}\times\pa_j\bff{v}\big)\cdot\pa_j\bff{w} \,\dx.
	\end{aligned}
\end{equation}

Finally, throughout this paper, the constant $c$ in an estimate denotes a
generic constant which may take different values at different occurrences. If
the dependence of $c$ on some variable, e.g.~$T$, needs to be highlighted, we will write
$c_T$ or $c(T)$. The notation $A \lesssim B$ means $A \le c B$, where the specific form of
the constant $c$ is not important to clarify.


\subsection{Assumptions}\label{sec:assum}

Throughout this paper, we assume that the bounded polytopal domain $\mathscr{D}$ is either: 
	\begin{enumerate}
		\item a two-dimensional polygonal domain (which may be \emph{non-convex}), or
		\item a three-dimensional \emph{convex} polyhedral domain.
	\end{enumerate}
	In case $\mathscr{D}$ is non-convex, for simplicity of presentation, we assume that $\mathscr{D}$ has exactly one re-entrant corner with angle $\omega\in (\pi,2\pi)$, and define the elliptic regularity index
	\begin{align}\label{equ:gamma 0}
		\alpha_0 := \pi/\omega.
	\end{align}
	Observe that $\alpha_0 \in (\frac12,1)$. In particular, since $\alpha_0> \frac12$, we exclude domains with cracks from our discussion. Let $\alpha$ be an arbitrary constant satisfying
    \begin{equation}\label{equ:condition alpha}
		\begin{aligned}
			& 1/2 <\alpha<\alpha_0,
			&&\,\text{if $\mathscr{D}$ is a two-dimensional non-convex polygonal domain},
			\\
			& \alpha=1,
			&&\,\text{if $\mathscr{D}$ is a convex domain}.
		\end{aligned}
	\end{equation}


\subsection{Definitions of solutions}

The notion of solutions to the $\epsilon$-LLBE and the LLBE needs to be clarified. To motivate these definitions, first multiplying~\eqref{equ:LLB eps} by~$\bff{v}$, then using integration by parts (specifically~\eqref{equ:vec Gre 1} and~\eqref{equ:vec Gre 2}), we obtain formally,
\begin{align*}
	\nonumber
	\inpro{\partial_t \bff{u}^\epsilon(t)}{\bff{v}}_{\bb{L}^2}
	&+
	\epsilon \inpro{\nabla \partial_t \bff{u}^\epsilon(t)}{\nabla\bff{v}}_{\bb{L}^2}
	+
	\kappa_1 \inpro{\nabla\bff{u}^\epsilon(t)}{\nabla\bff{v}}_{\bb{L}^2}
	\nn\\
	&=
	-
	\gamma
	\inpro{\bff{u}^\epsilon(t)\times\nabla\bff{u}^\epsilon(t)}{\nabla\bff{v}}_{\bb{L}^2}
	-
	\kappa_2
	\inpro{(1+\mu|\bff{u}^\epsilon(t)|^2)\bff{u}^\epsilon(t)}{\bff{v}}_{\bb{L}^2}.
\end{align*}
Note that the terms on the right-hand side are well-defined
if~$\bff{u}^\epsilon\in H^1_T(\bb{H}^1)\cap L^2_T(\bb{W}^{1,4})$
and~$\bff{v}\in\bb{H}^1$. We also note that the embedding $\bb{H}^{1+\alpha}\hookrightarrow \bb{W}^{1,4}$ holds for $\alpha$ satisfying~\eqref{equ:condition alpha}. Hence, we define the solution of \eqref{equ:LLB pro}
and of~\eqref{equ:LLB eps pro} as follows. 
\begin{definition}[solution of the $\epsilon$-LLBE]\label{def:wea sol}
	\normalfont
	Let~$\bff{u}_0^\epsilon\in\bb{H}^1_\Delta$, and $\alpha_0$ be as defined in \eqref{equ:gamma 0}.
	\begin{enumerate}
		\item 
		  For any $\alpha$ satisfying~\eqref{equ:condition alpha}, a function~$\bff{u}^\epsilon\in H^1_T(\bb{H}^1)\cap L^2_T(\bb{H}^{1+\alpha})$
		is a \emph{weak solution} of the $\epsilon$-LLBE~\eqref{equ:LLB eps pro} if, for any $t\in (0,T)$,
		\begin{align}\label{equ:weak sol}
			&\inpro{\bff{u}^\epsilon (t)}{\bff{v}}_{\bb{L}^2}
			+
			\epsilon 
			\inpro{\nabla \bff{u}^\epsilon(t)}{\nabla \bff{v}}_{\bb{L}^2}
			+
			\kappa_1 
			\int_0^t
			\inpro{\nabla\bff{u}^\epsilon(s)}{\nabla\bff{v}}_{\bb{L}^2} \ds
			\nn \\
			&+
			\gamma
			\int_0^t
			\inpro{\bff{u}^\epsilon(s)\times\nabla\bff{u}^\epsilon(s)}{\nabla\bff{v}}_{\bb{L}^2} \ds
			+
			\kappa_2
			\int_0^t
			\inpro{(1+\mu|\bff{u}^\epsilon(s)|^2)\bff{u}^\epsilon(s)}{\bff{v}}_{\bb{L}^2} \ds
			\nn\\
			&= 
            \inpro{\bff{u}^\epsilon_0}{\bff{v}}_{\bb{L}^2}
			+
			\epsilon 
			\inpro{\nabla \bff{u}^\epsilon_0}{\nabla \bff{v}}_{\bb{L}^2},
			\quad \forall \bff{v}\in \bb{H}^1(\mathscr{D}).
		\end{align}
		\item A weak solution $\bff{u}^\epsilon$ is said to be a \emph{strong solution}
		if $\bff{u}^\epsilon\in W^{1,\infty}_T(\bb{H}^1) \cap H^1_T(\bb{H}^1_\Delta)$. In this case, \eqref{equ:LLB eps pro} is satisfied
		for all~$t\in(0,T)$ and almost all $x\in \mathscr{D}$, and we can write
		\begin{align}\label{equ:u eps L2}
			\bff{u}^\epsilon(t)
            -
            \epsilon \Delta \partial_t\bff{u}^\epsilon(t) 
			&= 
			\bff{u}^\epsilon_0
			-
            \epsilon \Delta \bff{u}^\epsilon_0
			+
			\kappa_1\int_0^t\Delta\bff{u}^\epsilon(s)\, \ds
			+
			\gamma\int_0^t\bff{u}^\epsilon(s)\times  \Delta\bff{u}^\epsilon(s) \,\ds
			\nn\\
			&\quad-
			\kappa_2\int_0^t(1+\mu|\bff{u}^\epsilon(s)|^2)\bff{u}^\epsilon(s) \,\ds
			\quad\text{ in }\,\bb{L}^2(\mathscr{D}).
		\end{align}
	\end{enumerate}
	A weak (resp. strong) solution is said to be a \emph{global} weak (resp. strong) solution if it exists for any given positive time $T$. It is said to be a \emph{local} weak (resp. strong) solution if it only exists up to some particular time $T$.
\end{definition}

\begin{definition}[solution of the LLBE]\label{def:wea sol LLB}
	Let $\bff{u}_0\in \bb{H}^1$.
	\begin{enumerate}
		\item For any $\alpha$ satisfying \eqref{equ:condition alpha}, a function $\bff{u}\in \bb{W}^{1,4/3}_T(\bb{L}^2) \cap L^\infty_T(\bb{H}^1) \cap L^2_T(\bb{H}^{1+\alpha})$ is a \emph{weak solution} of the LLBE~\eqref{equ:LLB pro} if $\bff{u}(0)=\bff{u}_0$ and, for any $t\in (0,T)$, \eqref{equ:weak sol} holds with $\epsilon=0$. 
		\item A weak solution $\bff{u}$ is said to be a \emph{strong solution} if $\bff{u}\in \bb{W}^{1,\infty}_T(\bb{L}^2) \cap L^\infty_T(\bb{H}^1_\Delta)$. In this case, \eqref{equ:LLB pro} is satisfied
		for all~$t\in(0,T)$ and almost all $x\in \mathscr{D}$, and $\bff{u}$ satisfies \eqref{equ:u eps L2} with $\epsilon=0$.
	\end{enumerate}
\end{definition}

It follows from embedding theorems, see e.g.~\citep[Lemma 1.2]{Lio69}, that a
weak solution $\bff{u}^\epsilon$ of the $\epsilon$-LLBE belongs to $C([0,T];\bb{H}^1)$, while a weak solution $\bff{u}$ of the LLBE {only belongs} to $C([0,T]; \bb{L}^2)$.

\section{Existence and uniqueness of the solution to the LLBE}\label{sec:exist LLBE}

Throughout this section, let $T>0$ be a given number which can be arbitrarily large.
The Faedo--Galerkin method will be used to show the existence of a solution to the LLBE. 
This solution satisfies the following weak form: For every~$t\in[0,T]$,
	\begin{subequations}\label{equ:LLB wea exa}
		\begin{alignat}{1}
			\inpro{\pa_t\bff{u}(t)}{\bff{v}}_{\bb{L}^2}
			+
			&\kappa_1 \inpro{\nabla\bff{u}(t)}{\nabla\bff{v}}_{\bb{L}^2}
			+
			\kappa_2
			\inpro{\bff{u}(t)}{\bff{v}}_{\bb{L}^2}
			\nn\\
			&=
			\gamma
			\inpro{\bff{u}(t)\times\nabla\bff{u}(t)}{\nabla \bff{v}}_{\bb{L}^2}
			-
			\kappa_2 \mu
			\inpro{|\bff{u}(t)|^2\bff{u}(t)}{\bff{v}}_{\bb{L}^2},
			\quad \forall\bff{v}\in\bb{H}^1,
			\label{equ:LLB wea equ exa}
			\\
			\bff{u}(0) &= \bff{u}_{0}
			\label{equ:LLB wea con exa}
		\end{alignat}
	\end{subequations}
Let $\{(\bff{e}_i,\lambda_i)\}_{i=1}^{\infty}$ be a sequence of eigenpairs of the negative Neumann Laplace operator, i.e., for all $i=1,2,\ldots$,
\[
-\Delta\bff{e}_i = \lambda_i\bff{e}_i \text{ in } \mathscr{D}
\quad\text{and}\quad
\frac{\pa\bff{e}_i}{\pa\bff{\nu}} = 0 \text{ on }\pa \mathscr{D}.
\]
We can choose the eigenfunctions~$\bff{e}_i$ such that $\{\bff{e}_i\}_{i=1}^{\infty}$ forms an orthonormal
basis of~$\bb{L}^2$.


Let $\mathcal{S}_n:=\text{span}\{\bff{e}_1,\cdots,\bff{e}_n\}$ and let
$\Pi_n$ be the~$\bb{L}^2$-projection onto~$\mathcal{S}_n$. We consider the following
approximation to~\eqref{equ:LLB pro}: Find~$\bff{u}_n(\cdot,t)\in\mathcal{S}_n$
satisfying
\begin{equation}\label{equ:Gal LLB wea}
\begin{alignedat}{2}
 \pa_t\bff{u}_n
&=
\kappa_1\Delta\bff{u}_n
+
\gamma\Pi_n\bigl(\bff{u}_n\times\Delta\bff{u}_n\bigr)
-
\kappa_2\Pi_n\bigl((1+\mu|\bff{u}_n|^2)\bff{u}_n\bigr)
&\quad &\text{in } (0,T) \times \mathscr{D},
\\
\bff{u}_n(0) &= \bff{u}_{0,n}
		&\quad &\text{in } \mathscr{D}.
\end{alignedat}
\end{equation}
We assume that $\bff{u}_0$ in \eqref{equ:u ini con} and~$\bff{u}_{0,n}\in\mathcal{S}_n$ in \eqref{equ:Gal LLB wea} are chosen such that the initial data approximation satisfies
\begin{equation}\label{equ:fur ass H1}
		\lim_{n\to\infty}
		\norm{\bff{u}_{0,n}-\bff{u}_0}{\bb{H}^1} = 0.
	\end{equation}
Equivalently for every~$t\in [0,T]$,
	\begin{subequations}\label{equ:LLB wea}
		\begin{alignat}{1}
			\inpro{\pa_t\bff{u}_n(t)}{\bff{v}}_{\bb{L}^2}
			+
			&\kappa_1 \inpro{\nabla\bff{u}_n(t)}{\nabla\bff{v}}_{\bb{L}^2}
			+
			\kappa_2
			\inpro{\bff{u}_n(t)}{\bff{v}}_{\bb{L}^2}
			\nn\\
			&=
			\gamma
			\inpro{\bff{u}_n(t)\times \Delta\bff{u}_n(t)}{\bff{v}}_{\bb{L}^2}
			-
			\kappa_2 \mu
			\inpro{|\bff{u}_n(t)|^2\bff{u}_n(t)}{\bff{v}}_{\bb{L}^2},
			\quad \forall\bff{v}\in\mathcal{S}_n,
			\label{equ:LLB wea equ}
			\\
			\bff{u}_n(0) &= \bff{u}_{0,n}.
			\label{equ:LLB wea con}
		\end{alignat}
	\end{subequations}

The following lemma gives local existence of a solution to~\eqref{equ:Gal LLB wea}.

\begin{lemma}\label{lem:loc exi}
For each $n\in\bb{N}$, there exists a sufficiently small time $\tau_n>0$ such that the problem~\eqref{equ:Gal LLB wea}
admits a solution on $[0,\tau_n]$.
\end{lemma}
\begin{proof}
For $n\in\bb{N}$, let~$F_n^i:\mathcal{S}_n\to\mathcal{S}_n$, $i=1,2,3$, be defined by
\begin{align*}
F_n^1(\bff{v}) = \Delta\bff{v}, \quad
F_n^2(\bff{v}) = \Pi_n(\bff{v}\times\Delta\bff{v}), \quad
F_n^3(\bff{v}) = \Pi_n((1+\mu|\bff{v}|^2)\bff{v}),
\end{align*}
for any~$\bff{v}\in\mathcal{S}_n$. It is shown in~\citep[Lemma 3.1]{Le16} that~$F_n^1$ is 
globally Lipschitz, while~$F_n^2$ and~$F_n^3$ are locally Lipschitz. The local
existence of~$\bff{u}_n$ follows at once.
\end{proof}

We now show some bounds on the local solution~$\bff{u}_n$
of~\eqref{equ:Gal LLB wea}, which will imply its global existence, i.e., existence of solution on $[0,T]$ for any given $T>0$.

\begin{lemma}\label{lem:ene est 1 llb}
	Assume that~\eqref{equ:fur ass H1} holds. For any $n\in \bb{N}$ and $\alpha$ satisfying~\eqref{equ:condition alpha},
\begin{equation}\label{equ:un global}
\norm{\bff{u}_n(t)}{\bb{H}^1}^2
+
\int_0^t \norm{\Delta\bff{u}_n(s)}{\bb{L}^2}^2 \ds
+
\int_0^t 
\norm{\bff{u}_n(s)}{\bb{L}^4}^4
+
\int_0^t \norm{\bff{u}_n(s)}{\bb{H}^{1+\alpha}}^2 \ds
\le c,
\end{equation}
where the constant $c$ depends on $\norm{\bff{u}_0}{\bb{H}^1}$, but is independent of~$n$ and~$t$.
\end{lemma}

\begin{proof}
The inequality for the first three terms on the left-hand side of \eqref{equ:un global} is shown in~\citep[Lemma~3.2]{Le16}. For the last term, by the elliptic regularity estimate~\eqref{equ:shift polyg H1s}, we obtain
\begin{align*}
    \int_0^t \norm{\bff{u}_n(s)}{\bb{H}^{1+\alpha}}^2 \ds
	\lesssim
	\int_0^t \left(\norm{\bff{u}_n(s)}{\widetilde{\bb{H}}^{-1+\alpha}}^2 + \norm{\Delta \bff{u}_n(s)}{\widetilde{\bb{H}}^{-1+\alpha}}^2 \right) \ds
	\lesssim 1,
\end{align*}
where in the last step we also used the embedding $\bb{L}^2\hookrightarrow \widetilde{\bb{H}}^{-1+\alpha}$ and inequality~\eqref{equ:un global} for the first two terms. This completes the proof.
\end{proof}

\begin{lemma}
Assume that~\eqref{equ:fur ass H1} holds.
Then for any $n\in \bb{N}$, the approximate solution $\bff{u}_n$ satisfies
\begin{align}\label{equ:dt un 43}
	\int_0^t \norm{\bff{u}_n(s)}{\bb{L}^\infty}^4 \ds
	+
	\int_0^t \norm{\pa_t \bff{u}_n(s)}{\bb{L}^2}^{4/3} \ds \lesssim 1,
\end{align}
where the constant is independent of $n$.
\end{lemma}

\begin{proof}
We first derive an estimate for the first term on the left-hand side of \eqref{equ:dt un 43}. If $d=1$, the inequality for the first term on the left-hand side of~\eqref{equ:dt un 43} follows by the embedding $\bb{H}^1 \hookrightarrow \bb{L}^\infty$ and \eqref{equ:un global}.
For the case $d=2$, by successively applying the Sobolev embedding $\bb{H}^{1+ \alpha/4} \hookrightarrow \bb{L}^\infty$, interpolation inequalities for fractional Sobolev spaces~\citep{BreMir18}, and equation~\eqref{equ:un global}, we obtain
\begin{align*}
	\int_0^t \norm{\bff{u}_n(s)}{\bb{L}^\infty}^4 \ds
	\lesssim
	\int_0^t \norm{\bff{u}_n(s)}{\bb{H}^{1+\alpha/4}}^4 \ds
	&\lesssim
	\int_0^t \norm{\bff{u}_n(s)}{\bb{H}^1}^2 \norm{\bff{u}_n(s)}{\bb{H}^{1+\alpha/2}}^2 \ds
	\\
	&\lesssim
	\int_0^t \norm{\bff{u}_n(s)}{\bb{H}^{1+\alpha/2}}^2 \ds
	\lesssim 1,
\end{align*}
where in the last step we also used the embedding $\bb{L}^2\hookrightarrow \widetilde{\bb{H}}^{-1+\alpha/2}$.
For the case $d=3$, noting that the domain is convex by assumption (and so \eqref{equ:shift polyh W2p} holds), by using Agmon's inequality and \eqref{equ:un global}, we have
\begin{align*}
	\int_0^t \norm{\bff{u}_n(s)}{\bb{L}^\infty}^4 \ds
	\lesssim
	\int_0^t \norm{\bff{u}_n(s)}{\bb{H}^1}^2 \norm{\bff{u}_n(s)}{\bb{H}^2}^2 \ds
	\lesssim
	\int_0^t \norm{\bff{u}_n(s)}{\bb{H}^2}^2 \ds
	\lesssim 1.
\end{align*}

Next, taking the $\bb{L}^2$ norm of \eqref{equ:Gal LLB wea}, then applying H\"older's and Young's inequalities, we obtain
\begin{align*}
	\norm{\pa_t \bff{u}_n}{\bb{L}^2}^{4/3}
	&\leq
	\kappa_1 \norm{\Delta \bff{u}_n}{\bb{L}^2}^{4/3}
	+
	\gamma \norm{\bff{u}_n}{\bb{L}^\infty}^{4/3} \norm{\Delta \bff{u}_n}{\bb{L}^2}^{4/3}
	+
	\kappa_2 \norm{\bff{u}_n}{\bb{L}^2}^{4/3}
	+
	\kappa_2 \mu \norm{\bff{u}_n}{\bb{L}^6}^4
	\\
	&\leqs
	1+ \norm{\Delta \bff{u}_n}{\bb{L}^2}^2 + \norm{\bff{u}_n}{\bb{L}^\infty}^4
	+ \norm{\bff{u}_n}{\bb{L}^2}^2 + \norm{\bff{u}_n}{\bb{H}^1}^4
	\\
	&\leqs
	1+ \norm{\Delta \bff{u}_n}{\bb{L}^2}^2 + \norm{\bff{u}_n}{\bb{L}^\infty}^4
\end{align*}
where in the penultimate step we used the embedding $\bb{H}^1\hookrightarrow \bb{L}^6$, and in the last step we used \eqref{equ:un global}. Therefore, integrating over $(0,t)$ and using \eqref{equ:un global} again,
\begin{align*}
	\int_0^t \norm{\pa_t \bff{u}_n(s)}{\bb{L}^2}^{4/3} \ds
	&\leqs
	1+ \int_0^t \norm{\bff{u}_n(s)}{\bb{L}^\infty}^4 \ds.
\end{align*}
By the result established previously, the lemma is proven.
\end{proof}

The following lemmas show a uniform bound which hold possibly only locally in time.

\begin{lemma}\label{lem:H2 un}
	Let $T$ be a given positive real number. Assume that 
    \begin{equation}\label{equ:fur ass Hs}
		\lim_{n\to\infty}
		\norm{\bff{u}_{0,n}-\bff{u}_0}{\bb{H}^s_\Delta} = 0.
	\end{equation}
   holds for $s=1$. Then there exists $T^\ast$ satisfying
\begin{equation}\label{equ:T ast}
		\begin{cases}
			T^\ast = T, \quad & \text{if } d=1,2,
			\\
			T^\ast \le T, \quad & \text{if } d=3,
		\end{cases}
	\end{equation}
    such that for any $n\in \bb{N}$ and~$t\in[0,T^\ast]$,
	\begin{align}\label{equ:Delta un llb}
		\norm{\Delta \bff{u}_n(t)}{\bb{L}^2}^2
		+
		\int_0^t \norm{\nabla \Delta \bff{u}_n(s)}{\bb{L}^2}^2 \ds 
		&\leq c,
	\end{align}
	where $c$ is independent of $n$.
\end{lemma}

\begin{proof}
To prove~\eqref{equ:Delta un llb}, we
set~$\bff{v}=2\Delta^2\bff{u}_n(t)\in\mathcal{S}_n$
in~\eqref{equ:LLB wea equ} and obtain
\begin{align}\label{equ:ddt Lap une}
	&\ddt\norm{\Delta\bff{u}_n(t)}{\bb{L}^2}^2
	+
	2\kappa_1\norm{\nabla\Delta\bff{u}_n(t)}{\bb{L}^2}^2
	+
	2\kappa_2 \norm{\Delta \bff{u}_n(t)}{\bb{L}^2}^2
	\nn\\
	&=
	-2\gamma
	\inpro{\nabla\Big(\bff{u}_n(t)\times\Delta\bff{u}_n(t)\Big)}
	{\nabla\Delta\bff{u}_n(t)}_{\bb{L}^2}
	+2\kappa_2\mu
	\inpro{\nabla\Big(|\bff{u}_n(t)|^2\bff{u}_n(t)\Big)}
	{\nabla\Delta\bff{u}_n(t)}_{\bb{L}^2}
	\nn\\
	&=
	2\gamma
	\inpro{\Delta\bff{u}_n(t)\times\nabla\bff{u}_n(t)}
	{\nabla\Delta\bff{u}_n(t)}_{\bb{L}^2}
	+2\kappa_2\mu
	\inpro{\nabla\Big(|\bff{u}_n(t)|^2\bff{u}_n(t)\Big)}
	{\nabla\Delta\bff{u}_n(t)}_{\bb{L}^2}
	\nn\\
	&\le
	2\gamma
	\left|
	\inpro{\Delta\bff{u}_n(t)\times\nabla\bff{u}_n(t)}
	{\nabla\Delta\bff{u}_n(t)}_{\bb{L}^2}
	\right|
	+ 2\kappa_2\mu
	\left|
	\inpro{\nabla\Big(|\bff{u}_n(t)|^2\bff{u}_n(t)\Big)}
	{\nabla\Delta\bff{u}_n(t)}_{\bb{L}^2}
	\right|
	\nn\\
	&=
	R_1 + R_2,
\end{align}
where
\begin{align*}
	R_1 &:=
	2\gamma
	\left|
	\inpro{\Delta\bff{u}_n(t)\times\nabla\bff{u}_n(t)}
	{\nabla\Delta\bff{u}_n(t)}_{\bb{L}^2}
	\right|
	\nn\\
	R_2 &:=
	2\kappa_2\mu
	\left|
	\inpro{\nabla\Big(|\bff{u}_n(t)|^2\bff{u}_n(t)\Big)}
	{\nabla\Delta\bff{u}_n(t)}_{\bb{L}^2}
	\right|.
\end{align*}
Invoking Lemma~\ref{lem:tec lem} with~$\bff{v}=\Delta\bff{u}_n$,
$\bff{w}=\bff{u}_n$, and~$\delta=\kappa_1/3d$, we deduce
\[
R_1
\le
\frac{\kappa_1}{3}
\norm{\nabla\Delta\bff{u}_n(t)}{\bb{L}^2}^2
+
\Phi(\bff{u}_n(t))
\norm{\Delta\bff{u}_n(t)}{\bb{L}^2}^2,
\]
where $\Phi$ is defined in~\eqref{equ:Phi}.
On the other hand, we have
\begin{align*}
	\nn
	R_2 
	&\lesssim
	\sum_{i=1}^d
	\left|
	\inpro{2\Big(\pa_i\bff{u}_n(t)\cdot\bff{u}_n(t)\Big)\bff{u}_n(t)}
	{\pa_i\Delta\bff{u}_n(t)}_{\bb{L}^2}
	\right|
	+
	\sum_{i=1}^d
	\left|
	\inpro{|\bff{u}_n(t)|^2\pa_i\bff{u}_n(t)}
	{\pa_i\Delta\bff{u}_n(t)}_{\bb{L}^2}
	\right|
	\\
	\nn
	&\lesssim
	\norm{\bff{u}_n(t)}{\bb{L}^\infty}^2
	\sum_{i=1}^d
	\norm{\pa_i\bff{u}_n(t)}{\bb{L}^2}
	\norm{\pa_i\Delta\bff{u}_n(t)}{\bb{L}^2}
	\\
	\nn
	&\le
	\norm{\bff{u}_n(t)}{\bb{L}^\infty}^2
	\norm{\nabla\bff{u}_n(t)}{\bb{L}^2}
	\norm{\nabla\Delta\bff{u}_n(t)}{\bb{L}^2}
	\\
	&\leq
	\frac{\kappa_1}{3} \norm{\nabla\Delta\bff{u}_n(t)}{\bb{L}^2}^2
	+
	c \norm{\bff{u}_n(t)}{\bb{L}^\infty}^4,
\end{align*}
where in the last step we used Young's inequality and \eqref{equ:un global}.
Altogether, from~\eqref{equ:ddt Lap une} we deduce 
\begin{align*}
	\ddt{}\norm{\Delta\bff{u}_n(t)}{\bb{L}^2}^2
	+
	\norm{\nabla\Delta\bff{u}_n(t)}{\bb{L}^2}^2
	&\lesssim
	\norm{\bff{u}_n(t)}{\bb{L}^\infty}^4
	+
	\Phi(\bff{u}_n(t))
	\norm{\Delta\bff{u}_n(t)}{\bb{L}^2}^2.
\end{align*}
Integrating with respect to
the time variable we deduce, by using~\eqref{equ:fur ass Hs} with $s=1$ and \eqref{equ:dt un 43},
\begin{align}\label{equ:ineq Delta une}
	&\norm{\Delta\bff{u}_n(t)}{\bb{L}^2}^2
	+
	\int_0^t
	\norm{\nabla\Delta\bff{u}_n(s)}{\bb{L}^2}^2\ds
	\nn\\
	&\leq
	\norm{\Delta \bff{u}_0}{\bb{L}^2}^2
	+
	c\int_0^t \norm{\bff{u}_n(s)}{\bb{L}^\infty}^4 \ds
	+
	c\int_0^t
	\Phi(\bff{u}_n(s))
	\norm{\Delta\bff{u}_n(s)}{\bb{L}^2}^2
	\ds
    \nn\\
    &\leq
    c + c\int_0^t
	\Phi(\bff{u}_n(s))
	\norm{\Delta\bff{u}_n(s)}{\bb{L}^2}^2.
\end{align}
It remains to estimate the last term on the right-hand side of \eqref{equ:ineq Delta une}. If $d=1$, it is clear from Lemma~\ref{lem:tec lem} and \eqref{equ:un global} that $\Phi(\bff{u}_n(t)) \lesssim 1$.
	For the case $d=2$, note that by the Sobolev embedding $\bb{H}^{3/2}\hookrightarrow \bb{W}^{1,4}$, inequality \eqref{equ:shift polyg H1s} for $s=1/2$, and the Sobolev interpolation inequality, we have
	\begin{align*}
		\norm{\bff{u}_n(t)}{\bb{W}^{1,4}}^4
		\leqs
		\norm{\bff{u}_n(t)}{\bb{H}^{3/2}}^4
		&\leqs
		\norm{\bff{u}_n(t)}{\widetilde{\bb{H}}^{-1/2}}^4
		+
		\norm{\Delta \bff{u}_n(t)}{\widetilde{\bb{H}}^{-1/2}}^4
		\\
		&\leqs
		\norm{\bff{u}_n(t)}{\bb{L}^2}^4
		+
		\norm{\Delta \bff{u}_n(t)}{\bb{L}^2}^2
		\norm{\Delta \bff{u}_n(t)}{\widetilde{\bb{H}}^{-1}}^2
		\\
		&\leqs
		\norm{\bff{u}_n(t)}{\bb{L}^2}^4
		+
		\norm{\Delta \bff{u}_n(t)}{\bb{L}^2}^2
		\left( \sup_{\bff{0}\neq \bff{\phi}\in \bb{H}^1} \frac{\abs{\inpro{\Delta \bff{u}_n(t)}{\bff{\phi}}_{\bb{L}^2}}}{\norm{\bff{\phi}}{\bb{H}^1}} \right)^2
		\\
		&\leqs
		\norm{\bff{u}_n(t)}{\bb{L}^2}^4
		+
		\norm{\Delta \bff{u}_n(t)}{\bb{L}^2}^2
		\norm{\nabla \bff{u}_n(t)}{\bb{L}^2}^2
		\\
		&\leqs
		1+ \norm{\Delta \bff{u}_n(t)}{\bb{L}^2}^2,
	\end{align*}
	where in the penultimate step we also used integration by parts (noting the homogeneous Neumann boundary condition for $\bff{u}_n$), and in the last step we again used \eqref{equ:un global}.
	Thus, we infer the following estimates for
	$d=1,2$ and for all~$t\in[0,T]$:
	\begin{equation}\label{equ:Phi un less}
	\Phi(\bff{u}_n(t))
	\lesssim
	\begin{cases}
		1, \quad & d=1,
		\\
		1 + \norm{\Delta\bff{u}_n(t)}{\bb{L}^2}^2,
		\quad & d=2.
	\end{cases}
	\end{equation}
	By using
	\eqref{equ:un global}, we infer
	\begin{equation}\label{equ:int Phi}
		\int_0^t
		\Phi(\bff{u}_n(s)) \,\ds \lesssim 1,
	\end{equation}
	where the constant is independent of $n$.
	By using Gronwall's inequality and~\eqref{equ:int Phi}, we deduce~\eqref{equ:Delta un llb} for~$d=1,2$.
	
	Finally, consider the case~$d=3$. By the Sobolev embedding $\bb{H}^2\hookrightarrow \bb{W}^{1,6}$ and the $\bb{H}^2$ elliptic regularity, we have
	\begin{align*} 
		\Phi(\bff{u}_n(s))
		\norm{\Delta\bff{u}_n(s)}{\bb{L}^2}^2
		&\lesssim
		\left(1+ \norm{\bff{u}_n(s)}{\bb{H}^2}^4\right)
		\norm{\Delta \bff{u}_n(s)}{\bb{L}^2}^2
		\lesssim
		1 + \norm{\Delta\bff{u}_n(s)}{\bb{L}^2}^6,
	\end{align*}
	where we again used \eqref{equ:un global}. 
This, together with~\eqref{equ:ineq Delta une}, implies
\[
\norm{\Delta\bff{u}_n(t)}{\bb{L}^2}^2
+
\int_0^t
\norm{\nabla\Delta\bff{u}_n(s)}{\bb{L}^2}^2\ds
\lesssim
1 +
\int_{0}^{t} \norm{\Delta\bff{u}_n(s)}{\bb{L}^2}^6 \ds.
\]
We infer \eqref{equ:Delta un llb} by using the Bihari--Gronwall
inequality~\citep{Bih56}, completing the proof of the lemma.
\end{proof}

\begin{lemma}\label{lem:H2 dtu un}
	Assume that \eqref{equ:fur ass Hs} holds for $s=1$.
	Then for any $n\in \bb{N}$, $\alpha$ satisfying~\eqref{equ:condition alpha}, and~$t\in[0,T^\ast]$, we have
	\begin{align}\label{equ:un H1a llb}
		\norm{\bff{u}_n(t)}{\bb{H}^{1+\alpha}}^2
		+
		\norm{\pa_t \bff{u}_n(t)}{\bb{L}^2}^2
		+
		\int_0^t \norm{\pa_t \bff{u}_n(s)}{\bb{H}^1}^2 \ds 
		&\leq c,
	\end{align}
	where $c$ is independent of $n$, and~$T^\ast$ is defined by \eqref{equ:T ast}.
\end{lemma}

\begin{proof}
Firstly, by~\eqref{equ:shift polyg H1s} in the case $d=2$, and~\eqref{equ:shift polyh W2p} in the case $d=3$, we have for any $\alpha$ satisfying~\eqref{equ:condition alpha},
\begin{align}\label{equ:unt H1a}
	\norm{\bff{u}_n(t)}{\widetilde{\bb{H}}^{1+\alpha}}^2
	&\lesssim 
	\norm{\bff{u}_n(t)}{\widetilde{\bb{H}}^{-1+\alpha}}^2
	+
	\norm{\Delta \bff{u}_n(t)}{\widetilde{\bb{H}}^{-1+\alpha}}^2
	\nn\\
	&\lesssim
	\norm{\bff{u}_n(t)}{\bb{L}^2}^2
	+
	\norm{\Delta \bff{u}_n(t)}{\bb{L}^2}^2
	\leqs 1.
\end{align}

Next, taking the $\bb{L}^2$ norm of \eqref{equ:Gal LLB wea}, then applying H\"older's and Young's inequalities, we obtain
\begin{align*}
	\norm{\pa_t \bff{u}_n}{\bb{L}^2}^2
	&\leq
	\kappa_1 \norm{\Delta \bff{u}_n}{\bb{L}^2}^2
	+
	\gamma \norm{\bff{u}_n}{\bb{L}^\infty}^2 \norm{\Delta \bff{u}_n}{\bb{L}^2}^2
	+
	\kappa_2 \norm{\bff{u}_n}{\bb{L}^2}^2
	+
	\kappa_2 \mu \norm{\bff{u}_n}{\bb{L}^6}^6
	\\
	&\leqs
	\norm{\Delta \bff{u}_n}{\bb{L}^2}^2 + \norm{\bff{u}_n}{\bb{H}^{1+\alpha}}^2 \norm{\Delta \bff{u}_n}{\bb{L}^2}^2
	+ \norm{\bff{u}_n}{\bb{L}^2}^2 + \norm{\bff{u}_n}{\bb{H}^1}^6
	\leqs
	1,
\end{align*}
where in the last step we used \eqref{equ:un global}, Lemma~\ref{lem:H2 un}, and \eqref{equ:unt H1a}.

Similarly, we have
\begin{align*}
	\norm{\nabla \pa_t \bff{u}_n}{\bb{L}^2}^2
	&\leq
	\kappa_1 \norm{\nabla \Delta \bff{u}_n}{\bb{L}^2}^2
	+
	\gamma \norm{\bff{u}_n}{\bb{L}^\infty}^2 \norm{\nabla\Delta \bff{u}_n}{\bb{L}^2}^2
	+
	\gamma \norm{\nabla \bff{u}_n}{\bb{L}^4}^2 \norm{\Delta \bff{u}_n}{\bb{L}^4}^2
	+
	\kappa_2 \norm{\nabla\bff{u}_n}{\bb{L}^2}^2
	\\
	&\quad
	+
	3\kappa_2 \mu \norm{\bff{u}_n}{\bb{L}^\infty}^4
	\norm{\nabla \bff{u}_n}{\bb{L}^2}^2
	\\
	&\leqs
	\norm{\nabla \Delta \bff{u}_n}{\bb{L}^2}^2 + \norm{\bff{u}_n}{\bb{H}^{1+\alpha}}^2 \left( \norm{\nabla \Delta \bff{u}_n}{\bb{L}^2}^2+ \norm{\Delta \bff{u}_n}{\bb{L}^2}^2\right)
	+ \left(1+ \norm{\bff{u}_n}{\bb{H}^{1+\alpha}}^4 \right)  \norm{\bff{u}_n}{\bb{H}^1}^2
	\\
	&\leqs
	1+ \norm{\nabla\Delta \bff{u}_n}{\bb{L}^2}^2,
\end{align*}
where in the penultimate step we used the embeddings $\bb{H}^{1+\alpha} \hookrightarrow \bb{W}^{1,4} \hookrightarrow \bb{L}^\infty$ and $\bb{H}^1\hookrightarrow \bb{L}^4$, while in the last step we used \eqref{equ:unt H1a} and Lemma~\ref{lem:H2 un}. Integrating both sides of the inequality over $(0,t)$ and using Lemma~\ref{lem:H2 un} again then completes the proof.
\end{proof}

Altogether, the relevant results established in this section can be summarised as follows. Suppose $\alpha$ satisfies~\eqref{equ:condition alpha}.
	\begin{itemize}
		\item If~\eqref{equ:fur ass H1} holds, then we have:
		\begin{equation}\label{equ:un H1 summ}
				\norm{\bff{u}_n}{L^\infty_T(\bb{H}^1)}
				+
				\norm{\bff{u}_n}{L^2_T(\bb{H}^{1+\alpha})}
				+
				\norm{\bff{u}_n}{L^2_T(\bb{H}^1_\Delta)}
				+
				\norm{\pa_t \bff{u}_n}{L^{4/3}_T(\bb{L}^2)}
				\lesssim 1.
		\end{equation}
		\item If~\eqref{equ:fur ass Hs} holds for $s=1$, then additionally we have the estimate:
		\begin{equation}\label{equ:un H2 summary}
			\norm{\bff{u}_n}{L^\infty_{T^\ast}(\bb{H}^{1+\alpha})}
			+
			\norm{\bff{u}_n}{L^\infty_{T^\ast}(\bb{H}^1_\Delta)}
			+
			\norm{\Delta \bff{u}_n}{L^2_{T^\ast}(\bb{H}^1)}
			+
			\norm{\pa_t\bff{u}_n}{L^\infty_{T^\ast}(\bb{L}^2)}
			+
			\norm{\pa_t\bff{u}_n}{L^2_{T^\ast}(\bb{H}^1)}
			\lesssim 1,
		\end{equation}
		where $T^\ast$ is defined in \eqref{equ:T ast}.
		\item Furthermore, it follows from the estimates established in Appendix~\ref{sec:further regular proof} (summarised in~\eqref{equ:une est local}) that if~\eqref{equ:fur ass Hs} holds for $s=2$, then additionally we have
		\begin{equation}\label{equ:un H3 summary}
			\norm{\Delta \bff{u}_n}{L^\infty_{T^\ast}(\bb{L}^2)}
			+
			\norm{\Delta \bff{u}_n}{L^2_{T^\ast}(\bb{H}^1)}
			+
			\norm{\pa_t\bff{u}_n}{L^\infty_{T^\ast}(\bb{H}^1)}
			+
			\norm{\pa_t\bff{u}_n}{L^2_{T^\ast}(\bb{H}^1_\Delta)}
			+
			\norm{\pa_t\bff{u}_n}{L^2_{T^\ast}(\bb{H}^{1+\alpha})}
			\leqs 1.
		\end{equation}
	\end{itemize}

We are now ready to prove the following result on existence and regularity of the solution.



\begin{theorem}\label{the:con u eps}
Let $T>0$ be fixed and let $T^\ast$ be defined by \eqref{equ:T ast}.
	\begin{enumerate}
		\renewcommand{\labelenumi}{\theenumi}
		\renewcommand{\theenumi}{{\rm (\roman{enumi})}}
		\item 
		If~\eqref{equ:fur ass H1} holds, then there exists
		a global weak solution $\bff{u}$ to~\eqref{equ:LLB pro} in the sense of Definition~\ref{def:wea sol LLB}, satisfying
        \begin{equation}\label{equ:u H1 summ}
				\norm{\bff{u}}{L^\infty_T(\bb{H}^1)}
				+
				\norm{\bff{u}}{L^2_T(\bb{H}^{1+\alpha})}
				+
				\norm{\bff{u}}{L^2_T(\bb{H}^1_\Delta)}
				+
				\norm{\pa_t \bff{u}}{L^{4/3}_T(\bb{L}^2)}
				\leq c.
		\end{equation}
        Moreover, $\bff{u}\in C^{0,\widetilde{\gamma}}([0,T]; \bb{L}^2)$ for every $\widetilde{\gamma}\in [0,\frac14)$.
		\item 
		If~\eqref{equ:fur ass Hs} holds for $s=1$, then there exists a strong solution to~\eqref{equ:LLB pro} in the sense of Definition~\ref{def:wea sol LLB}, which is global if $d=1,2$ and local if $d=3$, satisfying 
        \begin{equation}\label{equ:u H2 summary}
			\norm{\bff{u}}{L^\infty_{T^\ast}(\bb{H}^{1+\alpha})}
			+
			\norm{\bff{u}}{L^\infty_{T^\ast}(\bb{H}^1_\Delta)}
			+
			\norm{\Delta \bff{u}}{L^2_{T^\ast}(\bb{H}^1)}
			+
			\norm{\pa_t\bff{u}}{L^\infty_{T^\ast}(\bb{L}^2)}
			+
			\norm{\pa_t\bff{u}}{L^2_{T^\ast}(\bb{H}^1)}
			\leq c.
		\end{equation}
        Furthermore, $\bff{u}\in C^{0,\widetilde{\gamma}}([0,T^\ast]; \bb{H}^1)$ for every $\widetilde{\gamma}\in [0,\frac12)$.
        	\item 
        If~\eqref{equ:fur ass Hs} holds for $s=2$, then there exists a strong solution to~\eqref{equ:LLB pro}, which is global if $d=1,2$ and local if $d=3$, satisfying 
        \begin{equation}\label{equ:u H3 summary}
        	\norm{\Delta \bff{u}}{L^\infty_{T^\ast}(\bb{L}^2)}
        	+
        	\norm{\Delta \bff{u}}{L^2_{T^\ast}(\bb{H}^1)}
        	+
        	\norm{\pa_t\bff{u}}{L^\infty_{T^\ast}(\bb{H}^1)}
        	+
        	\norm{\pa_t\bff{u}}{L^2_{T^\ast}(\bb{H}^1_\Delta)}
        	+
        	\norm{\pa_t\bff{u}}{L^2_{T^\ast}(\bb{H}^{1+\alpha})}
        	\leqs 1.
        \end{equation}
        Furthermore, $\bff{u}\in C^{0,\widetilde{\gamma}}([0,T^\ast]; \bb{H}^{1+\alpha})$ for every $\widetilde{\gamma}\in [0,\frac12)$.
	\end{enumerate}
\end{theorem}

\begin{proof}
Suppose that $\bff{u}_0\in\bb{H}^1$. It follows from~\eqref{equ:un H1 summ} that there exist
$\bff{u} \in W^{1, 4/3}_{T}(\bb{L}^2) \cap L^\infty_{T}(\bb{H}^1) \cap L^2_T(\bb{H}^{1+\alpha})$ and a subsequence
of~$\{\bff{u}_n\}_{n\in\bb{N}}$, which we do not relabel, satisfying
\begin{equation}\label{equ:wea con llb}
	\begin{aligned}
		& \bff{u}_n\to\bff{u}
		&&\,\text{weakly}^\ast \text{ in $W^{1, 4/3}_{T}(\bb{L}^2)$ as $n\to \infty$},
		\\
		& \bff{u}_n\to\bff{u}
		&&\,\text{weakly}^\ast \text{ in $L^\infty_{T}(\bb{H}^1)$ as $n\to \infty$},
		\\
		& \bff{u}_n\to \bff{u}
		&&\,\text{weakly} \text{ in $L^2_{T}(\bb{H}^{1+\alpha})$ as $n\to \infty$.}
	\end{aligned}
\end{equation}
Since the embeddings $\bb{H}^{1+\alpha}\hookrightarrow \bb{H}^1 \hookrightarrow \bb{L}^4$ are compact, it follows from the
Aubin--Lions lemma and~\eqref{equ:une wea con} that
\begin{equation}\label{equ:L2 L4 llb}
	\lim_{n\to\infty} \norm{\bff{u}_n-\bff{u}}{L^2(\bb{L}^4)} 
	= 0
	\quad\text{and}\quad 
	\lim_{n\to\infty} \norm{\bff{u}_n-\bff{u}}{L^2(\bb{H}^1)} 
	= 0.
\end{equation}

We now prove that~$\bff{u}$ satisfies~\eqref{equ:weak sol} with $\epsilon=0$.
	It follows from~\eqref{equ:LLB wea equ} that for any test function
	$\bff{v}\in \bb{H}^1$ and for every $t\in[0,T]$
	\begin{align}\label{equ:weakLLB int}
		&\int_0^t \inpro{\pa_t\bff{u}_n(s)}{\tilde\Pi_n \bff{v}}_{\bb{L}^2} \ds
		+\kappa_1\int_0^t \inpro{\nabla\bff{u}_n(s)}{\nabla\tilde\Pi_n \bff{v}}_{\bb{L}^2}\ds 
		\nn\\
		&
		+\gamma\int_0^t \inpro{\bff{u}_n(s)\times\nabla\bff{u}_n(s)}{\nabla \tilde\Pi_n \bff{v}}_{\bb{L}^2}\ds
		+\kappa_2\int_0^t \inpro{(1+\mu|\bff{u}_n(s)|^2)\bff{u}_n(s)}{\tilde\Pi_n \bff{v}}_{\bb{L}^2}\ds
		= 0,
	\end{align}
	where~$\tilde\Pi_n:\bb{H}^1\to \mathcal{S}_n$ is the orthogonal projection with respect to the~$\bb{H}^1$-inner product.
	We recall that, for all~$\bff{\xi}\in\bb{H}^1$,
	\[
	\lim_{n\to\infty}\norm{\tilde\Pi_n\bff{\xi}-\bff{\xi}}{\bb{H}^1} = 0.
	\]
	Thus, it follows from~\eqref{equ:wea con llb} 
	that, when~$n\to\infty$,
	\begin{align*}
		&\int_0^t \inpro{\pa_t\bff{u}_n(s)}{\tilde\Pi_n \bff{v}}_{\bb{L}^2} \ds
		\to
		\int_0^t \inpro{\pa_t\bff{u}(s)}{\bff{v}}_{\bb{L}^2} \ds,
		\\
		&\int_0^t \inpro{\nabla\bff{u}_n(s)}{\nabla\tilde\Pi_n \bff{v}}_{\bb{L}^2}\ds 
		\to
		\int_0^t \inpro{\nabla\bff{u}(s)}{\nabla\bff{v}}_{\bb{L}^2}\ds.
	\end{align*}
	For the first nonlinear term in~\eqref{equ:weakLLB int} we have, by using the
	triangle inequality and~\eqref{equ:un H1 summ},
	\begin{align*}
		&
		\left|
		\int_0^t \inpro{\bff{u}_n(s)\times\nabla\bff{u}_n(s)}{\nabla\tilde\Pi_n \bff{v}}_{\bb{L}^2}\ds
		-
		\int_0^t \inpro{\bff{u}(s)\times\nabla\bff{u}(s)}{\nabla\bff{v}}_{\bb{L}^2}\ds
		\right|
		\\
		&\quad\le
		\left|
		\int_0^t
		\inpro{\bff{u}_n(s)\times\nabla\bff{u}_n(s)}{\nabla(\tilde\Pi_n
			\bff{v}-\bff{v})}_{\bb{L}^2}\ds \right|
		+
		\left| \int_0^t \inpro{(\bff{u}_n(s)-\bff{u}(s))\times \nabla
			\bff{u}_n(s)}{\nabla\bff{v}}_{\bb{L}^2} \ds \right| 
		\\
		&\qquad
		+
		\left| \int_0^t \inpro{\bff{u}(s) \times (\nabla \bff{u}_n(s)-
			\nabla \bff{u}(s))}{\nabla\bff{v}}_{\bb{L}^2} \ds \right| 
		\\
		&\quad\le
		\norm{\bff{u}_n}{L^2(\bb{L}^\infty)}
		\norm{\nabla\bff{u}_n}{L^2(\bb{L}^2)}
		\norm{\nabla(\tilde\Pi_n\bff{v}-\bff{v})}{\bb{L}^2}
		+
		\norm{\bff{u}_n-\bff{u}}{L^2(\bb{L}^4)}
		\norm{\nabla\bff{u}_n}{L^2(\bb{L}^4)}
		\norm{\nabla\bff{v}}{\bb{L}^2}
		\\
		&\qquad
		+
		\norm{\bff{u}}{L^2(\bb{L}^\infty)}
		\norm{\nabla \bff{u}_n- \nabla \bff{u}}{L^2(\bb{L}^2)}
		\norm{\nabla\bff{v}}{\bb{L}^2}
		\\
		&\quad\lesssim
		\norm{\nabla(\tilde\Pi_n\bff{v}-\bff{v})}{\bb{L}^2}
		+
		\norm{\bff{u}_n-\bff{u}}{L^2(\bb{L}^4)}
		+
		\norm{\nabla \bff{u}_n- \nabla \bff{u}}{L^2(\bb{L}^2)}.
	\end{align*}
	The right-hand side converges to zero due to the definition 
	of~$\tilde\Pi$ and \eqref{equ:L2 L4 llb}, implying the convergence of this first
	nonlinear term. Similar arguments can be used for the second nonlinear term
	in~\eqref{equ:weakLLB int}.  Therefore, by taking the limit when~$n\to\infty$ in
	equation~\eqref{equ:weakLLB int}, we obtain that
	$\bff{u}$ is a global weak solution. Taking the limit $n\to\infty$ in \eqref{equ:un H1 summ} yields \eqref{equ:u H1 summ}. This function $\bff{u}$ admits a H\"older continous representative such that $\bff{u}\in C^{0,\widetilde{\gamma}}([0,T^\ast];\bb{L}^2)$ for every $\widetilde{\gamma}\in [0,\frac14)$ by a version of vector-valued Morrey's inequality~\citep{AreKre18}.
	
	Next, suppose $\bff{u}_0\in \bb{H}^1_\Delta$. We prove that~$\bff{u}$ satisfies \eqref{equ:u eps L2} with $\epsilon=0$.
	Since~$\bff{u}$ is a weak solution, we deduce from~\eqref{equ:weak sol} with $\epsilon=0$
	and integration by parts that, for all $t\in [0,T^\ast]$,
	\begin{align*}
		\bff{u}(t)
		&= 
		\bff{u}_0
		+
		\kappa_1\int_0^t\Delta\bff{u}(s) \,\ds
		+
		\gamma\int_0^t\bff{u}(s)\times  \Delta\bff{u}(s) \,\ds
		-\kappa_2\int_0^t(1+\mu|\bff{u}(s)|^2)\bff{u}(s) \,\ds
		\qquad\text{ in }\,\widetilde{\bb{H}}^{-1}.
	\end{align*}
	We will show that the above equation holds in~$\bb{L}^2$.
	Indeed, inequality~\eqref{equ:u H1 summ} implies that all the linear terms belong
	to~$\bb{L}^2$. For the nonlinear terms, by using Minkowski's inequality, the
	continuous embeddings $\bb{H}^{3/2}\hookrightarrow \bb{L}^\infty$ and
	$\bb{H}^1\hookrightarrow \bb{L}^6$, and~\eqref{equ:wea con llb}, we deduce that
	for any $t\in [0,T^\ast]$
	\begin{align*}
		&\norm{\int_0^t\bff{u}(s)\times \Delta\bff{u}(s) \,\ds}{\bb{L}^2}
		+
		\norm{\int_0^t(1+\mu|\bff{u}(s)|^2)\bff{u}(s)\, \ds}{\bb{L}^2} 
		\\
		&\quad \leq 
		\int_0^t \norm{\bff{u}(s)\times  \Delta\bff{u}(s)}{\bb{L}^2} \,\ds
		+
		\int_0^t \norm{(1+\mu|\bff{u}(s)|^2)\bff{u}(s)}{\bb{L}^2} \,\ds\\
		&\quad \leq 
		\norm{\bff{u}}{L^2_{T^\ast}(\bb{L}^\infty)} \norm{\Delta \bff{u}}{L^2_{T^\ast}(\bb{L}^2)}
		+
		\norm{\bff{u}}{L^1_{T^\ast}(\bb{L}^2)} + \mu \norm{\bff{u}}{L^3_{T^\ast}(\bb{L}^6)}^3
		\\
		&\quad \lesssim
		\norm{\bff{u}}{L^2_{T^\ast}(\bb{H}^{3/2})} \norm{\Delta \bff{u}}{L^2_{T^\ast}(\bb{L}^2)} 
		+
		\norm{\bff{u}}{L^1_{T^\ast}(\bb{L}^2)} 
		+
		\norm{\bff{u}}{L^\infty_{T^\ast}(\bb{H}^1)}^3
		< \infty.
	\end{align*}
	This implies that the nonlinear terms also belong to~$\bb{L}^2$ and
	therefore~$\bff{u}$ satisfies~\eqref{equ:u eps L2} with $\epsilon=0$. Estimate~\eqref{equ:u H2 summary} follows by taking the limit $n\to\infty$ in \eqref{equ:un H2 summary}, thus showing that $\bff{u}$ is a strong solution. This function $\bff{u}$ admits a H\"older continous representative such that $\bff{u}\in C^{0,\widetilde{\gamma}}([0,T^\ast];\bb{H}^1)$ for every $\widetilde{\gamma}\in [0,\frac12)$ by a version of vector-valued Morrey's inequality~\citep{AreKre18}. 
	
	Finally, if $\bff{u}_0\in \bb{H}^2_\Delta$, then similar argument applies to~\eqref{equ:un H3 summary}, which yields~\eqref{equ:u H3 summary}. In this case, $\bff{u}$ admits a H\"older continous representative such that $\bff{u}\in C^{0,\widetilde{\gamma}}([0,T^\ast];\bb{H}^{1+\alpha})$ for every $\widetilde{\gamma}\in [0,\frac12)$.
	This completes the proof of the proposition.
\end{proof}

Next, we show a continuous dependence result for the solution of the LLBE with respect to the initial data, which implies uniqueness.

\begin{theorem}\label{the:unique llb}
	Let~$\bff{u}_0, \bff{v}_0 \in \bb{H}^1_\Delta$.
	Let~$\bff{u}$ and~$\bff{v}$ be the strong solutions
	associated with the initial data~$\bff{u}_0$
	and~$\bff{v}_0$, respectively, as conferred by
	Theorem~\ref{the:con u eps}. Then
	\[
	\norm{\bff{u}-\bff{v}}{L^\infty_{T^\ast}(\bb{L}^2)} 
	\lesssim
	\norm{\bff{u}_0-\bff{v}_0}{\bb{L}^2}, 
	\]
	where the constant may depend on~$T^\ast$. This implies
	uniqueness of the solution~$\bff{u}$.
\end{theorem}

\begin{proof}
	Let~$\bff{w}_0:=\bff{u}_0-\bff{v}_0$ 
	and~$\bff{w}:=\bff{u}-\bff{v}$.
	Then it follows from~\eqref{equ:LLB wea equ} that, for any $\bff{\phi}\in
	\bb{H}^1$ and $t\in [0,T^\ast]$,
	\begin{align*}
		&\inpro{\partial_t \bff{w}(t)}{\bff{\phi}}_{\bb{L}^2}
		+
		\kappa_1 \inpro{\nabla\bff{w}(t)}{\nabla\bff{\phi}}_{\bb{L}^2}
		+
		\gamma \inpro{\bff{w}(t)\times\nabla\bff{u}(t)}{\nabla\bff{\phi}}_{\bb{L}^2}
		+
		\gamma \inpro{\bff{v}(t)\times\nabla\bff{w}(t)}{\nabla\bff{\phi}}_{\bb{L}^2}
		\\
		&\quad
		+
		\kappa_2 \inpro{\bff{w}(t)}{\bff{\phi}}_{\bb{L}^2}
		+
		\kappa_2\mu \inpro{|\bff{u}(t)|^2\bff{w}(t) + (|\bff{u}(t)|^2-|\bff{v}(t)|^2)\bff{v}(t)}{\bff{\phi}}_{\bb{L}^2}
		= 
		0.
	\end{align*}
	Putting $\bff{\phi} = 2\bff{w}$ in the above 
	and
	integrating we obtain, for any $t\in [0,T^\ast]$,
	\begin{align*}
		&\norm{\bff{w}(t)}{\bb{L}^2}^2
		+
		2\kappa_1\int_0^t \|\nabla\bff{w}(s)\|^2_{\bb{L}^2}\,\ds
		+
		2\kappa_2\int_0^t \|\bff{w}(s)\|^2_{\bb{L}^2}\,\ds
		+
		2\kappa_2\mu\int_0^t\||\bff{u}(s)|
		|\bff{w}(s)|\|^2_{\bb{L}^2}\,\ds
		\\
		&=
		\norm{\bff{w}_0}{\bb{L}^2}^2
		-
		2 \gamma
		\int_0^t 
		\inpro{\bff{w}(s)\times\nabla\bff{u}(s)}
		{\nabla\bff{w}(s)}_{\bb{L}^2}\,\ds
		-
		2 \kappa_2\mu
		\int_0^t 
		\inpro{(|\bff{u}(s)|^2-|\bff{v}(s)|^2)\bff{v}(s)}
		{\bff{w}(s)}_{\bb{L}^2} \ds.
	\end{align*}
	The last two terms on the right-hand side can be estimated as
	follows. For the middle term, by Lemma~\ref{lem:tec lem} we have
	\begin{align*}
		\left| \int_0^t \inpro{\bff{w}(s)\times\nabla\bff{u}(s)}{\nabla\bff{w}(s)}_{\bb{L}^2}\ds \right| 
		&\leq 
		\int_0^t \Phi(\bff{u}(s)) \norm{\bff{w}(s)}{\bb{L}^2}^2 \ds
		+
		\delta \int_0^t \norm{\nabla \bff{w}(s)}{\bb{L}^2}^2 \ds 
	\end{align*}
	for any $\delta>0$, where $\Phi$ is defined in~\eqref{equ:Phi}.
	For the last term, we have
	\begin{align}\label{equ:ue ve w llb}
		&\left| \int_0^t \inpro{(|\bff{u}(s)|^2-|\bff{v}(s)|^2)\bff{v}(s)}{\bff{w}(s)}_{\bb{L}^2} \ds \right| 
		\nn\\
		&\leq
		\int_0^t \norm{\bff{u}(s)+\bff{v}(s)}{\bb{L}^6}
		\norm{\bff{w}(s)}{\bb{L}^6}
		\norm{\bff{v}(s)}{\bb{L}^6}
		\norm{\bff{w}(s)}{\bb{L}^2} \ds 
		\nn \\
		&\leq
		\int_0^t \norm{\bff{u}(s)+\bff{v}(s)}{\bb{H}^1}^2
		\norm{\bff{v}(s)}{\bb{H}^1}^2
		\norm{\bff{w}(s)}{\bb{L}^2} \ds 
		+
		\delta \int_0^t \norm{\bff{w}(s)}{\bb{H}^1}^2 \ds 
		\nn \\
		&\leq
		c \int_0^t \norm{\bff{w}(s)}{\bb{L}^2}^2 \ds 
		+
		\delta \int_0^t \norm{\bff{w}(s)}{\bb{H}^1}^2 \ds 
	\end{align}
	for any $\delta>0$, where in the penultimate step we used Young's
	inequality and the Sobolev embedding $\bb{H}^1\hookrightarrow \bb{L}^6$, and
	in the last step we used~\eqref{equ:u H2 summary}. 
	
	Altogether, rearranging the terms and choosing $\delta>0$ sufficiently small, we infer that
	\begin{align*}
		\norm{\bff{w}(t)}{\bb{L}^2}^2
		\leq
		\norm{\bff{w}_0}{\bb{L}^2}^2
		+
		\int_0^t \Phi(\bff{u}(s)) \norm{\bff{w}(s)}{\bb{L}^2}^2 \ds.
	\end{align*}
	Note that~$\bff{u}\in L^\infty_{T^\ast}(\bb{H}^{1+\alpha})$ for all $\alpha$ satisfying~\eqref{equ:condition alpha} by \eqref{equ:u H2 summary}. Thus, by the Sobolev embeddings $\bb{H}^{1+\alpha}\hookrightarrow \bb{W}^{1,2d}$ for $d=1,2$ and $\bb{H}^2\hookrightarrow \bb{W}^{1,6}$ for $d=3$, we have
		\begin{equation}\label{equ:Phi c llb}
			\int_0^t \Phi(\bff{u}(s))\,\ds \le c.
		\end{equation}
	Applying Gronwall's inequality and
	noting~\eqref{equ:Phi c llb}, we obtain the required result.
\end{proof}

As mentioned in the introduction, a remarkable property of the macroscopic magnetisation above the Curie
temperature is that it will spontaneously relax towards the zero state in the
long-run. In other words, the solution $\bff{u}(t)$ to the LLBE
should converge to $\bff{0}$ as $t\to \infty$. We conclude this section with a rigorous proof of this physical fact below.

\begin{theorem}\label{the:ut inf}
Let~$\bff{u}$ be a strong solution to~\eqref{equ:LLB pro} and let $p\in
[2,\infty]$. For all $t\in [0,\infty)$,
\begin{align}\label{equ:est u Lp}
	\norm{\bff{u}(t)}{\bb{L}^p} \leq e^{-\kappa_2 t} \norm{\bff{u}_0}{\bb{L}^p}.
\end{align}
\end{theorem}

\begin{proof}
First we consider~$p\in[2,\infty)$. Putting~$\bff{v}=|\bff{u}|^{p-2}
\bff{u}$ in~\eqref{equ:LLB wea equ exa} and rearranging, we obtain
\begin{align*}
	\frac{1}{p} \ddt \norm{\bff{u}}{\bb{L}^p}^p
	+
	\kappa_1 
	\inpro{\nabla \bff{u}}{\nabla (\abs{\bff{u}}^{p-2} \bff{u})}_{\bb{L}^2}
	+
	\kappa_2 \norm{\bff{u}}{\bb{L}^p}^p
	+
	\kappa_2\mu \norm{\bff{u}}{\bb{L}^{p+2}}^{p+2}
	=
	0,
\end{align*}
Note that, for $q\geq 0$,
\begin{align*}
	\nabla (\abs{\bff{u}}^q \bff{u})
	&=
	\abs{\bff{u}}^q \nabla \bff{u}
	+
	q \abs{\bff{u}}^{q-2} \bff{u} (\bff{u}\cdot \nabla \bff{u}).
\end{align*}
Therefore,
\begin{align*}
	\frac{1}{p} \ddt \norm{\bff{u}}{\bb{L}^p}^p
	+
	\kappa_1 \norm{|\bff{u}|^{\frac{p-2}{2}} |\nabla \bff{u}|}{\bb{L}^2}^2
	+
	\kappa_1 (p-2) \norm{|\bff{u}|^{\frac{p-4}{2}} |\bff{u}\cdot\nabla\bff{u}|}{\bb{L}^2}^2
	+
	\kappa_2 \norm{\bff{u}}{\bb{L}^p}^p
	+
	\kappa_2\mu \norm{\bff{u}}{\bb{L}^{p+2}}^{p+2}
	=
	0,
\end{align*}
which implies
\[
\ddt \norm{\bff{u}(t)}{\bb{L}^p}^p + p \kappa_2 \norm{\bff{u}(t)}{\bb{L}^p}^p
\leq 0,
\]
so that
\[
\ddt
\Big(
e^{p\kappa_2 t} \norm{\bff{u}(t)}{\bb{L}^p}^p
\Big)
\le 0.
\]
Integrating and taking $p$-th root give
\[
\norm{\bff{u}(t)}{\bb{L}^p}
\leq
e^{-\kappa_2 t} \norm{\bff{u}_0}{\bb{L}^p}.
\]
This proves~\eqref{equ:est u Lp} for~$p\in[2,\infty)$. Letting $p\to \infty$ in
the above completes the proof for the case $p=\infty$.
\end{proof}

\section{Fully-discrete finite element approximation for the LLBE}\label{sec:fem LLB}

In this section, we assume that $\bff{u}_0 \in \bb{H}^2_\Delta$. In particular, this guarantees a strong solution $\bff{u}$ of the LLBE with regularity 
\begin{align}\label{equ:regularity llb}
	\bff{u} \in {W}^{1,\infty}_{T^\ast}(\bb{H}^1)\cap H^1_{T^\ast}(\bb{H}^{1+\alpha}).
\end{align}
exists (cf. Theorem~\ref{the:con u eps}), where $\alpha$ is defined in~\eqref{equ:condition alpha}.

Let $\mathscr{D}\subset \bb{R}^d$ be a domain satisfying the assumptions in Section~\ref{sec:assum}, and let $\{\mathcal{T}_h\}_{h>0}$ be a family of quasi-uniform triangulations of $\mathscr{D}$ with maximal mesh-size $h$. To discretise the LLBE, we introduce the Lagrange finite element space $\bb{V}_h \subset \bb{H}^1$, which is the space of all piecewise linear polynomials on $\mathcal{T}_h$. 
Let $P_h:\bb{L}^2\to \bb{V}_h$ be the {$\bb{L}^2$-orthogonal projection onto $\bb{V}_h$. For any $\bff{v}\in \bb{L}^2$, $P_h\bff{v}$ is defined} by
\begin{equation}\label{equ:orth proj}
	\inpro{\bff{v}-P_h \bff{v}}{\bff{\chi}}_{\bb{L}^2}=0, \quad \forall \bff{\chi}\in \bb{V}_h.
\end{equation}
Under the assumption of quasi-uniformity of the triangulations, we have the following stability~\citep{CroTho87} and approximation properties~\citep[Chapter~22]{ErnGue21} of $P_h$, which we now state. Let $\alpha_0$ be as defined in \eqref{equ:gamma 0}. Then for any $p\in [1,\infty]$ and $\alpha\in (\frac12,\alpha_0)$, there exist constants $c_p$ and $c_\alpha$ such that
\begin{align}
	\label{equ:proj stab}
	\norm{P_h \bff{v}}{\bb{L}^p}
	&\leq
	c_p \norm{\bff{v}}{\bb{L}^p}, \quad {\forall \bff{v}\in \bb{L}^p},
	\\
	\label{equ:proj ineq}
	\norm{\bff{v}-P_h \bff{v}}{\bb{L}^2}
	+
	h \norm{\nabla(\bff{v}- P_h \bff{v})}{\bb{L}^2}
	&\leq
	c_\alpha h^{1+\alpha} \norm{\bff{v}}{\bb{H}^{1+\alpha}}, \quad {\forall \bff{v}\in \bb{H}^{1+\alpha}}.
\end{align}
If $\mathscr{D}$ is a convex polytopal domain, then \eqref{equ:proj ineq} also holds for $\alpha=1$.

We also need a well-known inverse estimate. Let $\ell_h$ be defined by
\begin{equation}\label{equ:ell h}
	\ell_h:=
	\begin{cases}
		1, &\text{if } d=1,
		\\
		\abs{\log h}^{1/2}, &\text{if } d=2,
		\\
		h^{-1/2}, &\text{if }d=3.
	\end{cases}
\end{equation}
If the family of triangulations $\{\mathcal{T}_h\}_{h>0}$ is quasi-uniform, then there exists a constant $c_{\mathrm{i}}$ (dependent on the regularity of the triangulation, but independent of $h$) such that for all $\bff{v}\in \bb{V}_h$,
\begin{equation}\label{equ:inverse Vh}
	\norm{\bff{v}}{\bb{L}^\infty} \leq
	c_{\mathrm{i}} \ell_h \norm{\bff{v}}{\bb{H}^1}.
\end{equation}

We now propose a finite element scheme for the LLBE and 
provide error estimates for the approximate solution. To discretise in time, let $N\in\bb{N}$ and $k=T/N$. We
partition $[0,T]$ into $N$ uniform subintervals with nodes $t_n=kn$ for $n=0,1,2,\ldots,N$. For a time-discrete function $\bff{v}^{(j)}$, define $\mathrm{d}_t \bff{v}^{(j)}:= (\bff{v}^{(j)}-\bff{v}^{(j-1)})/k$. 

\begin{algorithm}[Linear FEM for the LLBE]\label{alg:fem llbe}
{Let $h>0$ and $k>0$ be given.
\\
\textbf{Input}: Given $\bff{u}_h^{(0)}= P_h \bff{u}(0)\in \bb{V}_h$.
\\
\textbf{For} $j=1$ to $N$ \textbf{do}: Find $\bff{u}_h^{(j)}\in\bb{V}_h$ such that 
\begin{align}\label{equ:dis1 llb}
	\inpro{\mathrm{d}_t\bff{u}_h^{(j)}}{\bff{\phi}_h}_{\bb{L}^2}
	&+
	\kappa_1 \inpro{\nabla\bff{u}_h^{(j)}}{\nabla\bff{\phi}_h}_{\bb{L}^2}
	+
	\gamma \inpro{\bff{u}_h^{(j-1)}\times\nabla\bff{u}_h^{(j)}}{\nabla\bff{\phi}_h}_{\bb{L}^2}
	\nn\\
	&\quad
	+
	\kappa_2 \mu
	\inpro{|\bff{u}_h^{(j-1)}|^2 \bff{u}_h^{(j)}}{\bff{\phi}_h}_{\bb{L}^2}
	+
	\kappa_2 \inpro{\bff{u}_h^j}{\bff{\phi}_h}_{\bb{L}^2}
	=0,
	\quad\forall \bff{\phi}_h\in\bb{V}_h.
\end{align}
\textbf{Output}: a sequence of finite element functions $\big\{\bff{u}_h^{(j)}\big\}_{j=1}^N$.}
\end{algorithm}

Algorithm~\ref{alg:fem llbe} is well-posed by the Lax--Milgram theorem. Stability of the finite element solution in $\ell^\infty(0,T;\bb{L}^2)\cap \ell^2(0,T;\bb{H}^1)$,
which we state below, holds unconditionally for an arbitrary number of iterations $n$.


\begin{lemma}\label{lem: stability2}
{Let $\big\{\bff{u}_h^{(j)}\big\}_{j=1}^N$ be the sequence defined by Algorithm~\ref{alg:fem llbe}.}	For any $n\in \{1,2,\ldots,N\}$, there holds
	\begin{align*}
		\|\bff{u}_h^{(n)}\|^2_{\bb{L}^2}
		+
		2k\sum_{j=1}^n\|\nabla\bff{u}_h^{(j)}\|^2_{\bb{L}^2}
		+
		\sum_{j=1}^n\|\bff{u}_h^{(j)}-\bff{u}_h^{(j-1)}\|^2_{\bb{L}^2}
		\leq 
		\|\bff{u}_h^{(0)}\|^2_{\bb{L}^2}.
	\end{align*}
\end{lemma}

\begin{proof}
Let $\bff{\phi}_h = \bff{u}_h^{(j)}$ in equation~\eqref{equ:dis1 llb}. Noting the vector identity
\begin{equation}\label{equ:vec ab a}
	2\bff{a}\cdot (\bff{a}-\bff{b})
	= 
	\abs{\bff{a}}^2-\abs{\bff{b}}^2
	+ \abs{\bff{a}-\bff{b}}^2, \quad \forall \bff{a},\bff{b}\in \bb{R}^3,
\end{equation} 
we obtain
\begin{align}\label{equ:uhj L2}
\frac12 \norm{\bff{u}_h^{(j)}}{\bb{L}^2}^2
&+
\frac12 \norm{\bff{u}_h^{(j)}-\bff{u}_h^{(j-1)}}{\bb{L}^2}^2
+
k\kappa_1 \norm{\nabla\bff{u}_h^{(j)}}{\bb{L}^2}^2
\nn \\
&+
k \kappa_2 \norm{\bff{u}_h^{(j)}}{\bb{L}^2}^2
+
k \kappa_2\mu \norm{\big|\bff{u}_h^{(j-1)}\big| \big|\bff{u}_h^{(j)}\big|}{L^2}^2
=
\frac12 \norm{\bff{u}_h^{(j-1)}}{\bb{L}^2}^2.
\end{align}
Summing over $j=1,2,\ldots,n$ yields
\begin{align*}
\frac12 \norm{\bff{u}_h^{(n)}}{\bb{L}^2}^2
&+
\frac12\sum_{j=1}^n \norm{\bff{u}_h^{(j)}-\bff{u}_h^{(j-1)}}{\bb{L}^2}^2
+
k\kappa_1 \sum_{j=1}^n \norm{\nabla\bff{u}_h^{(j)}}{\bb{L}^2}^2
+
k\kappa_2 \sum_{j=1}^n \norm{\bff{u}_h^{(j)}}{\bb{L}^2}^2
\\
&+
k\kappa_2\mu \sum_{j=1}^n \norm{\big|\bff{u}_h^{(j-1)}\big| \big|\bff{u}_h^{(j)}\big|}{L^2}^2
=
\frac12 \norm{\bff{u}_h^{(0)}}{\bb{L}^2}^2,
\end{align*}
as required.
\end{proof}

The following estimates on certain nonlinear terms will be needed in the proof of the main theorem.

\begin{lemma}
	\label{lem:aux lem}
	Let~$\bff{v}^{(j)}$, $j=1,\ldots,n$, be a (finite) sequence of functions
	in~$\bb{H}^1\cap \bb{L}^\infty$ and~$\bff{w}$ be a function in~$L^\infty_T(\bb{H}^{1+\alpha})\cap
	W^{1,\infty}_T(\bb{H}^1)$. Assume that we can write
	\[
	\bff{v}^{(j)} - \bff{w}(t_j)
	=
	\bff{\eta}^{(j)} + \bff{\mu}^{(j)}.
	\]
	Define, for~$\bff{\phi}\in\bb{H}^1$,
	\begin{align*} 
		H_1(\bff{v}^{(j)},\bff{w},\bff{\phi})
		&:=
		k \inpro{\bff{v}^{(j-1)} \times \nabla\bff{v}^{(j)}}%
		{\nabla\bff{\phi}}_{\bb{L}^2}
		-
		\int_{t_{j-1}}^{t_j} 
		\inpro{\bff{w}(s) \times \nabla\bff{w}(s)}%
		{\nabla\bff{\phi}}_{\bb{L}^2} \ds,
		\\
		H_2(\bff{v}^{(j)},\bff{w},\bff{\phi})
		&:=
		k \inpro{|\bff{v}^{(j-1)}|^2 \bff{w}(t_j)}{\bff{\phi}}_{\bb{L}^2}
		-
		\int_{t_{j-1}}^{t_j} 
		\inpro{|\bff{w}(s)|^2 \bff{w}(s)}{\bff{\phi}}_{\bb{L}^2} \ds.
	\end{align*}
	Then, for any~$ \delta>0$, there exists a positive constant~$c$
	depending on~$\delta$ but independent of~$k$ and~$j$ such that
	\begin{align}\label{equ:H1 vj}
			\big| H_1(\bff{v}^{(j)},\bff{w},\bff{\mu}^{(j)}) \big|
			&\le
			ck\norm{\bff{v}^{(j-1)}}{\bb{L}^\infty}^2
			\norm{\nabla\bff{\eta}^{(j)}}{\bb{L}^2}^2
			+
			ck
			\Big(\norm{\bff{\eta}^{(j-1)}}{\bb{H}^1}^2
			+ \norm{\bff{\mu}^{(j-1)}}{\bb{L}^2}^2 \Big)
			\norm{\bff{w}}{L^\infty_T(\bb{H}^{1+\alpha})}^2
			\nn\\
			&\quad
			+
			c k^3
			\Big(
			\norm{\bff{v}^{(j-1)}}{\bb{L}^\infty}^2
			+
			\norm{\bff{w}}{L^\infty_T(\bb{H}^{1+\alpha})}^2
			\Big)  	\norm{\bff{w}}{W^{1,\infty}_T(\bb{H}^1)}^2
			\nn\\
			&\quad
			+ \delta k \norm{\nabla\bff{\mu}^{(j-1)}}{\bb{L}^2}^2 \norm{\bff{w}}{L^\infty_T(\bb{H}^{1+\alpha})}^2
			+ 4 \delta k \norm{\nabla\bff{\mu}^{(j)}}{\bb{L}^2}^2
	\end{align}
	and
	\begin{align}\label{equ:H2 vj}
			| H_2(\bff{v}^{(j)},\bff{w},\bff{\mu}^{(j)}) |
			&\le
			ck
			\norm{\bff{\eta}^{(j-1)}+\bff{\mu}^{(j-1)}}{\bb{L}^2}^2
			\norm{\bff{v}^{(j-1)}+\bff{w}(t_{j-1})}{\bb{H}^1}^2
			\norm{\bff{w}}{L^\infty_T(\bb{H}^1)}^2
			\nn\\
			&\quad
			+
			ck^3
			\Big(
			\norm{\bff{w}}{W^{1,\infty}_T(\bb{L}^2)}^2
			\norm{\bff{w}}{L^\infty_T(\bb{H}^1)}^4
			+
			\norm{\bff{v}^{(j-1)}}{\bb{H}^1}^2
			\norm{\bff{w}}{W^{1,\infty}_T(\bb{H}^1)}^2
			\Big)
			\nn\\
			&\quad
			+
			2 \delta k \norm{\bff{\mu}^{(j)}}{\bb{H}^1}^2
			+
			\delta k
			\norm{|\bff{v}^{(j-1)}||\bff{\mu}^{(j)}|}{\bb{L}^2}^2,
	\end{align}
	provided that all the norms on the right-hand sides are well defined.

    Furthermore, if $\mathscr{D}$ is \emph{convex} and $\bff{w}\in L^\infty_T(\bb{W}^{2,3})\cap W^{1,\infty}_T(\bb{H}^2)$, then
    \begin{align}\label{equ:H1 new}
	\big| H_1(\bff{v}^{(j)},\bff{w},\bff{\mu}^{(j)}) \big|
	&\le
	ck\norm{\bff{v}^{(j-1)}}{\bb{H}^1}^2
	\norm{\nabla\bff{\eta}^{(j)}}{\bb{L}^3}^2
	+
	ck
	\Big(\norm{\bff{\eta}^{(j-1)}}{\bb{H}^1}^2
	+ \norm{\bff{\mu}^{(j-1)}}{\bb{L}^2}^2 \Big)
	\norm{\bff{w}}{L^\infty_T(\bb{H}^2)}^2
	\nn\\
	&\quad
	+
	c k^3
	\Big(
	\norm{\bff{v}^{(j-1)}}{\bb{H}^1}^2
	+
	\norm{\bff{w}}{L^\infty_T(\bb{H}^2)}^2
	\Big)  	\norm{\bff{w}}{W^{1,\infty}_T(\bb{H}^2)}^2
	\nn\\
	&\quad
	+ \delta k \norm{\nabla\bff{\mu}^{(j-1)}}{\bb{L}^2}^2 \norm{\bff{w}}{L^\infty_T(\bb{H}^2)}^2
	+ 4 \delta k \norm{\nabla\bff{\mu}^{(j)}}{\bb{L}^2}^2.
    \end{align}
\end{lemma}

\begin{proof}
	We first write~$H_1(\bff{v}^{(j)},\bff{w},\bff{\mu}^{(j)})$ as:
	\begin{align}\label{equ:H1 vwm} 
		H_1(\bff{v}^{(j)},\bff{w},\bff{\mu}^{(j)})
		&=
		\int_{t_{j-1}}^{t_j} 
		\inpro{\bff{v}^{(j-1)} \times (\nabla\bff{\eta}^{(j)} +
			\nabla\bff{\mu}^{(j)})}{\nabla\bff{\mu}^{(j)}}_{\bb{L}^2} \ds
		\nn\\
		&\quad +
		\int_{t_{j-1}}^{t_j} 
		\inpro{\bff{v}^{(j-1)} \times (\nabla\bff{w}(t_j) -
			\nabla\bff{w}(s))}{\nabla\bff{\mu}^{(j)}}_{\bb{L}^2} \ds
		\nn\\
		&\quad +
		\int_{t_{j-1}}^{t_j} 
		\inpro{
			\big(\bff{\eta}^{(j-1)} + \bff{\mu}^{(j-1)}\big)
			\times \nabla\bff{w}(s)}{\nabla\bff{\mu}^{(j)}}_{\bb{L}^2} \ds
		\nn\\
		&\quad +
		\int_{t_{j-1}}^{t_j} 
		\inpro{
			\big(\bff{w}(t_{j-1})-\bff{w}(s)\big) \times \nabla\bff{w}(s)}
		{\nabla\bff{\mu}^{(j)}}_{\bb{L}^2} \ds
		\nn\\
		&=: T_1 + T_2 + T_3 + T_4.
	\end{align}
	We now estimate each term on the right-hand side of~\eqref{equ:H1 vwm}. Firstly, by using
	H\"older's and Young's inequalities we have for the first term
	\begin{align}\label{equ:est T1}
		|T_1|
		&:=
		\Big|
		\int_{t_{j-1}}^{t_j} 
		\inpro{\bff{v}^{(j-1)} \times (\nabla\bff{\eta}^{(j)} +
			\nabla\bff{\mu}^{(j)})}{\nabla\bff{\mu}^{(j)}}_{\bb{L}^2} \ds
		\Big|
		\nn\\
		&=
		\Big|
		\int_{t_{j-1}}^{t_j} 
		\inpro{\bff{v}^{(j-1)} \times
			\nabla\bff{\eta}^{(j)}}{\nabla\bff{\mu}^{(j)}}_{\bb{L}^2} \ds 
		\Big|
		\nn\\
		&\le
		k \norm{\bff{v}^{(j-1)}}{\bb{L}^\infty} 
		\norm{\nabla\bff{\eta}^{(j)}}{\bb{L}^2} 
		\norm{\nabla\bff{\mu}^{(j)}}{\bb{L}^2} 
		\nn\\
		&\le
		ck \norm{\bff{v}^{(j-1)}}{\bb{L}^\infty}^2
			\norm{\nabla\bff{\eta}^{(j)}}{\bb{L}^2}^2
		+
		\delta k
		\norm{\nabla\bff{\mu}^{(j)}}{\bb{L}^2}^2,
	\end{align}
	where in the last step we used the Sobolev embedding.
	Similarly, we have for the third term,
	\begin{align*} 
		|T_3|
		&:=
		\Big|
		\int_{t_{j-1}}^{t_j} 
		\inpro{\big(\bff{\eta}^{(j-1)} + \bff{\mu}^{(j-1)}\big)
			\times \nabla\bff{w}(s)}{\nabla\bff{\mu}^{(j)}}_{\bb{L}^2} \ds
		\Big|
		\\
		&\le
		k
		\norm{\bff{\eta}^{(j-1)} + \bff{\mu}^{(j-1)}}{\bb{L}^4} 
		\norm{\nabla\bff{w}}{L^\infty_T(\bb{L}^4)}
		\norm{\nabla\bff{\mu}^{(j)}}{\bb{L}^2}
		\\
		&\le
		c k
		\norm{\bff{\eta}^{(j-1)} + \bff{\mu}^{(j-1)}}{\bb{L}^4}^2
		\norm{\bff{w}}{L^\infty_T(\bb{W}^{1,4})}^2
		+
		\delta k
		\norm{\nabla\bff{\mu}^{(j)}}{\bb{L}^2}^2
		\\
		&\leq
		ck \norm{\bff{\eta}^{(j-1)}}{\bb{H}^1}^2  \norm{\bff{w}}{L^\infty_T(\bb{H}^{1+\alpha})}^2
			+
			ck \norm{\bff{\mu}^{(j-1)}}{\bb{L}^2}^2 \norm{\bff{w}}{L^\infty_T(\bb{H}^{1+\alpha})}^2
		\\
		&\quad
		+
			\delta k \norm{\nabla \bff{\mu}^{(j-1)}}{\bb{L}^2}^2 \norm{\bff{w}}{L^\infty_T(\bb{H}^{1+\alpha})}^2
			+
			\delta k \norm{\nabla \bff{\mu}^{(j)}}{\bb{L}^2}^2,
	\end{align*}
	where in the last step we used the Gagliardo--Nirenberg and the Young inequalities.
	For the second term~$T_2$, we have by using~\eqref{equ:uj uj1},
	H\"older's and Young's inequalities
	\begin{align}\label{equ:est T2} 
		|T_2|
		&:=
		\Big|
		\int_{t_{j-1}}^{t_j} 
		\inpro{\bff{v}^{(j-1)} \times (\nabla\bff{w}(t_j) -
			\nabla\bff{w}(s))}{\nabla\bff{\mu}^{(j)}}_{\bb{L}^2} \ds
		\Big|
		\nn\\
		&\le
		\norm{\bff{v}^{(j-1)}}{\bb{L}^\infty}
		\norm{\nabla\bff{\mu}^{(j)}}{\bb{L}^2}
		\int_{t_{j-1}}^{t_j} 
		\norm{\nabla\bff{w}(t_j)-\nabla\bff{w}(s)}{\bb{L}^2} \ds
		\nn\\
		&\le
		k^2  
		\norm{\bff{v}^{(j-1)}}{\bb{L}^\infty}
		\norm{\nabla\bff{\mu}^{(j)}}{\bb{L}^2}
		\norm{\nabla\bff{w}}{W^{1,\infty}_T(\bb{L}^2)} 
		\nn\\
		&\le
		ck^3
			\norm{\bff{v}^{(j-1)}}{\bb{L}^\infty}^2
			\norm{\nabla\bff{w}}{W^{1,\infty}_T(\bb{L}^2)}^2
			+
			\delta k
			\norm{\nabla\bff{\mu}^{(j)}}{\bb{L}^2}^2.
	\end{align}
	Similarly, for the last term~$T_4$ we have, again by using~\eqref{equ:uj uj1},
	\begin{align*} 
		|T_4|
		&:=
		\Big|
		\int_{t_{j-1}}^{t_j} 
		\inpro{
			\big(\bff{w}(t_{j-1})-\bff{w}(s)\big) \times \nabla\bff{w}(s)}
		{\nabla\bff{\mu}^{(j)}}_{\bb{L}^2} \ds
		\Big|
		\\
		&\le
		k^2 \norm{\bff{w}}{W^{1,\infty}_T(\bb{L}^4)} 
		\norm{\nabla\bff{w}}{L^\infty_T(\bb{L}^4)} 
		\norm{\nabla\bff{\mu}^{(j)}}{\bb{L}^2}
		\\
		&\le
		ck^3
			\norm{\bff{w}}{W^{1,\infty}_T(\bb{H}^1)}^2
			\norm{\bff{w}}{L^\infty_T(\bb{H}^{1+\alpha})}^2
			+
			\delta k
			\norm{\nabla\bff{\mu}^{(j)}}{\bb{L}^2}^2.
	\end{align*}
	Altogether, we deduce~\eqref{equ:H1 vj}. To prove~\eqref{equ:H2 vj} we
	write~$H_2(\bff{v}^{(j)},\bff{w},\bff{\mu}^{(j)})$ as
	\begin{align*} 
		H_2(\bff{v}^{(j)},\bff{w},\bff{\mu}^{(j)})
		&=
		\int_{t_{j-1}}^{t_j} 
		\inpro{|\bff{v}^{(j-1)}|^2
			\big(\bff{w}(t_j)-\bff{w}(s)\big)}{\bff{\mu}^{(j)}}_{\bb{L}^2}
		\ds
		\\
		&\quad
		+
		\int_{t_{j-1}}^{t_j} 
		\inpro{\big( |\bff{v}^{(j-1)}|^2 - |\bff{w}(t_{j-1})|^2 \big)
			\bff{w}(s)}{\bff{\mu}^{(j)}}_{\bb{L}^2} \ds
		\\
		&\quad
		+
		\int_{t_{j-1}}^{t_j} 
		\inpro{\big( |\bff{w}(t_{j-1})|^2 - |\bff{w}(s)|^2 \big)
			\bff{w}(s)}{\bff{\mu}^{(j)}}_{\bb{L}^2} \ds
		\\
		&=: S_1 + S_2 + S_3.
	\end{align*}
	By using H\"older's and Young's inequalities, together
	with~\eqref{equ:uj uj1}, we deduce
	\begin{align*} 
		|S_1|
		&:=
		\Big|
		\int_{t_{j-1}}^{t_j} 
		\inpro{|\bff{v}^{(j-1)}|^2
			\big(\bff{w}(t_j)-\bff{w}(s)\big)}{\bff{\mu}^{(j)}}_{\bb{L}^2}
		\Big|
		\\
		&\le
		\norm{\bff{v}^{(j-1)}}{\bb{L}^6} 
		\norm{|\bff{v}^{(j-1)}||\bff{\mu}^{(j)}|}{\bb{L}^2} 
		\int_{t_{j-1}}^{t_j} 
		\norm{\bff{w}(t_j)-\bff{w}(s)}{\bb{L}^3} \ds
		\\
		&\le
		k^2
		\norm{\bff{v}^{(j-1)}}{\bb{H}^1} 
		\norm{|\bff{v}^{(j-1)}||\bff{\mu}^{(j)}|}{\bb{L}^2} 
		\norm{\bff{w}}{W^{1,\infty}_T(\bb{L}^3)} 
		\\
		&\le
		ck^3
		\norm{\bff{v}^{(j-1)}}{\bb{H}^1}^2
		\norm{\bff{w}}{W^{1,\infty}_T(\bb{H}^1)}^2
		+
		\delta k
		\norm{|\bff{v}^{(j-1)}||\bff{\mu}^{(j)}|}{\bb{L}^2}^2,
	\end{align*}
	\begin{align*} 
		|S_2|
		&:=
		\Big|
		\int_{t_{j-1}}^{t_j} 
		\inpro{\big( |\bff{v}^{(j-1)}|^2 - |\bff{w}(t_{j-1})|^2 \big)
			\bff{w}(s)}{\bff{\mu}^{(j)}}_{\bb{L}^2} \ds
		\Big|
		\\
		&\le
		k
		\norm{\bff{\eta}^{(j-1)}+\bff{\mu}^{(j-1)}}{\bb{L}^2}
		\norm{\bff{v}^{(j-1)}+\bff{w}(t_{j-1})}{\bb{L}^6}
		\norm{\bff{w}}{L^\infty_T(\bb{L}^6)} 
		\norm{\bff{\mu}^{(j)}}{\bb{L}^6}
		\\
		&\le
		ck
		\norm{\bff{\eta}^{(j-1)}+\bff{\mu}^{(j-1)}}{\bb{L}^2}^2
		\norm{\bff{v}^{(j-1)}+\bff{w}(t_{j-1})}{\bb{H}^1}^2
		\norm{\bff{w}}{L^\infty_T(\bb{H}^1)}^2
		+
		\delta k
		\norm{\bff{\mu}^{(j)}}{\bb{H}^1}^2,
	\end{align*}
	and
	\begin{align*} 
		|S_3|
		&:=
		\Big|
		\int_{t_{j-1}}^{t_j} 
		\inpro{\big( |\bff{w}(t_{j-1})|^2 - |\bff{w}(s)|^2 \big)
			\bff{w}(s)}{\bff{\mu}^{(j)}}_{\bb{L}^2} \ds
		\Big|
		\\
		&\le
		c k^2
		\norm{\bff{w}}{W^{1,\infty}_T(\bb{L}^2)} 
		\norm{\bff{w}(t_{j-1})+\bff{w}}{L^\infty_T(\bb{L}^6)} 
		\norm{\bff{w}}{L^\infty_T(\bb{L}^6)} 
		\norm{\bff{\mu}^{(j)}}{\bb{L}^6} 
		\\
		&\le
		c k^3 
		\norm{\bff{w}}{W^{1,\infty}_T(\bb{L}^2)}^2 
		\norm{\bff{w}}{L^{\infty}_T(\bb{H}^1)}^4
		+
		\delta k
		\norm{\bff{\mu}^{(j)}}{\bb{H}^1}^2.
	\end{align*}
	Altogether, we obtain~\eqref{equ:H2 vj}.

    Next, we show \eqref{equ:H1 new}.
    By estimating the terms $T_1$ and $T_2$ differently in~\eqref{equ:H1 vwm}, we obtain instead of~\eqref{equ:est T1} and~\eqref{equ:est T2},
\begin{align*}
	\abs{T_1} &\leq
	ck \norm{\bff{v}^{(j-1)}}{\bb{L}^6}^2 \norm{\nabla \bff{\eta}^{(j)}}{\bb{L}^3}^2 
	+ \delta k\norm{\nabla \bff{\mu}^{(j)}}{\bb{L}^2}^2,
	\\
	\abs{T_2} &\leq
	ck^3 \norm{\bff{v}^{(j-1)}}{\bb{L}^6}^2 \norm{\nabla \bff{w}}{W^{1,\infty}_T(\bb{L}^3)}^2 
	+ \delta k \norm{\nabla \bff{\mu}^{(j)}}{\bb{L}^2}^2.
\end{align*}
    The terms $T_3$ and $T_4$ are estimated in the same manner as before. Therefore, by the embedding $\bb{H}^1 \hookrightarrow \bb{L}^6$, we have~\eqref{equ:H1 new}. This completes the proof of the lemma.
\end{proof}

For any~$t\in[0,T]$, we define~$\bff{\rho}(t) :=
\bff{u}(t)-P_h\bff{u}(t)$. It follows 
that~$\pa_t\bff{\rho}(t) = \pa_t\bff{u}(t)-P_h\pa_t\bff{u}(t)$
so that~\eqref{equ:proj ineq} implies, for~$\alpha\in (\frac12, \alpha_0)$, we have
	\begin{equation}\label{equ:rho rhot}
		\begin{aligned}
			\norm{\bff{\rho}(t)}{\bb{L}^2} 
			+
			h \norm{\nabla\bff{\rho}(t)}{\bb{L}^2}
			&\le
			c h^{1+\alpha} \norm{\bff{u}(t)}{\bb{H}^{1+\alpha}}, 
			\\
			\norm{\partial_t \bff{\rho}(t)}{\bb{L}^2} 
			+
			h \norm{\nabla \partial_t \bff{\rho}(t)}{\bb{L}^2}
			&\le
			c h^{1+\alpha} \norm{\partial_t \bff{u}(t)}{\bb{H}^{1+\alpha}}.
		\end{aligned}
	\end{equation}
For any~$j=1,\ldots,n$, 
letting~$\bff{\rho}^{(j)} := \bff{\rho}(t_j)$ and~$\bff{\theta}^{(j)} := P_h\bff{u}(t_j)-\bff{u}_h^{(j)}$,
we split the error into
\begin{align}\label{equ:error theta rho}
	\bff{u}(t_j)-\bff{u}_h^{(j)}
	=
	\bff{\rho}^{(j)} + \bff{\theta}^{(j)}.
\end{align}

In the following proof, we will also use the following inequality: If $\bff{v}\in W^{1,\infty}_T(\bb{L}^p)$ for some $p\in [1,\infty]$ and if $s_1,s_2\in [t_{j-1},t_j]$ for some $j=1,\ldots,N$ such that $s_1<s_2$, then
\begin{equation}\label{equ:uj uj1}
	\norm{\bff{v}(s_2)-\bff{v}(s_1)}{\bb{L}^p}
	\leq
	\int_{s_1}^{s_2} \norm{\partial_t \bff{v} (s)}{\bb{L}^p}\, \ds
	\leq
	k \norm{\bff{v}}{W^{1,\infty}_T(\bb{L}^p)}.
\end{equation}
We can now prove an \emph{a priori} error estimate for the scheme presented in Algorithm~\ref{alg:fem llbe}. {By using interpolation in the temporal variable, one can deduce from the theorem below that an approximate solution $\bff{u}_{h,k}$ can be defined from $\bff{u}_h^{(n)}$ such that $\norm{\bff{u}_{h,k}-\bff{u}}{L^\infty(0,T;\bb{L}^2)}+\norm{\bff{u}_{h,k}-\bff{u}}{L^2(0,T;\bb{H}^1)}$ satisfies \eqref{equ:err est}.}

\begin{theorem}\label{the:err llb}
	Assume that the exact solution $\bff{u}$ of the LLBE exists with regularity given by \eqref{equ:regularity llb}. {Let $\big\{\bff{u}_h^{(j)}\big\}_{j=1}^N$ be the sequence defined by Algorithm~\ref{alg:fem llbe}.}
	For any $n=1,\ldots,N$, where $t_n\in [0,T]$, we have
	\begin{align}\label{equ:err est}
	\norm{\bff{u}_h^{(n)}-\bff{u}(t_n)}{\bb{L}^2}
	+
	\left(k \sum_{j=1}^n \norm{\bff{u}_h^{(j)}-\bff{u}(t_j)}{\bb{H}^1}^2 \right)^{\frac12}
	\leq c\ell_h (h^\alpha+k),
	\end{align}
	where~$\ell_h$ is defined in~\eqref{equ:ell h}, and $c$ is a constant independent of $h$ and $k$ (but may depend on $T$). In particular, under the time-step constraint $k=o(\ell_h^{-1})$ as $h,k\to 0^+$, we have 
	\begin{equation*}
	\norm{\bff{u}_h^{(n)}-\bff{u}(t_n)}{\bb{L}^2}+
	\left(k \sum_{j=1}^n \norm{\bff{u}_h^{(j)}-\bff{u}(t_j)}{\bb{H}^1}^2 \right)^{\frac12}\to 0 \;\text{ as }\; h,k\to 0^+.
	\end{equation*}
\end{theorem}

\begin{proof}
	For any $ \bff{\phi}_h\in\bb{V}_h$, the exact solution $\bff{u}$ of~\eqref{equ:LLB pro} satisfies, for any~$\bff{\phi}_h\in\bb{V}_h$,
	\begin{align}\label{equ:dis ex llb}
		&\inpro{\bff{u}(t_j)-\bff{u}(t_{j-1})}{\bff{\phi}_h}_{\bb{L}^2}
		+
		\epsilon\inpro{\nabla\bff{u}(t_j)-\nabla\bff{u}(t_{j-1})}{\nabla\bff{\phi}_h}_{\bb{L}^2}
		+
		\int_{t_{j-1}}^{t_j} \kappa_1 \inpro{\nabla\bff{u}(s)}{\nabla\bff{\phi}_h}_{\bb{L}^2}\ds
		\nn\\
		&\quad
		+
		\int_{t_{j-1}}^{t_j}
		\gamma \inpro{\bff{u}(s)\times\nabla\bff{u}(s)}{\nabla\bff{\phi}_h}_{\bb{L}^2}\ds
		+
		\int_{t_{j-1}}^{t_j}
		\kappa_2 \inpro{(1+\mu |\bff{u}(s)|^2)\bff{u}(s)}{\bff{\phi}_h}_{\bb{L}^2} \ds
		=0.
	\end{align}
	Subtracting~\eqref{equ:dis ex llb} from~\eqref{equ:dis1 llb}, noting~\eqref{equ:error theta rho} and~\eqref{equ:orth proj}, we obtain
	\begin{align}\label{equ:theta llb}
		\inpro{\bff{\theta}^{(j)}-\bff{\theta}^{(j-1)}}{\bff{\phi}_h}_{\bb{L}^2}
		&+
		k\kappa_1 \inpro{\nabla\bff{\theta}^{(j)}+\nabla \bff{\rho}^{(j)}}{\nabla\bff{\phi}_h}_{\bb{L}^2}
		+
		\kappa_1
		\int_{t_{j-1}}^{t_j} 
		\inpro{\nabla\bff{u}^\epsilon(t_j)- \nabla \bff{u}^\epsilon(s)}{\nabla\bff{\phi}_h}_{\bb{L}^2} \ds
		\nn\\
		&
		+ k \kappa_2 \inpro{\bff{\theta}^{(j)}}{\bff{\phi}_h}_{\bb{L}^2}
		+
		\int_{t_{j-1}}^{t_j} \kappa_2 \inpro{\bff{u}^\epsilon(t_j)- \bff{u}^\epsilon(s)}{\bff{\phi}_h}_{\bb{L}^2} \ds 
		\nn\\
		&+
		\gamma H_1(\bff{u}_h^{(j)},\bff{u}^\epsilon,\bff{\phi}_h) 
		+ 
		\kappa_2 \mu
		H_2(\bff{u}_h^{(j)},\bff{u}^\epsilon,\bff{\phi}_h) 
		\nn\\
		&
		+
		k \kappa_2 \mu
		\inpro{|\bff{u}_h^{(j-1)}|^2
			\big(\bff{\theta}^{(j)}+\bff{\rho}^{(j)}\big)}{\bff{\phi}_h}_{\bb{L}^2}
		=0,
	\end{align}
	where $H_1$ and~$H_2$ are defined in Lemma~\ref{lem:aux lem}.

    Putting~$\bff{\phi}_h=\bff{\theta}^{(j)}$ in~\eqref{equ:theta llb}, using the vector identity~\eqref{equ:vec ab a}, and rearranging the terms, we have
\begin{align}\label{equ:theta sum llb}
	&\frac{1}{2} \left(\norm{\bff{\theta}^{(j)}}{\bb{L}^2}^2 - \norm{\bff{\theta}^{(j-1)}}{\bb{L}^2}^2\right)
	+
	\frac{1}{2} \norm{\bff{\theta}^{(j)}- \bff{\theta}^{(j-1)}}{\bb{L}^2}^2
	\nn \\
	&\quad
	+
	k\kappa_1 \norm{\nabla \bff{\theta}^{(j)}}{\bb{L}^2}^2
	+
	k\kappa_2 \norm{\bff{\theta}^{(j)}}{\bb{L}^2}^2
	+
	\kappa_2 \mu k \norm{\big|\bff{u}_h^{(j-1)}\big| \big|\bff{\theta}^{(j)}\big|}{\bb{L}^2}^2
	\nn \\
	&=
	-
	k\kappa_1 \inpro{\nabla\bff{\rho}^{(j)}}{\nabla\bff{\theta}^{(j)}}_{\bb{L}^2}
	-
	\int_{t_{j-1}}^{t_j} \kappa_1 \inpro{\nabla\bff{u}(t_j)- \nabla \bff{u}(s)}{\nabla\bff{\theta}^{(j)}}_{\bb{L}^2} \ds
	\nn \\
	&\quad
	-
	\int_{t_{j-1}}^{t_j} \kappa_2 \inpro{\bff{u}(t_j)- \bff{u}(s)}{\bff{\theta}^{(j)}}_{\bb{L}^2} \ds
	-
	\gamma H_1(\bff{u}_h^{(j)},\bff{u},\bff{\theta}^{(j)}) 
	\nn \\
	&\quad
 	+ 
 	\kappa_2 \mu
	H_2(\bff{u}_h^{(j)},\bff{u},\bff{\theta}^{(j)}) 
 	-
 	k \kappa_2 \mu
 	\inpro{|\bff{u}_h^{(j-1)}|^2 \bff{\rho}{(j)}}{\bff{\theta}^{(j)}}_{\bb{L}^2}
	\nn \\
	&\le
	c\norm{\nabla\bff{\theta}^{(j)}}{\bb{L}^2} 
	\left(
		 k \norm{\nabla\bff{\rho}^{(j)}}{\bb{L}^2} 
		+
		 \int_{t_{j-1}}^{t_j}
		\norm{\nabla\bff{u}(t_j)-\nabla\bff{u}(s)}{\bb{L}^2} \ds
	\right)
	\nn \\
	&\quad
	+
	c\norm{\bff{\theta}^{(j)}}{\bb{L}^2} 
	\int_{t_{j-1}}^{t_j} 
	\norm{\bff{u}(t_j)-\bff{u}(s)}{\bb{L}^2} \ds
	+
	c
	\abs{H_1(\bff{u}_h^{(j)},\bff{u},\bff{\theta}^{(j)})}
	\nn \\
	&\quad
 	+ 
	c
	\abs{ H_2(\bff{u}_h^{(j)},\bff{u},\bff{\theta}^{(j)})}
	+
	ck
	\norm{|\bff{u}_h^{(j-1)}||\bff{\theta}^{(j)}|}{\bb{L}^2}
	\norm{\bff{u}_h^{(j-1)}}{\bb{L}^4}
	\norm{\bff{\rho}^{(j)}}{\bb{L}^4}
	\nn\\
	&=:
	R_1 + \cdots + R_5.
\end{align}
We now estimate each term on the right-hand side. Let $\delta>0$ be given and let $\alpha$ satisfies~\eqref{equ:condition alpha}. For
the first term, by~\eqref{equ:uj uj1} and Young's inequality (noting the regularity~\eqref{equ:regularity llb}),
\begin{align}\label{equ:llb R1}
	R_1
	&:=
	c \norm{\nabla\bff{\theta}^{(j)}}{\bb{L}^2} 
	\left(
		 k \norm{\nabla\bff{\rho}^{(j)}}{\bb{L}^2} 
		+
		\int_{t_{j-1}}^{t_j}
		\norm{\nabla\bff{u}(t_j)-\nabla\bff{u}(s)}{\bb{L}^2} \ds
	\right)
	\nn\\
	&\leq
	\norm{\nabla\bff{\theta}^{(j)}}{\bb{L}^2} 
	\left(
		c k \norm{\nabla\bff{\rho}^{(j)}}{\bb{L}^2} 
		+
		c k^2 
		\norm{\nabla \bff{u}}{W^{1,\infty}_T(\bb{L}^2)} 
	\right)
	\nn\\
	&\leq
	\delta k \norm{\nabla\bff{\theta}^{(j)}}{\bb{L}^2}^2
	+
	ck \norm{\nabla \bff{\rho}^{(j)}}{\bb{L}^2}^2
	+ 
	ck^3 \norm{\nabla\bff{u}}{W^{1,\infty}_T(\bb{L}^2)}^2
	\nn\\
	&\leq
	ck h^{2\alpha} + ck^3
	+
	\delta k \norm{\nabla\bff{\theta}^{(j)}}{\bb{L}^2}^2,
\end{align}
where in the last step we used~\eqref{equ:rho rhot}.
For the second term, by H\"older's and Young's inequalities, and \eqref{equ:uj uj1}, we have
\begin{align*}
	R_2
	&:=
	c\norm{\bff{\theta}^{(j)}}{\bb{L}^2} 
	\int_{t_{j-1}}^{t_j} 
	\norm{\bff{u}(t_j)-\bff{u}(s)}{\bb{L}^2} \ds
	\\
	&\leq
	c k^2
	\norm{\bff{\theta}^{(j)}}{\bb{L}^2} 
	\norm{\bff{u}}{W^{1,\infty}_T(\bb{L}^2)}
	\le
	\delta k \norm{\bff{\theta}^{(j)}}{\bb{L}^2}^2
	+
	ck^3.
\end{align*}
For the term $R_3$, we invoke Lemma~\ref{lem:aux lem} (with~$\bff{v}^{(j)}=\bff{u}_h^{(j)}$,
$\bff{w}=\bff{u}$, $\bff{\eta}^{(j)}=\bff{\rho}^{(j)}$,
and~$\bff{\mu}^{(j)}=\bff{\theta}^{(j)}$), use the fact that $\bff{u}$ is a strong solution, and apply~\eqref{equ:rho rhot} to obtain 
\begin{align*}
	R_3
	:=
	c\, \big| H_1(\bff{u}_h^{(j)},\bff{u},\bff{\theta}^{(j)}) \big|
	&\le
	c\left(1+\norm{\bff{u}_h^{(j-1)}}{\bb{L}^\infty}^2\right) k (h^{2\alpha}+k^2) 
	+ ck \norm{\bff{\theta}^{(j-1)}}{\bb{L}^2}^2
	+ c k^3 
	\\
	&\quad
	+ \delta k \norm{\nabla \bff{\theta}^{(j-1)}}{\bb{L}^2}^2
	+ 4\delta k \norm{\nabla \bff{\theta}^{(j)}}{\bb{L}^2}^2,
\end{align*}
Similarly, we have
\begin{align*} 
	R_4
	:=
	c\, \big| H_2(\bff{u}_h^{(j)},\bff{u},\bff{\theta}^{(j)}) \big|
	&\leq
	ck \left(1+\norm{\bff{u}_h^{(j-1)}}{\bb{H}^1}^2 \right) \left(h^{2(1+\alpha)}+ k^2 \right) 
	+
	ck \left(1+\norm{\bff{u}_h^{(j-1)}}{\bb{H}^1}^2 \right) \norm{\bff{\theta}^{(j-1)}}{\bb{L}^2}^2
	\\
	&\quad
	+
	ck^3 \left(1+ \norm{\bff{u}_h^{(j-1)}}{\bb{H}^1}^2 \right)
	+
	2 \delta k \norm{\bff{\theta}^{(j)}}{\bb{H}^1}^2
		+
		\delta k
		\norm{|\bff{u}_h^{(j-1)}||\bff{\theta}^{(j)}|}{\bb{L}^2}^2.
\end{align*}
Finally,
\begin{align*} 
	R_5
	&:=
	ck
	\norm{|\bff{u}_h^{(j-1)}||\bff{\theta}^{(j)}|}{\bb{L}^2}
	\norm{\bff{u}_h^{(j-1)}}{\bb{L}^4}
	\norm{\bff{\rho}^{(j)}}{\bb{L}^4}
	\leq
	ck h^{2(1+\alpha)} \norm{\bff{u}_h^{(j-1)}}{\bb{H}^1}^2
	+
	\delta k
	\norm{|\bff{u}_h^{(j-1)}||\bff{\theta}^{(j)}|}{\bb{L}^2}^2.
\end{align*}
Altogether, we derive from~\eqref{equ:theta sum llb} that
\begin{align*}
	&\frac{1}{2} 
	\Big(
	\norm{\bff{\theta}^{(j)}}{\bb{L}^2}^2 - \norm{\bff{\theta}^{(j-1)}}{\bb{L}^2}^2
	\Big)
	+
	\frac{1}{2} 
	\norm{\bff{\theta}^{(j)}- \bff{\theta}^{(j-1)}}{\bb{L}^2}^2
	\nn\\
	&\quad
	+
	k\kappa_1 \norm{\nabla \bff{\theta}^{(j)}}{\bb{L}^2}^2
	+
	k\kappa_2 \norm{\bff{\theta}^{(j)}}{\bb{L}^2}^2
	+
	\kappa_2 \mu k \norm{\big|\bff{u}_h^{(j-1)}\big| \big|\bff{\theta}^{(j)}\big|}{\bb{L}^2}^2
	\nn\\
	&\leq
		ck \left(1+\norm{\bff{u}_h^{(j-1)}}{\bb{H}^1}^2 +\norm{\bff{u}_h^{(j-1)}}{\bb{L}^\infty}^2 \right) \left(h^{2\alpha}+ k^2 \right) 
	+
	ck \left(1+\norm{\bff{u}_h^{(j-1)}}{\bb{H}^1}^2 \right) \norm{\bff{\theta}^{(j-1)}}{\bb{L}^2}^2
	\nn\\
	&\quad
	+
	\delta k \norm{\nabla \bff{\theta}^{(j-1)}}{\bb{L}^2}^2
	+
	8 \delta k \norm{\bff{\theta}^{(j)}}{\bb{H}^1}^2
	+
	2\delta k
	\norm{|\bff{u}_h^{(j-1)}||\bff{\theta}^{(j)}|}{\bb{L}^2}^2.
\end{align*}
We now choose $\delta=\frac{1}{10} \min\{\kappa_1,\kappa_2,\kappa_2 \mu\}$ to absorb the last three terms on the right-hand side to the last three terms on the left-hand side. In this manner we obtain 
\begin{align*}
	&\norm{\bff{\theta}^{(j)}}{\bb{L}^2}^2 
	- 
	\norm{\bff{\theta}^{(j-1)}}{\bb{L}^2}^2
	+
	k\kappa_1 \Big(\norm{\nabla\bff{\theta}^{(j)}}{\bb{L}^2}^2 
	- 
	\norm{\nabla\bff{\theta}^{(j-1)}}{\bb{L}^2}^2 \Big)
	\\
	&\leq
	ck \left(1+\norm{\bff{u}_h^{(j-1)}}{\bb{H}^1}^2 +\norm{\bff{u}_h^{(j-1)}}{\bb{L}^\infty}^2 \right) \left(h^{2\alpha}+ k^2 \right) 
	+
	ck \left(1+\norm{\bff{u}_h^{(j-1)}}{\bb{H}^1}^2 \right) \norm{\bff{\theta}^{(j-1)}}{\bb{L}^2}^2
	\\
	&\leq
	ck\ell_h^2 \left(1+\norm{\bff{u}_h^{(j-1)}}{\bb{H}^1}^2 +\norm{\bff{u}_h^{(j)}}{\bb{H}^1}^2 \right) \left(h^{2\alpha}+ k^2 \right) 
	+
	ck \left(1+\norm{\bff{u}_h^{(j-1)}}{\bb{H}^1}^2 \right) \norm{\bff{\theta}^{(j-1)}}{\bb{L}^2}^2,
\end{align*}
where inverse inequality~\eqref{equ:inverse Vh} was used in the last step.
Summing this over $j=1,\ldots, n$, and using Lemma~\ref{lem: stability2}, we deduce that
\[
	\norm{\bff{\theta}^{(n)}}{\bb{L}^2}^2
	+
	k \sum_{j=1}^n \norm{\nabla \bff{\theta}^{(j)}}{\bb{L}^2}^2
	\le
	ck \ell_h^2 (h^{2\alpha}+k^2) + ck \sum_{j=1}^n \norm{\bff{\theta}^{(j)}}{\bb{L}^2}^2,
\]
where $c$ is independent of $n$, $h$, and $k$.
Applying the discrete version of the Gronwall inequality~\citep[Lemma~1]{Sug69}, we
deduce that
\begin{equation*}
	\norm{\bff{\theta}^{(n)}}{\bb{L}^2} + k \sum_{j=1}^n \norm{\nabla \bff{\theta}^{(j)}}{\bb{L}^2}^2 
	\leq ce^{cT}\ell_h (h^{\alpha}+k).
\end{equation*}
Using~\eqref{equ:rho rhot}
and the triangle inequality, we obtain the required estimate.
\end{proof}

\section{Solution to the $\epsilon$-LLBE}\label{sec:reg tec}

{The theoretical rate $O(\ell_h(h^\alpha+k))$ established in~\eqref{equ:err est} is not optimal in light of the numerical results from Simulation 1 in Section~\ref{sec:num sim}; see Figures~\ref{fig:order sim4} and~\ref{fig:order exp 1b k}.}
This motivates us to consider a regularisation of the LLBE and to investigate a numerical method for discretising the $\epsilon$-LLBE, which serves as a theoretical proxy for the original LLBE. First, we need establish the existence of a solution to the regularised problem~\eqref{equ:LLB eps pro} and its uniqueness. In Section~\ref{sec:conver}, we then show convergence of $\bff{u}^\epsilon$ to $\bff{u}$. Throughout this section, let $T>0$ be a given number which can be arbitrarily large.


\subsection{Existence and uniqueness}

As before, the Faedo--Galerkin method will be used to show the existence of a solution to the $\epsilon$-LLBE.
Let $\{(\bff{e}_i,\lambda_i)\}_{i=1}^{\infty}$ be a sequence of eigenpairs of the negative Neumann Laplace operator.
Let $\mathcal{S}_n:=\text{span}\{\bff{e}_1,\cdots,\bff{e}_n\}$ and let
$\Pi_n$ be the~$\bb{L}^2$-projection onto~$\mathcal{S}_n$. We consider the following
approximation to~\eqref{equ:LLB eps pro}: Find~$\bff{u}_n^\epsilon(\cdot,t)\in\mathcal{S}_n$
satisfying
\begin{equation}\label{equ:Gal LLB eps}
\begin{alignedat}{2}
 \pa_t\bff{u}^\epsilon_n
-\epsilon \Delta\pa_t\bff{u}^\epsilon_n
&=
\kappa_1\Delta\bff{u}^\epsilon_n
+
\gamma\Pi_n\bigl(\bff{u}^\epsilon_n\times\Delta\bff{u}^\epsilon_n\bigr)
-
\kappa_2\Pi_n\bigl((1+\mu|\bff{u}^\epsilon_n|^2)\bff{u}^\epsilon_n\bigr)
&\quad &\text{in } (0,T) \times \mathscr{D},
\\
\bff{u}_n^\epsilon(0) &= \bff{u}_{0,n}^\epsilon
		&\quad &\text{in } \mathscr{D}.
\end{alignedat}
\end{equation}
We assume that $\bff{u}_0^\epsilon$ in \eqref{equ:ue ini con} and~$\bff{u}_{0,n}^\epsilon\in\mathcal{S}_n$ in \eqref{equ:Gal LLB eps} are chosen such that the initial data approximations satisfy
\begin{subequations}\label{equ:approx u0 s}
	\begin{align}
		\lim_{\epsilon\to 0^+}
		\norm{\bff{u}_0^\epsilon-\bff{u}_0}{\bb{H}^s_\Delta} &= 0,
		\label{equ:u0 eps con}
		\\
		\lim_{n\to\infty}
		\norm{\bff{u}_{0,n}^\epsilon-\bff{u}_0^\epsilon}{\bb{H}^s_\Delta} &= 0
		\quad\text{uniformly with respect to $\epsilon\in (0,1)$},
		\label{equ:u0n H2}
	\end{align}
\end{subequations} 
for some $s\in \{1,2\}$ to be specified in the estimates.
It follows from \eqref{equ:u0 eps con} and~\eqref{equ:u0n H2} that
\begin{equation}\label{equ:u0n eps con}
\norm{\bff{u}_{0,n}^\epsilon}{\bb{H}^s_\Delta} \le c, \quad\forall\epsilon\in(0,1), \ \forall n\in \bb{N}.
\end{equation}
Local existence of a solution to~\eqref{equ:Gal LLB eps} can be shown as in Lemma~\ref{lem:loc exi}.

We now show some bounds on the local solution~$\bff{u}_n^\epsilon$
of~\eqref{equ:Gal LLB eps}, which will imply its global existence, i.e., existence of solution on $[0,T]$ for any given $T>0$. We also show that
for each~$\epsilon\in (0,1)$, the sequence~$\{\bff{u}_n^{\epsilon}\}$ converges as $n\to\infty$
and its limit is the solution~$\bff{u}^\epsilon$ of~\eqref{equ:LLB eps pro}.

First, we rewrite~\eqref{equ:Gal LLB eps} in the following equivalent form: for
every~$t\in [0,T]$,
\begin{subequations}\label{equ:LLB eps wea}
	\begin{alignat}{1}
		\inpro{\pa_t\bff{u}_n^\epsilon(t)}{\bff{v}}_{\bb{L}^2}
		&+
		\epsilon \inpro{\nabla\pa_t \bff{u}_n^\epsilon(t)}{\nabla\bff{v}}_{\bb{L}^2}
		+
		\kappa_1 \inpro{\nabla\bff{u}_n^\epsilon(t)}{\nabla\bff{v}}_{\bb{L}^2}
		+
		\kappa_2
		\inpro{\bff{u}_n^\epsilon(t)}{\bff{v}}_{\bb{L}^2}
		\nn\\
		&=
		\gamma
		\inpro{\bff{u}_n^\epsilon(t)\times\Delta\bff{u}_n^\epsilon(t)}{ \bff{v}}_{\bb{L}^2}
		-
		\kappa_2 \mu
		\inpro{|\bff{u}_n^\epsilon(t)|^2 \bff{u}_n^\epsilon(t)}{\bff{v}}_{\bb{L}^2},
		\quad \forall\bff{v}\in\bb{H}^1,
		\label{equ:LLB eps wea equ}
		\\
		\bff{u}_n^\epsilon(0) &= \bff{u}^\epsilon_{0,n}.
		\label{equ:LLB eps wea con}
	\end{alignat}
\end{subequations}
The stability of the approximate solution $\bff{u}_n^\epsilon$ in various spaces will be shown in a series of lemmas.
\begin{lemma}\label{lem:ene est 1}
	Assume that~\eqref{equ:u0 eps con} and~\eqref{equ:u0n H2} hold for $s=1$. For any $\epsilon\in [0,1)$ and
	$n\in \bb{N}$,
\begin{equation}\label{equ:un H1 H2}
\norm{\bff{u}^\epsilon_n(t)}{\bb{H}^1}^2
+
\epsilon \norm{\Delta\bff{u}^\epsilon_n(t)}{\bb{L}^2}^2
+
\int_0^t \norm{\Delta\bff{u}^\epsilon_n(s)}{\bb{L}^2}^2 \ds
+
\int_0^t 
\norm{\bff{u}^\epsilon_n(s)}{\bb{L}^4}^4 \ds
\le c,
\end{equation}
where the constant $c$ depends on $\norm{\bff{u}_0}{\bb{H}^1_\Delta}$, but is independent of~$n$, $\epsilon$, and~$t$.
\end{lemma}
\begin{proof}
Letting~$\bff{v}=\bff{u}_n^\epsilon(t)$ in~\eqref{equ:LLB eps wea equ} and
noting $(\bff{a}\times\bff{b})\cdot\bff{b}=0$ give
\begin{align}\label{equ:un L2 H1}
\frac{1}{2}\ddt{ }\|\bff{u}^\epsilon_n(t)\|_{\bb{L}^2}^2 
&+
\frac{\epsilon}{2}\ddt{ }\|\nabla\bff{u}^\epsilon_n(t)\|_{\bb{L}^2}^2
+
\kappa_1\|\nabla\bff{u}^\epsilon_n(t)\|_{\bb{L}^2}^2
+
\kappa_2 \norm{\bff{u}_n^\epsilon(t)}{\bb{L}^2}^2
+
\kappa_2\mu \norm{\bff{u}_n^\epsilon(t)}{\bb{L}^4}^4
= 0.
\end{align}
In a similar fashion, by choosing $\bff{v}=-\Delta\bff{u}^\epsilon_n(t)\in \mathcal{S}_n$
in~\eqref{equ:LLB eps wea equ} we obtain
\begin{align}\label{equ:un H1}
\frac{1}{2}\ddt{}\|\nabla\bff{u}^\epsilon_n(t)\|_{\bb{L}^2}^2 
&+
\frac{\epsilon}{2}\ddt{ }\|\Delta\bff{u}^\epsilon_n(t)\|_{\bb{L}^2}^2
+
\kappa_1\|\Delta\bff{u}^\epsilon_n(t)\|_{\bb{L}^2}^2
+
\kappa_2 \norm{\nabla \bff{u}_n^\epsilon(t)}{\bb{L}^2}^2
\nn\\
&=
\kappa_2\mu
\inpro{|\bff{u}^\epsilon_n(t)|^2 \bff{u}^\epsilon_n(t)}{\Delta\bff{u}^\epsilon_n(t)}_{\bb{L}^2}.
\end{align}
Note that by applying integration by parts,
\begin{align*}
	\kappa_2 \mu 
	\inpro{|\bff{u}^\epsilon_n|^2 \bff{u}^\epsilon_n}{\Delta\bff{u}^\epsilon_n}_{\bb{L}^2}
	&=
	- \kappa_2 \mu 
	\inpro{\nabla\big(|\bff{u}^\epsilon_n|^2 \bff{u}_n^\epsilon\big)}
	{\nabla\bff{u}_n^\epsilon}_{\bb{L}^2}
	\\
	&=
	-2 \kappa_2 \mu 
	\norm{\bff{u}_n^\epsilon \cdot \nabla \bff{u}_n^\epsilon}{\bb{L}^2}^2
	-
	\kappa_2 \mu \norm{|\bff{u}_n^\epsilon|\, |\nabla \bff{u}_n^\epsilon|}{L^2}^2 
	\leq 0.
\end{align*}
Integrating~\eqref{equ:un L2 H1} and~\eqref{equ:un H1} yields~\eqref{equ:un H1 H2}, thus completing the proof of the lemma.
\end{proof}


\begin{lemma}\label{lem:ene est dtu}
Assume that~\eqref{equ:u0 eps con} and~\eqref{equ:u0n H2} hold for $s=1$. 
For any~$\epsilon\in(0,1)$ and~$n\in\bb{N}$, the approximate
solution~$\bff{u}_n^\epsilon$ satisfies
\begin{equation}\label{equ:est dt u L2}
	\norm{\pa_t\bff{u}_n^\epsilon(t)}{\bb{L}^2}^2
	+
	\epsilon \norm{\nabla\pa_t\bff{u}_n^\epsilon(t)}{\bb{L}^2}^2
	\lesssim
	\frac{1}{\epsilon^\beta},
\end{equation}
where $\beta$ is defined by
	\begin{align}
		\label{equ:d ast}
		\beta &:= \min\{2,d\}.
	\end{align}
The constant depends on $\norm{\bff{u}_0}{\bb{H}^1_\Delta}$, but is independent of $n$, $\epsilon$, and $t$.
\end{lemma}

\begin{proof}
Setting $\bff{v}=\partial_t \bff{u}_n^\epsilon(t)$ in~\eqref{equ:LLB eps wea equ} and applying Young's inequality give
\begin{align*}
	&\norm{\partial_t \bff{u}_n^\epsilon(t)}{\bb{L}^2}^2
	+
	\epsilon \norm{\nabla\partial_t \bff{u}_n^\epsilon (t)}{\bb{L}^2}^2
	\\
	&=
	\kappa_1 \inpro{\Delta \bff{u}_n^\epsilon (t)}{\pa_t \bff{u}_n^\epsilon (t)}_{\bb{L}^2}
	-
	\kappa_2 \inpro{\bff{u}_n^\epsilon (t)}{\pa_t \bff{u}_n^\epsilon (t)}_{\bb{L}^2}
	+
	\gamma
	\inpro{\bff{u}_n^\epsilon(t)\times\Delta\bff{u}_n^\epsilon(t)}{\pa_t \bff{u}_n^\epsilon (t)}_{\bb{L}^2}
	\\
	&\quad
	-
	\kappa_2 \mu
	\inpro{|\bff{u}_n^\epsilon(t)|^2\bff{u}_n^\epsilon(t)}{\pa_t \bff{u}_n^\epsilon (t)}_{\bb{L}^2}
	\\
	&\leq
	c \norm{\Delta \bff{u}_n^\epsilon (t)}{\bb{L}^2}^2
	+
	c \norm{\bff{u}_n^\epsilon (t)}{\bb{L}^2}^2
	+
	4\norm{\bff{u}_n^\epsilon (t)}{\bb{L}^\infty}^2 \norm{\Delta \bff{u}_n^\epsilon (t)}{\bb{L}^2}^2
	+
	c \norm{\bff{u}_n^\epsilon (t)}{\bb{L}^6}^6
	+
	\frac14 \norm{\pa_t \bff{u}_n^\epsilon (t)}{\bb{L}^2}^2
	\\
	&\leq
	c + c\epsilon^{-1} \norm{\bff{u}_n^\epsilon}{L^\infty_T(\bb{L}^\infty)}^2
	+
	\frac14 \norm{\pa_t \bff{u}_n^\epsilon (t)}{\bb{L}^2}^2,
\end{align*}
where in the last step we used the Sobolev embedding $\bb{H}^1\hookrightarrow \bb{L}^6$ and \eqref{equ:un H1 H2}. The constant $c$ depends on $\norm{\bff{u}_0}{\bb{H}^1_\Delta}$, but is independent of $n$, $\epsilon$, and $t$.
This implies 
\begin{align*}
    \norm{\partial_t \bff{u}_n^\epsilon(t)}{\bb{L}^2}^2
	+
	\epsilon \norm{\nabla\partial_t \bff{u}_n^\epsilon (t)}{\bb{L}^2}^2
    \leq
    c + c\epsilon^{-1} \norm{\bff{u}_n^\epsilon}{L^\infty_T(\bb{L}^\infty)}^2.
\end{align*}
It remains to prove that $\norm{\bff{u}_n^\epsilon}{L^\infty_T(\bb{L}^\infty)}^2 \leq \epsilon^{1-\beta}$.
Now, if $d=1$, then 
\begin{equation}\label{equ:Linfty 1}
	\norm{\bff{u}_n^\epsilon}{L^\infty_T(\bb{L}^\infty)}^2
	\leqs
	\norm{\bff{u}_n^\epsilon}{L^\infty_T(\bb{H}^1)}^2
	\leqs 1.
\end{equation}
For $d=2$, we apply the Sobolev embedding $\bb{H}^{3/2}\hookrightarrow \bb{L}^\infty$, the regularity shift estimate \eqref{equ:shift polyg H1s}, and Lemma~\ref{lem:ene est 1} to obtain
\begin{align}\label{equ:Linfty 23}
	\norm{\bff{u}_n^\epsilon}{L^\infty_T(\bb{L}^\infty)}^2
    \lesssim
    \norm{\bff{u}_n^\epsilon}{L^\infty_T(\bb{H}^{3/2})}^2
	\lesssim
	\norm{\bff{u}_n^\epsilon}{L^\infty_T(\bb{L}^2)}^2
	+
	\norm{\Delta \bff{u}_n^\epsilon}{L^\infty_T(\bb{L}^2)}^2
	\lesssim
	\epsilon^{-1}.
\end{align}
The same result holds for $d=3$ if we use the Sobolev embedding $\bb{H}^2\hookrightarrow \bb{L}^\infty$ and the $\bb{H}^2$ regularity result instead.
This completes the proof of the lemma.
\end{proof}

\begin{lemma}\label{lem:ene est 3}
	Assume that~\eqref{equ:u0 eps con} and~\eqref{equ:u0n H2} hold for $s=1$. 
	For any~$\epsilon\in(0,1)$ and~$n\in\bb{N}$, the approximate
	solution~$\bff{u}_n^\epsilon$ satisfies
\begin{equation}\label{equ:pat Lap une}
	\norm{\nabla\pa_t\bff{u}_n^\epsilon(s)}{\bb{L}^2}^2 
	+
 	\epsilon\norm{\Delta\pa_t\bff{u}_n^\epsilon(s)}{\bb{L}^2}^2 
 	\lesssim 
	\frac{1}{\epsilon^{\beta+1}},
\end{equation}
where $\beta$ is defined by~\eqref{equ:d ast}.
\end{lemma}

\begin{proof}
Choosing~$\bff{v}=-\Delta\pa_t\bff{u}_n^\epsilon(t)\in\mathcal{S}_n$ in~\eqref{equ:LLB eps wea equ}
and using H\"older's inequality, we obtain
\begin{align*}\label{equ:nab pat une}
&\norm{\nabla\pa_t\bff{u}_n^\epsilon(t)}{\bb{L}^2}^2
+
\epsilon \norm{\Delta\pa_t\bff{u}_n^\epsilon(t)}{\bb{L}^2}^2
\nn\\
&\quad=
-\kappa_1 \inpro{\Delta \bff{u}_n^\epsilon(t)}{\Delta \pa_t \bff{u}_n^\epsilon(t)}_{\bb{L}^2}
+
\kappa_2 \inpro{\bff{u}_n^\epsilon(t)}{\Delta \pa_t \bff{u}_n^\epsilon(t)}_{\bb{L}^2}
\nn\\
&\quad
+
\gamma
\inpro{\bff{u}_n^\epsilon(t)\times\Delta\bff{u}_n^\epsilon(t)}
{\Delta\pa_t\bff{u}_n^\epsilon(t)}_{\bb{L}^2}
-
\kappa_2 \mu
\inpro{|\bff{u}_n^\epsilon(t)|^2\bff{u}_n^\epsilon(t)}{\Delta\pa_t\bff{u}_n^\epsilon(t)}_{\bb{L}^2}
\nn\\
&\quad\lesssim
\Big(
\norm{\Delta\bff{u}_n^\epsilon(t)}{\bb{L}^2}
+
\norm{\bff{u}_n^\epsilon(t)}{\bb{L}^2}
+
\norm{\bff{u}_n^\epsilon(t)}{\bb{L}^\infty} \norm{\Delta\bff{u}_n^\epsilon(t)}{\bb{L}^2}
+
\norm{\bff{u}_n^\epsilon(t)}{\bb{L}^\infty}^2\norm{\bff{u}_n^\epsilon(t)}{\bb{L}^2}
\Big)
\norm{\Delta\pa_t\bff{u}_n^\epsilon(t)}{\bb{L}^2}
\nn\\
&\quad\le
\frac{c}{\epsilon}
\Big(
\norm{\Delta\bff{u}_n^\epsilon(t)}{\bb{L}^2}^2
+
\norm{\bff{u}_n^\epsilon(t)}{\bb{L}^2}^2
+
\norm{\bff{u}_n^\epsilon(t)}{\bb{L}^\infty}^2
\norm{\Delta\bff{u}_n^\epsilon(t)}{\bb{L}^2}^2
+
\norm{\bff{u}_n^\epsilon(t)}{\bb{L}^\infty}^4
\norm{\bff{u}_n^\epsilon(t)}{\bb{L}^2}^2
\Big)
\nn\\
&\qquad
+
\frac{\epsilon}{2}
\norm{\Delta\pa_t\bff{u}_n^\epsilon(t)}{\bb{L}^2}^2.
\end{align*}
Rearranging the above equation, and using \eqref{equ:u0n eps con} and Lemma~\ref{lem:ene est 1} yield
\begin{align*}
	\norm{\nabla\pa_t\bff{u}_n^\epsilon(t)}{\bb{L}^2}^2
	+
	\epsilon \norm{\Delta\pa_t\bff{u}_n^\epsilon(t)}{\bb{L}^2}^2
	&\lesssim
	\frac{1}{\epsilon} \left(\frac{1}{\epsilon} +1 + \frac{1}{\epsilon} \norm{\bff{u}_n^\epsilon}{L^\infty_T(\bb{L}^\infty)}^2 + \norm{\bff{u}_n^\epsilon}{L^\infty_T(\bb{L}^\infty)}^4 \right).
\end{align*}
By~\eqref{equ:Linfty 1} and \eqref{equ:Linfty 23}, we have the required result.
\end{proof}

\begin{lemma}\label{lem:ene est 4}
	Assume that~\eqref{equ:u0 eps con} and~\eqref{equ:u0n H2} hold for $s=1$. 
	For any~$\epsilon\in(0,1)$, $n\in\bb{N}$, and $\alpha$ satisfying~\eqref{equ:condition alpha}, the approximate
	solution~$\bff{u}_n^\epsilon$ satisfies
	\begin{equation}\label{equ:Lap H1s est}
		\epsilon \norm{\bff{u}_n^\epsilon(t)}{\bb{H}^{1+\alpha}}^2
		+
		\epsilon^{\beta+2}
		\norm{\pa_t\bff{u}_n^\epsilon(t)}{\bb{H}^{1+\alpha}}^2 
		\lesssim 
		1,
	\end{equation}
	where $\beta$ is defined by \eqref{equ:d ast}.
\end{lemma}

\begin{proof}
By~\eqref{equ:shift polyg H1s} in the case $d=2$, and~\eqref{equ:shift polyh W2p} in the case $d=3$, we have for any $\alpha$ satisfying~\eqref{equ:condition alpha},
\begin{align*}
	\epsilon \norm{\bff{u}_n^\epsilon(t)}{\bb{H}^{1+\alpha}}^2
	+
	\epsilon^{\beta+2} 
	\norm{\pa_t\bff{u}_n^\epsilon(s)}{\widetilde{\bb{H}}^{1+\alpha}}^2 
	&\lesssim 
	\epsilon \left(\norm{\bff{u}_n^\epsilon(t)}{\widetilde{\bb{H}}^{-1+\alpha}}^2
	+
	\norm{\Delta \bff{u}_n^\epsilon(t)}{\widetilde{\bb{H}}^{-1+\alpha}}^2 \right)
	\\
	&\quad
	+
	\epsilon^{\beta+2} \left(
	\norm{\pa_t\bff{u}_n^\epsilon(s)}{\widetilde{\bb{H}}^{-1+\alpha}}^2 +
	\norm{\Delta \pa_t\bff{u}_n^\epsilon(s)}{\widetilde{\bb{H}}^{-1+\alpha}}^2 \right)
	\\
	&\lesssim
	\epsilon\left(\norm{\bff{u}_n^\epsilon(t)}{\bb{L}^2}^2
	+
	\norm{\Delta \bff{u}_n^\epsilon(t)}{\bb{L}^2}^2 \right)
	\\
	&\quad
	+
	\epsilon^{\beta+2} 
	\left(\norm{\pa_t\bff{u}_n^\epsilon(s)}{\bb{L}^2}^2 +
	\norm{\Delta \pa_t\bff{u}_n^\epsilon(s)}{\bb{L}^2}^2 \right) 
	\lesssim 1,
\end{align*}
where in the last step we used Lemma~\ref{lem:ene est 1}, \ref{lem:ene est dtu}, and \ref{lem:ene est 3}, as required.
\end{proof}

The results in Lemmas~\ref{lem:ene est 1}, \ref{lem:ene est dtu}, \ref{lem:ene est 3}, \ref{lem:ene est 4} imply that the solution $\bff{u}_n^\epsilon$ exist on $[0,T]$ for any given $T>0$. They can be summarised as follows: if \eqref{equ:u0 eps con} and~\eqref{equ:u0n H2} hold for $s=1$, then we have the following estimates for any $\alpha$ satisfying~\eqref{equ:condition alpha}:
\begin{equation}\label{equ:une all}
	\begin{aligned}
		\norm{\bff{u}_n^\epsilon}{L^\infty_T(\bb{H}^1)}
		+
		\norm{\bff{u}_n^\epsilon}{L^2_T(\bb{H}^{1+\alpha})}
		+
		\sqrt{\epsilon} \norm{\bff{u}_n^\epsilon}{L^\infty_T(\bb{H}^{1+\alpha})}
		+
		\sqrt{\epsilon} \norm{\bff{u}_n^\epsilon}{L^\infty_T(\bb{H}^{1}_\Delta)}
		&\lesssim 1,
		\\
		\epsilon^{(\beta+1)/2} \norm{\pa_t\bff{u}_n^\epsilon}{L^\infty_T(\bb{H}^1)}
		+
		\epsilon^{(\beta+2)/2} \norm{\pa_t\bff{u}_n^\epsilon}{L^\infty_T(\bb{H}^{1+\alpha})}
		+
		\epsilon^{(\beta+2)/2} \norm{\pa_t\bff{u}_n^\epsilon}{L^\infty_T(\bb{H}^1_\Delta)}
		&\lesssim 1.
	\end{aligned}
\end{equation}
Further technical estimates are shown in Appendix~\ref{sec:further regular proof}, which yields the following result (see~\eqref{equ:une est local}): if \eqref{equ:u0 eps con} and~\eqref{equ:u0n H2} hold for $s=2$, then for any $\alpha$ satisfying~\eqref{equ:condition alpha}:
\begin{equation}\label{equ:une all app}
		\norm{\Delta \bff{u}_n^\epsilon}{L^\infty_{T^\ast}(\bb{L}^2)}
        +
        \sqrt{\epsilon} \norm{\Delta \bff{u}_n^\epsilon}{L^\infty_{T^\ast}(\bb{H}^1)}
		+
		\norm{\Delta \bff{u}_n^\epsilon}{L^2_{T^\ast}(\bb{H}^1)}
		\lesssim 1.
\end{equation}

We prove the existence of a global solution~$\bff{u}^\epsilon$ to~\eqref{equ:LLB eps pro} in the following proposition.

\begin{theorem}\label{the:ue sol}
	Let~$ \epsilon>0 $ and $T>0$ be fixed. Let $T^\ast$ be defined by~\eqref{equ:T ast}.
    \begin{enumerate}
		\renewcommand{\labelenumi}{\theenumi}
		\renewcommand{\theenumi}{{\rm (\roman{enumi})}}
		\item If~\eqref{equ:u0 eps con} and~\eqref{equ:u0n
			H2} hold for $s=1$, then there exists
		a global strong solution $\bff{u}^\epsilon$ to~\eqref{equ:LLB eps pro} in the sense of Definition~\ref{def:wea sol}. Furthermore, for any $\alpha$ satisfying~\eqref{equ:condition alpha}, there exists
		a positive constant~$c$ independent of~$\epsilon$
		satisfying
		\begin{equation}\label{equ:est ue unif 1}
			\begin{aligned}
				\norm{\bff{u}^\epsilon}{L^\infty_T(\bb{H}^1)}
				+
				\norm{\bff{u}^\epsilon}{L^2_T(\bb{H}^{1+\alpha})}
				+
				\sqrt{\epsilon} \norm{\bff{u}^\epsilon}{L^\infty_T(\bb{H}^{1+\alpha})}
				+
				\sqrt{\epsilon} \norm{\bff{u}^\epsilon}{L^\infty_T(\bb{H}^{1}_\Delta)}
				&\le c,
				\\
				\epsilon^{(\beta+1)/2} \norm{\pa_t\bff{u}^\epsilon}{L^\infty_T(\bb{H}^1)}
				+
				\epsilon^{(\beta+2)/2} \norm{\pa_t\bff{u}^\epsilon}{L^\infty_T(\bb{H}^{1+\alpha})}
				+
				\epsilon^{(\beta+2)/2} \norm{\pa_t\bff{u}^\epsilon}{L^\infty_T(\bb{H}^1_\Delta)}
				&\le c.
			\end{aligned}
		\end{equation}
        The constant $\beta$ is defined by~\eqref{equ:d ast}. Furthermore, $\bff{u}^\epsilon\in C^{0,\widetilde{\gamma}}([0,T];\bb{H}^{1+\alpha})$ for every $\widetilde{\gamma}\in [0,1)$.
        \item If~\eqref{equ:u0 eps con} and~\eqref{equ:u0n
			H2} hold for $s=2$, then for any $\alpha$ satisfying~\eqref{equ:condition alpha}, we have
    \begin{equation}\label{equ:ue all app}
		\norm{\Delta \bff{u}^\epsilon}{L^\infty_{T^\ast}(\bb{L}^2)}
        +
        \sqrt{\epsilon} \norm{\Delta \bff{u}^\epsilon}{L^\infty_{T^\ast}(\bb{H}^1)}
		+
		\norm{\Delta \bff{u}^\epsilon}{L^2_{T^\ast}(\bb{H}^1)}
		\lesssim 1.
\end{equation}
    \end{enumerate}
\end{theorem}

\begin{proof}
	Fix $ \epsilon > 0$.
	It follows from~\eqref{equ:une all} that there exist
	$\bff{u}^{\epsilon}\in W^{1,\infty}_{T}(\bb{H}^{1+\alpha})$ and a subsequence
	of~$\{\bff{u}_n^\epsilon\}_{n\in\bb{N}}$,
	which we do not relabel, satisfying
	\begin{equation}\label{equ:une wea con}
		\bff{u}^\epsilon_n\to\bff{u}^\epsilon
		\;\;\text{weakly}^\ast \text{ in $W^{1,\infty}_{T}(\bb{H}^{1+\alpha})$ as $n\to \infty$},
			\\
	\end{equation}
	where the convergence is not uniform with respect to~$ \epsilon$.
	Since the embeddings $\bb{H}^{1+\alpha}\hookrightarrow \bb{H}^1 \hookrightarrow \bb{L}^4$ are compact, it follows from the
	Aubin--Lions lemma and~\eqref{equ:une wea con} that 
	\begin{equation*}
		\lim_{n\to\infty} \norm{\bff{u}_n^\epsilon-\bff{u}^\epsilon}{L^2_T(\bb{L}^4)} 
		= 0
		\quad\text{and}\quad 
		\lim_{n\to\infty} \norm{\bff{u}_n^\epsilon-\bff{u}^\epsilon}{L^2_T(\bb{H}^1)} 
		= 0.
	\end{equation*}
	The rest of the proof follows the same argument as in Theorem~\ref{the:con u eps}. Estimates \eqref{equ:est ue unif 1} and \eqref{equ:ue all app} follow from \eqref{equ:une all} and \eqref{equ:une all app}, respectively.
\end{proof}

We conclude this section by showing the stability of the strong solution~$\bff{u}^\epsilon$ of \eqref{equ:LLB eps pro} with respect to initial data, which implies uniqueness.

\begin{theorem}\label{the:unique llbe}
	Let~$\bff{u}_0^\epsilon, \bff{v}_0^\epsilon \in \bb{H}^1_\Delta$.
	Let~$\bff{u}^\epsilon$ and~$\bff{v}^\epsilon$ be the global strong solutions
	associated with the initial data~$\bff{u}_0^\epsilon$
	and~$\bff{v}_0^\epsilon$, respectively, as conferred by
	Theorem~\ref{the:ue sol}. Then
	\[
	\norm{\bff{u}^\epsilon-\bff{v}^\epsilon}{L^\infty_T(\bb{H}^1)} 
	\lesssim
	\norm{\bff{u}_0^\epsilon-\bff{v}_0^\epsilon}{\bb{H}^1_\Delta}, 
	\]
	where the constant may depend on~$\epsilon$ and $T$. This implies
	uniqueness of the solution~$\bff{u}^\epsilon$.
\end{theorem}

\begin{proof}
	Let~$\bff{w}_0^\epsilon:=\bff{u}_0^\epsilon-\bff{v}_0^\epsilon$ 
	and~$\bff{w}^\epsilon:=\bff{u}^\epsilon-\bff{v}^\epsilon$.
	Then it follows from~\eqref{equ:weak sol} that, for any $\bff{\phi}\in
	\bb{H}^1$ and $t\in [0,T]$,
	\begin{align}\label{equ:diff w}
		&\inpro{\partial_t \bff{w}^\epsilon(t)}{\bff{\phi}}_{\bb{L}^2}
		+
		\epsilon \inpro{\nabla\partial_t \bff{w}^\epsilon(t)}{\nabla\bff{\phi}}_{\bb{L}^2}
		+
		\kappa_1 \inpro{\nabla\bff{w}^\epsilon(t)}{\nabla\bff{\phi}}_{\bb{L}^2}
		\nn \\
		&\quad
		+
		\gamma \inpro{\bff{w}^\epsilon(t)\times\nabla\bff{u}^\epsilon(t)}{\nabla\bff{\phi}}_{\bb{L}^2}
		+
		\gamma \inpro{\bff{v}^\epsilon(t)\times\nabla\bff{w}^\epsilon(t)}{\nabla\bff{\phi}}_{\bb{L}^2}
		\nn 
		\\
		&\quad
		+
		\kappa_2 \inpro{\bff{w}^\epsilon(t)}{\bff{\phi}}_{\bb{L}^2}
		+
		\kappa_2\mu \inpro{|\bff{u}^\epsilon(t)|^2\bff{w}^\epsilon(t) + (|\bff{u}^\epsilon(t)|^2-|\bff{v}^\epsilon(t)|^2)\bff{v}^\epsilon(t)}{\bff{\phi}}_{\bb{L}^2}
		= 
		0.
	\end{align}
	Putting $\bff{\phi} = 2\bff{w}^\epsilon$ in~\eqref{equ:diff w} and
	integrating we obtain, for any $t\in [0,T]$,
	\begin{align*}
		&\norm{\bff{w}^\epsilon(t)}{\bb{L}^2}^2
		+
		\epsilon \norm{\nabla\bff{w}^\epsilon(t)}{\bb{L}^2}^2
		+
		2\kappa_1\int_0^t \|\nabla\bff{w}^\epsilon(s)\|^2_{\bb{L}^2}\,\ds
		\\
		&\quad
		+
		2\kappa_2\int_0^t \|\bff{w}^\epsilon(s)\|^2_{\bb{L}^2}\,\ds
		+
		2\kappa_2\mu\int_0^t\||\bff{u}^\epsilon(s)|
		|\bff{w}^\epsilon(s)|\|^2_{\bb{L}^2}\,\ds
		\\
		&=
		\norm{\bff{w}_0^\epsilon}{\bb{L}^2}^2
		-
		2 \gamma
		\int_0^t 
		\inpro{\bff{w}^\epsilon(s)\times\nabla\bff{u}^\epsilon(s)}
		{\nabla\bff{w}^\epsilon(s)}_{\bb{L}^2}\,\ds
		-
		2 \kappa_2\mu
		\int_0^t 
		\inpro{(|\bff{u}^\epsilon(s)|^2-|\bff{v}^\epsilon(s)|^2)\bff{v}^\epsilon(s)}
		{\bff{w}^\epsilon(s)}_{\bb{L}^2} \ds.
	\end{align*}
	The last two terms on the right-hand side can be estimated as
	follows. For the middle term, by Lemma~\ref{lem:tec lem} we have
	\begin{align*}
		\left| \int_0^t \inpro{\bff{w}^\epsilon(s)\times\nabla\bff{u}^\epsilon(s)}{\nabla\bff{w}^\epsilon(s)}_{\bb{L}^2}\ds \right| 
		&\leq 
		\int_0^t \Phi(\bff{u}^\epsilon(s)) \norm{\bff{w}^\epsilon(s)}{\bb{L}^2}^2 \ds
		+
		\delta \int_0^t \norm{\nabla \bff{w}^\epsilon(s)}{\bb{L}^2}^2 \ds 
	\end{align*}
	for any $\delta>0$, where $\Phi$ is defined in~\eqref{equ:Phi}.
	For the last term, we have
	\begin{align}\label{equ:ue ve w}
		&\left| \int_0^t \inpro{(|\bff{u}^\epsilon(s)|^2-|\bff{v}^\epsilon(s)|^2)\bff{v}^\epsilon(s)}{\bff{w}^\epsilon(s)}_{\bb{L}^2} \ds \right| 
		\nn\\
		&\leq
		\int_0^t \norm{\bff{u}^\epsilon(s)+\bff{v}^\epsilon(s)}{\bb{L}^6}
		\norm{\bff{w}^\epsilon(s)}{\bb{L}^6}
		\norm{\bff{v}^\epsilon(s)}{\bb{L}^6}
		\norm{\bff{w}^\epsilon(s)}{\bb{L}^2} \ds 
		\nn \\
		&\leq
		\int_0^t \norm{\bff{u}^\epsilon(s)+\bff{v}^\epsilon(s)}{\bb{H}^1}^2
		\norm{\bff{v}^\epsilon(s)}{\bb{H}^1}^2
		\norm{\bff{w}^\epsilon(s)}{\bb{L}^2} \ds 
		+
		\delta \int_0^t \norm{\bff{w}^\epsilon(s)}{\bb{H}^1}^2 \ds 
		\nn \\
		&\leq
		c \int_0^t \norm{\bff{w}^\epsilon(s)}{\bb{L}^2}^2 \ds 
		+
		\delta \int_0^t \norm{\bff{w}^\epsilon(s)}{\bb{H}^1}^2 \ds 
	\end{align}
	for any $\delta>0$, where in the penultimate step we used Young's
	inequality and the Sobolev embedding $\bb{H}^1\hookrightarrow \bb{L}^6$, and
	in the last step we used Theorem~\ref{the:ue sol}. 
	
	Altogether, rearranging the terms and choosing $\delta>0$ sufficiently small, we infer that
	\begin{align*}
		\norm{\bff{w}^\epsilon(t)}{\bb{L}^2}^2
		+
		\epsilon \norm{\nabla\bff{w}^\epsilon(t)}{\bb{L}^2}^2
		\leq
		\norm{\bff{w}_0^\epsilon}{\bb{L}^2}^2
		+
		\int_0^t \Phi(\bff{u}^\epsilon(s)) \norm{\bff{w}^\epsilon(s)}{\bb{L}^2}^2 \ds.
	\end{align*}
	Note that~$\bff{u}^\epsilon\in L^\infty_T(\bb{H}^{1+\alpha})$ for all $\alpha$ satisfying~\eqref{equ:condition alpha} by
	Theorem~\ref{the:ue sol}. Thus, by the Sobolev embeddings $\bb{H}^{1+\alpha}\hookrightarrow \bb{W}^{1,2d}$ for $d=1,2$ and $\bb{H}^2\hookrightarrow \bb{W}^{1,6}$ for $d=3$, we have
	\begin{equation}\label{equ:Phi c}
		\int_0^t \Phi(\bff{u}^\epsilon(s))\,\ds \le c,
	\end{equation}
	where $c$ may depend on $\epsilon$.
	 The Gronwall inequality and~\eqref{equ:Phi c} {give the} required result.
\end{proof}

\subsection{Convergence of $\bff{u}^\epsilon$ to $\bff{u}$} \label{sec:conver}

Here, we show convergence of the strong solution of the $\epsilon$-LLBE to that of the LLBE as $\epsilon\to 0^+$.

\begin{theorem}\label{the:u eps con u}
	The following statements hold true.
	\begin{enumerate}
		\renewcommand{\labelenumi}{\theenumi}
		\renewcommand{\theenumi}{{\rm (\roman{enumi})}}
		\item \label{ite:u eps u i}
		Consider the case when~$d\in\{1,2\}$. Assume that~$\bff{u}_0\in\bb{H}^1_\Delta$
		and~$\bff{u}_0^\epsilon$ is chosen such that~\eqref{equ:u0 eps con} holds for $s=1$.
		Let $\bff{u}$ and $\bff{u}^\epsilon$ be global strong
		solutions of~\eqref{equ:LLB pro} and~\eqref{equ:LLB eps pro},
		respectively.  Then for any $t\in[0,T]$,
		\begin{equation}\label{equ:solus}
			\norm{\bff{u}^\epsilon(t)-\bff{u}(t)}{\bb{L}^2}^2
			+
			\int_0^t 
			\norm{\nabla\bff{u}^\epsilon(s) - \nabla\bff{u}(s)}{\bb{L}^2}^2 \ds
			\lesssim 
			\norm{\bff{u}_0^\epsilon-\bff{u}_0}{\bb{L}^2}^2
			+
			\epsilon^2.
		\end{equation}
		\item \label{ite:u eps u ii}
		If~$d\in\{1,2,3\}$, $\bff{u}_0\in \bb{H}^2_\Delta$, and~$\bff{u}_0^\epsilon$ is
		chosen to satisfy~\eqref{equ:u0 eps con} for $s=2$, then for any $t\in [0,T^\ast]$,
		\begin{align}\label{equ:sol u H1}
			\norm{\bff{u}^\epsilon(t)-\bff{u}(t)}{\bb{H}^1}^2
			+
			\int_0^t \norm{\Delta\bff{u}^\epsilon(s) - \Delta\bff{u}(s)}{\bb{L}^2}^2 \ds
			\lesssim 
			\norm{\bff{u}_0^\epsilon-\bff{u}_0}{\bb{H}^1}^2
			+
			\epsilon^2.
		\end{align}
	\end{enumerate}
	The constants in the inequalities are independent of $\epsilon$.
\end{theorem}
\begin{proof}
	Let $\bff{v} = \bff{u}^\epsilon-\bff{u}$ 
	and~$\bff{v}_0 = \bff{u}_0^\epsilon-\bff{u}_0$. We first prove~\ref{ite:u eps u i} for the case $d\in \{1,2\}$.
	By Definition~\ref{def:wea sol}, for any $\bff{\phi}\in \bb{H}^1$ and $t\in [0,T]$ we have
	\begin{align}\label{equ:vecv}
		&\inpro{\partial_t \bff{v}(t)}{\bff{\phi}}_{\bb{L}^2}
		+
		\kappa_1 \inpro{\nabla\bff{v}(t)}{\nabla\bff{\phi}}_{\bb{L}^2}
		+
		\gamma \inpro{\bff{v}(t)\times\nabla\bff{u}^\epsilon(t)}{\nabla\bff{\phi}}_{\bb{L}^2}
		+
		\gamma \inpro{\bff{u}(t)\times\nabla\bff{v}}{\nabla\bff{\phi}}_{\bb{L}^2}
		\nn 
		\\
		&\;
		+
		\kappa_2 \inpro{\bff{v}(t)}{\bff{\phi}}_{\bb{L}^2}
		+
		\kappa_2\mu \inpro{|\bff{u}^\epsilon(t)|^2\bff{v}(t) + (|\bff{u}^\epsilon(t)|^2-|\bff{u}(t)|^2)\bff{u}(t)}{\bff{\phi}}_{\bb{L}^2}
		= 
		-\epsilon \inpro{\nabla\partial_t \bff{u}^\epsilon(t)}{\nabla\bff{\phi}}_{\bb{L}^2}.
	\end{align}
	Putting $\bff{\phi} = \bff{v}$ in~\eqref{equ:vecv} and integrating yield, for any $t\in [0,T]$,
	\begin{align}\label{equ:vecv2}
		\frac12 \|\bff{v}(t)\|_{\bb{L}^2}^2
		&+\kappa_1\int_0^t \|\nabla\bff{v}(s)\|^2_{\bb{L}^2}\,\ds
		+
		\kappa_2\int_0^t \|\bff{v}(s)\|^2_{\bb{L}^2}\,\ds
		+
		\kappa_2\mu\int_0^t\||\bff{u}(s)| |\bff{v}(s)|\|^2_{\bb{L}^2}\,\ds
		\nn\\
		&\le
		\frac12 \|\bff{v}_0\|_{\bb{L}^2}^2
		+
		\epsilon\int_0^t 
		\Big|
		\inpro{\nabla\partial_t \bff{u}^\epsilon(s)}{\nabla\bff{v}(s)}_{\bb{L}^2}\,\ds
		\Big|
		+
		\gamma\int_0^t 
		\Big|
		\inpro{\bff{v}(s)\times\nabla\bff{u}^\epsilon(s)}{\nabla\bff{v}(s)}_{\bb{L}^2}\,\ds
		\Big|
		\nn \\
		&\quad+\kappa_2\mu
		\int_0^t 
		\Big|
		\inpro{(|\bff{u}^\epsilon(s)|^2-|\bff{u}(s)|^2)\bff{u}^\epsilon(s)}{\bff{v}(s)}_{\bb{L}^2} \,\ds
		\Big|.
	\end{align}
	We now estimate each term on the right hand side of~\eqref{equ:vecv2}. 
	The first term can be estimated by using Young's inequality as follows:
	\begin{equation}\label{equ:uni4}
		\epsilon \int_0^t 
		\Big|
		\inpro{\nabla\partial_t \bff{u}^\epsilon(s)}{\nabla\bff{v}(s)}_{\bb{L}^2}\,\ds
		\Big|
		\leq 
		\frac{\epsilon^2}{\kappa_1}\int_0^t \|\nabla\partial_t \bff{u}^\epsilon(s)\|_{\bb{L}^2}^2\,\ds
		+
		\frac{\kappa_1}{4}\int_0^t \|\nabla\bff{v}(s)\|^2_{\bb{L}^2}\,\ds.
	\end{equation}
	For the second term, by using Lemma~\ref{lem:tec lem} we obtain
	\begin{equation}\label{equ:v cross nab u}
		\int_0^t 
		\Big|
		\inpro{\bff{v}(s)\times\nabla\bff{u}^\epsilon(s)}{\nabla\bff{v}(s)}_{\bb{L}^2}\,\ds
		\Big|
		\leq 
		\int_0^t \Phi(\bff{u}^\epsilon(s)) \norm{\bff{v}(s)}{\bb{L}^2}^2 \ds
		+
		\frac{\kappa_1}{4} \int_0^t \norm{\nabla \bff{v}(s)}{\bb{L}^2}^2 \ds.
	\end{equation}
	For the last term, by the same argument as in~\eqref{equ:ue ve w} we have
	\begin{equation}\label{equ:u ue v}
		\int_0^t 
		\Big|
		\inpro{(|\bff{u}^\epsilon(s)|^2-|\bff{u}(s)|^2)\bff{u}(s)}{\bff{v}(s)}_{\bb{L}^2} \ds
		\Big|
		\leq
		c \int_0^t \norm{\bff{v}(s)}{\bb{L}^2}^2 \ds 
		+
		\frac14 \min\{\kappa_1,\kappa_2\} \int_0^t \norm{\bff{v}(s)}{\bb{H}^1}^2 \ds,
	\end{equation}
	where the constant $c$ is independent of $\epsilon$ by~\eqref{equ:est ue unif
		1}. Altogether, \eqref{equ:vecv2}, \eqref{equ:uni4}, \eqref{equ:v cross nab u},
	and~\eqref{equ:u ue v} yield
	\begin{equation*}
		\norm{\bff{v}(t)}{\bb{L}^2}^2 
		+ \int_0^t \norm{\nabla \bff{v}(s)}{\bb{L}^2}^2 \ds 
		\leq
		c \|\bff{v}_0\|_{\bb{L}^2}^2
		+
		c\epsilon^2 
		+ 
		c\int_0^t \big(1+ \Phi(\bff{u}^\epsilon(s))\big) \norm{\bff{v}(s)}{\bb{L}^2}^2 \ds,
	\end{equation*}
	where we also noted that $\norm{\nabla\partial_t
		\bff{u}^\epsilon}{L^2_T(\bb{L}^2)}$ can be bounded by a constant independent of
	$\epsilon$ by~\eqref{equ:une dt est local}. Applying Gronwall's inequality, we obtain
	\begin{align*}
		\norm{\bff{v}(t)}{\bb{L}^2}^2
		+ \int_0^t \norm{\nabla \bff{v}(s)}{\bb{L}^2}^2 \ds 
		&\leq
		c \left(
		\norm{\bff{u}_0^\epsilon-\bff{u}_0}{\bb{L}^2}^2
		+
		\epsilon^2 
		\right)
		\exp\left(c \int_0^t \big(1+\Phi(\bff{u}^\epsilon(s))\big) \ds\right)
		\\
		&\leq
		c \left(
		\norm{\bff{u}_0^\epsilon-\bff{u}_0}{\bb{L}^2}^2
		+
		\epsilon^2 
		\right),
	\end{align*}
	where we used~\eqref{equ:Phi une less 123} for $d=1,2$ in the last step. The constant~$c$ is
	independent of $\epsilon$, thus proving~\ref{ite:u eps u i}.
	
	Next, we prove~\ref{ite:u eps u ii}. Repeating the above argument, we obtain the estimate~\eqref{equ:solus} on $[0,T^\ast]$ for~$d=3$.
	Putting $\bff{\phi}=-\Delta \bff{\bff{v}}$ in~\eqref{equ:vecv},
	integrating, and noting~\eqref{equ:vec Gre 1} we obtain, for any $t\in [0,T]$,
	\begin{align}\label{equ:vec Delta v}
		&\frac12 \|\nabla \bff{v}(t)\|_{\bb{L}^2}^2
		+\kappa_1\int_0^t \|\Delta\bff{v}(s)\|^2_{\bb{L}^2}\,\ds
		+
		\kappa_2\int_0^t \norm{\nabla \bff{v}(s)}{\bb{L}^2}^2\,\ds
		\nn\\
		&\le
		\frac12 \norm{\nabla \bff{v}_0}{\bb{L}^2}^2
		+
		\epsilon\int_0^t 
		\Big|
		\inpro{\Delta \partial_t \bff{u}^\epsilon(s)}{\Delta \bff{v}(s)}_{\bb{L}^2} \Big|\,\ds
		+
		\gamma\int_0^t 
		\Big|
		\inpro{\bff{v}(s)\times\Delta\bff{u}^\epsilon(s)}{\Delta\bff{v}(s)}_{\bb{L}^2}\Big|\,\ds
		\nn \\
		&\quad
		+\kappa_2\mu
		\int_{0}^{t} 
		\Big|
		\inpro{|\bff{u}^\epsilon(s)|^2\bff{v}(s)}{\Delta \bff{v}(s)}_{\bb{L}^2}\Big| \,\ds
		\nn \\
		&\quad
		+\kappa_2\mu
		\int_0^t 
		\Big|
		\inpro{(|\bff{u}^\epsilon(s)|^2-|\bff{u}(s)|^2)\bff{u}^\epsilon(s)}{\Delta \bff{v}(s)}_{\bb{L}^2}\Big| \,\ds.
	\end{align}
	The second term on the right-hand side can be estimated by using Young's
	inequality and~\eqref{equ:une est local} as follows:
	\begin{align*}
		\epsilon \int_0^t 
		\Big|
		\inpro{\Delta\partial_t \bff{u}^\epsilon(s)}{\Delta \bff{v}(s)}_{\bb{L}^2} \Big| \,\ds
		&\leq 
		\frac{\epsilon^2}{\kappa_1}\int_0^t \|\Delta\partial_t \bff{u}^\epsilon(s)\|_{\bb{L}^2}^2\,\ds
		+
		\frac{\kappa_1}{5}\int_0^t \|\Delta\bff{v}(s)\|^2_{\bb{L}^2}\,\ds
		\nn\\
		&\leq 
		c \epsilon^2
		+
		\frac{\kappa_1}{5}\int_0^t \|\Delta\bff{v}(s)\|^2_{\bb{L}^2}\,\ds.
	\end{align*}
	For the third term, applying integration by parts and Young's inequality we obtain
	\begin{align*}
		\int_0^t 
		\Big|
		\inpro{\bff{v}(s)\times\Delta\bff{u}^\epsilon(s)}{\Delta\bff{v}(s)}_{\bb{L}^2}\Big| \,\ds
		&\leq 
		\int_0^t \norm{\bff{v}(s)}{\bb{L}^4} \norm{\Delta \bff{u}^\epsilon(s)}{\bb{L}^4} \norm{\Delta \bff{v}(s)}{\bb{L}^2} \ds 
		\nn \\
		&\leq 
		\int_0^t \norm{\Delta \bff{u}^\epsilon(s)}{\bb{H}^1}^2 \norm{ \bff{v}(s)}{\bb{H}^1}^2 \ds
		+
		\frac{\kappa_1}{5} \int_0^t \norm{\Delta \bff{v}(s)}{\bb{L}^2}^2 \ds,
	\end{align*}
	where in the last step we used the Sobolev embedding $\bb{H}^1\hookrightarrow
	\bb{L}^4$.
	For the second last term, it is easy to see that
	\[
	\int_{0}^{t} 
	\Big|
	\inpro{|\bff{u}^\epsilon(s)|^2\bff{v}(s)}{\Delta \bff{v}(s)}_{\bb{L}^2}\Big| \,\ds
	\leq
	c \int_0^t \norm{\bff{v}(s)}{\bb{H}^1}^2 \ds 
	+
	\frac{\kappa_1}{5} \int_0^t \norm{\Delta \bff{v}(s)}{\bb{L}^2}^2 \ds.
	\]
	For the last term, by the same argument as in~\eqref{equ:ue ve w} we have
	\begin{equation}\label{equ:u ue Delta v}
		\int_0^t 
		\Big|
		\inpro{(|\bff{u}^\epsilon(s)|^2-|\bff{u}(s)|^2)\bff{u}(s)}{\Delta \bff{v}(s)}_{\bb{L}^2} \ds
		\Big|
		\leq
		c \int_0^t \norm{\bff{v}(s)}{\bb{H}^1}^2 \ds 
		+
		\frac{\kappa_1}{5} \int_0^t \norm{\Delta \bff{v}(s)}{\bb{L}^2}^2 \ds,
	\end{equation}
	where the constant $c$ is independent of $\epsilon$ by~\eqref{equ:est ue unif 1}.
	Altogether, \eqref{equ:solus}, \eqref{equ:vec Delta v}--\eqref{equ:u ue	Delta v}
	yield
	\begin{equation*}
		\norm{\bff{v}(t)}{\bb{H}^1}^2 
		+ \int_0^t \norm{\Delta \bff{v}(s)}{\bb{L}^2}^2 \ds 
		\leq
		c \|\bff{v}_0\|_{\bb{H}^1}^2
		+
		c\epsilon^2 
		+ 
		c\int_0^t \big(1+ \norm{\Delta \bff{u}^\epsilon(s)}{\bb{H}^1}^2 \big)
		\norm{\bff{v}(s)}{\bb{H}^1}^2 \ds.
	\end{equation*}
	Applying Gronwall's inequality and using~\eqref{equ:une est local}, we
	obtain~\eqref{equ:sol u H1}, completing the proof.
\end{proof}

\section{A fully-discrete finite element approximation for the $\epsilon$-LLBE} \label{sec:fem}
Throughout this section, we assume that $\bff{u}_0^\epsilon \in \bb{H}^1_\Delta$. Under this assumption, Theorem~\ref{the:ue sol} guarantees the existence of a unique global strong solution $\bff{u}^\epsilon$ of the $\epsilon$-LLBE with regularity
\begin{equation}\label{equ:regularity llbe}
\bff{u}^\epsilon\in W^{1,\infty}_T(\bb{H}^{1+\alpha}),
\end{equation}
where $\alpha$ satisfies \eqref{equ:condition alpha}.

A linear finite element scheme for numerically solving the $\epsilon$-LLBE
can now be proposed in the following algorithm. To discretise in time, let $N\in\bb{N}$ and $k=T/N$. We
partition $[0,T]$ into $N$ uniform subintervals with nodes $t_n=kn$ for $n=0,1,2,\ldots,N$. 

\begin{algorithm}[Linear FEM for the $\epsilon$-LLBE]\label{alg:fem eps llbe}
{Let $h>0$ and $k>0$ be given.
\\
\textbf{Input}: Given $\bff{u}_h^{\epsilon, (0)}= P_h \bff{u}^\epsilon(0)\in \bb{V}_h$.
\\
\textbf{For} $j=1$ to $N$ \textbf{do}: Find $\bff{u}_h^{\epsilon, (j)}\in\bb{V}_h$ such that 
\begin{align}\label{equ:dis1}
\inpro{\mathrm{d}_t\bff{u}_h^{\epsilon, (j)}}{\bff{\phi}_h}_{\bb{L}^2}
&+
\epsilon\inpro{\nabla \mathrm{d}_t\bff{u}_h^{\epsilon, (j)}}{\nabla\bff{\phi}_h}_{\bb{L}^2}
+
\kappa_1 \inpro{\nabla\bff{u}_h^{\epsilon, (j)}}{\nabla\bff{\phi}_h}_{\bb{L}^2}
+
\gamma \inpro{\bff{u}_h^{\epsilon, (j-1)}\times \nabla\bff{u}_h^{\epsilon, (j)}}{\nabla\bff{\phi}_h}_{\bb{L}^2}
\nn\\
&
+
\kappa_2 \inpro{\bff{u}_h^{\epsilon, (j)}}{\bff{\phi}_h}_{\bb{L}^2}
+
\kappa_2 \mu
\inpro{|\bff{u}_h^{\epsilon, (j-1)}|^2 \bff{u}_h^{\epsilon, (j)}}{\bff{\phi}_h}_{\bb{L}^2}
=0,
\quad\forall \bff{\phi}_h\in\bb{V}_h,
\end{align}
\textbf{Output}: a sequence of finite element functions $\big\{\bff{u}_h^{\epsilon, (j)}\big\}_{j=1}^N$.}
\end{algorithm}

The above scheme is well-defined by the Lax--Milgram theorem.
Stability of the approximate solution in the $\ell^\infty(0,T;\bb{H}^1)$ norm, which holds unconditionally for an arbitrary number of iterations~$n$, is shown in the following lemma.

\begin{lemma}\label{lem: stability}
For any $n=1,2,\ldots,N$, there holds
\begin{align*}
\|\bff{u}_h^{\epsilon, (n)}\|^2_{\bb{L}^2}
+\epsilon\|\nabla\bff{u}_h^{\epsilon, (n)}\|^2_{\bb{L}^2}
+
2k\sum_{j=1}^n\|\nabla\bff{u}_h^{\epsilon, (j)}\|^2_{\bb{L}^2}
+
\sum_{j=1}^n\|\bff{u}_h^{\epsilon, (j)}-\bff{u}_h^{\epsilon, (j-1)}\|^2_{\bb{L}^2}
\leq 
\|\bff{u}_h^{\epsilon, (0)}\|^2_{\bb{L}^2}
+
\epsilon\|\nabla\bff{u}_h^{\epsilon, (0)}\|^2_{\bb{L}^2}.
\end{align*}
\end{lemma}
\begin{proof}
The proof is similar to that of Lemma~\ref{lem: stability2}.
\end{proof}

We are now ready to prove the error estimate for the scheme~\eqref{equ:dis1}. As before, for $j=1,\ldots,N$, letting $\bff{\rho}^{\epsilon, (j)}:= \bff{u}^\epsilon (t_j)- P_h \bff{u}^\epsilon (t_j)$ and $\bff{\theta}^{\epsilon, (j)}:= P_h \bff{u}^\epsilon (t_j)- \bff{u}_h^{\epsilon, (j)}$, we decompose the error into
\begin{align}\label{equ:error theta rho eps}
    \bff{u}^\epsilon (t_j)- \bff{u}_h^{\epsilon, (j)}= \bff{\rho}^{\epsilon, (j)} + \bff{\theta}^{\epsilon, (j)}.
\end{align}

{The following theorem shows an \emph{a priori} error estimate for the numerical scheme given by Algorithm~\ref{alg:fem eps llbe}. By using interpolation in the temporal variable, an approximate solution $\bff{u}^\epsilon_{h,k}$ can be defined from $\bff{u}_h^{\epsilon, (n)}$ such that $\norm{\bff{u}^\epsilon_{h,k}-\bff{u}^\epsilon}{L^\infty(0,T;\bb{H}^1)}$ satisfies \eqref{equ:err}.}

\begin{theorem}\label{the:err} 
Let $\bff{u}^\epsilon$ denote the strong solution of the $\epsilon$-LLBE with initial data $\bff{u}_0^\epsilon\in \bb{H}^1_\Delta$, {and let $\big\{\bff{u}_h^{\epsilon, (j)}\big\}_{j=1}^N$ be the sequence generated by Algorithm~\ref{alg:fem eps llbe}}. Assume that $k\ell_h\leqs 1$, where $\ell_h$ is defined in~\eqref{equ:ell h}.
Then for any $n=1,\ldots,N$, where $t_n\in [0,T]$, and $\alpha$ satisfying~\eqref{equ:condition alpha}, we have
\begin{equation}\label{equ:err}
 \norm{\bff{u}_h^{\epsilon, (n)}-\bff{u}^\epsilon(t_n)}{\bb{H}^1}
 \leq c e^{cT} (h^\alpha+k),
\end{equation}
where the positive constant~$c$ is independent of $n$, $h$, or $k$, but may depend on $\epsilon^{-1}$.
\end{theorem}

\begin{proof}
For any $ \bff{\phi}_h\in\bb{V}_h$, the exact solution $\bff{u}^\epsilon$
of~\eqref{equ:LLB eps pro} satisfies, for any~$\bff{\phi}_h\in\bb{V}_h$,
\begin{align}\label{equ:dis_ex}
&\inpro{\bff{u}^\epsilon(t_j)-\bff{u}^\epsilon(t_{j-1})}{\bff{\phi}_h}_{\bb{L}^2}
+
\epsilon\inpro{\nabla\bff{u}^\epsilon(t_j)-\nabla\bff{u}^\epsilon(t_{j-1})}{\nabla\bff{\phi}_h}_{\bb{L}^2}
+
\int_{t_{j-1}}^{t_j} \kappa_1 \inpro{\nabla\bff{u}^\epsilon(s)}{\nabla\bff{\phi}_h}_{\bb{L}^2}\ds
\nn\\
&\quad
+
\int_{t_{j-1}}^{t_j}
\gamma \inpro{\bff{u}^\epsilon(s)\times\nabla\bff{u}^\epsilon(s)}{\nabla\bff{\phi}_h}_{\bb{L}^2}\ds
+
\int_{t_{j-1}}^{t_j}
\kappa_2 \inpro{(1+\mu |\bff{u}^\epsilon(s)|^2)\bff{u}^\epsilon(s)}{\bff{\phi}_h}_{\bb{L}^2} \ds
=0.
\end{align}
Subtracting~\eqref{equ:dis_ex} from~\eqref{equ:dis1}, noting~\eqref{equ:error theta rho} and~\eqref{equ:orth proj}, we obtain
\begin{align}\label{equ:theta}
\inpro{\bff{\theta}^{\epsilon, (j)}-\bff{\theta}^{\epsilon, (j-1)}}{\bff{\phi}_h}_{\bb{L}^2}
&+
 \epsilon\inpro{\nabla\bff{\theta}^{\epsilon, (j)}-\nabla\bff{\theta}^{\epsilon, (j-1)}}{\nabla\bff{\phi}_h}_{\bb{L}^2}
 +
 \epsilon\inpro{\nabla\bff{\rho}^{\epsilon, (j)}-\nabla\bff{\rho}^{\epsilon, (j-1)}}{\nabla\bff{\phi}_h}_{\bb{L}^2}
 \nn\\
 &
 +
 k\kappa_1 \inpro{\nabla\bff{\theta}^{\epsilon, (j)}+\nabla \bff{\rho}^{\epsilon, (j)}}{\nabla\bff{\phi}_h}_{\bb{L}^2}
 +
 \kappa_1
 \int_{t_{j-1}}^{t_j} 
 \inpro{\nabla\bff{u}^\epsilon(t_j)- \nabla \bff{u}^\epsilon(s)}{\nabla\bff{\phi}_h}_{\bb{L}^2} \ds
 \nn\\
 &
 + k \kappa_2 \inpro{\bff{\theta}^{\epsilon, (j)}}{\bff{\phi}_h}_{\bb{L}^2}
 +
 \int_{t_{j-1}}^{t_j} \kappa_2 \inpro{\bff{u}^\epsilon(t_j)- \bff{u}^\epsilon(s)}{\bff{\phi}_h}_{\bb{L}^2} \ds 
 \nn\\
 &+
 \gamma H_1(\bff{u}_h^{\epsilon, (j)},\bff{u}^\epsilon,\bff{\phi}_h) 
 + 
 \kappa_2 \mu
 H_2(\bff{u}_h^{\epsilon, (j)},\bff{u}^\epsilon,\bff{\phi}_h) 
 \nn\\
 &
 +
 k \kappa_2 \mu
 \inpro{|\bff{u}_h^{\epsilon, (j-1)}|^2
 \big(\bff{\theta}^{\epsilon, (j)}+\bff{\rho}^{\epsilon, (j)}\big)}{\bff{\phi}_h}_{\bb{L}^2}
 =0,
\end{align}
where $H_1$ and~$H_2$ are defined in Lemma~\ref{lem:aux lem}.

To show \eqref{equ:err}, we will proceed inductively. Let $T>0$ be given. Let $\ell_h$ be as defined in~\eqref{equ:ell h} and suppose that $h$ and $k$ are sufficiently small such that
\begin{equation}\label{equ:timestep rest}
	k\ell_h \leq \frac{c_\infty \norm{\bff{u}^\epsilon}{L^\infty_T(\bb{L}^\infty)}}{c_{\mathrm{i}} e^{cT} \left(1+3c_\infty  \norm{\bff{u}^\epsilon}{L^\infty_T(\bb{L}^\infty)}\right)}
	\quad 
	\text{and}
	\quad
	h^\alpha \ell_h \leq \frac{c_\infty \norm{\bff{u}^\epsilon}{L^\infty_T(\bb{L}^\infty)}}{c_{\mathrm{i}} e^{cT} \left(1+3c_\infty  \norm{\bff{u}^\epsilon}{L^\infty_T(\bb{L}^\infty)}\right)},
\end{equation}
where $c_\infty$ and $c_{\mathrm{i}}$ are constants in~\eqref{equ:proj stab} for $p=\infty$ and in~\eqref{equ:inverse Vh}, respectively.
In other words, we impose a mild time-step restriction $k\ell_h \leqs 1$.
Let $M:=3 c_\infty \norm{\bff{u}^\epsilon}{L^\infty_T(\bb{L}^\infty)}$. For the inductive step, we assume that
\begin{equation}\label{equ:uhj infty}
	\norm{\bff{u}_h^{\epsilon, (j)}}{\bb{L}^\infty} \leq M, \quad \text{for } j=0,1,\ldots,n-1.
\end{equation}
Assuming this, we will show that
\begin{equation}\label{equ:induct theta j}
	\norm{\bff{\theta}^{\epsilon, (n)}}{\bb{H}^1} \leq e^{cT} (1+M) (h^\alpha+k), \quad \text{and}\quad \norm{\bff{u}_h^{\epsilon, (n)}}{\bb{L}^\infty} \leq M,
\end{equation}
where $c$ is a constant independent of $n$, $h$, and $k$.

To this end, putting~$\bff{\phi}_h=\bff{\theta}^{\epsilon, (j)}$ in~\eqref{equ:theta}, using the vector
identity~\eqref{equ:vec ab a}, and rearranging the terms, we have
\begin{align}\label{equ:theta sum}
	&\frac{1}{2} \left(\norm{\bff{\theta}^{\epsilon, (j)}}{\bb{L}^2}^2 - \norm{\bff{\theta}^{\epsilon, (j-1)}}{\bb{L}^2}^2\right)
	+
	\frac{1}{2} \norm{\bff{\theta}^{\epsilon, (j)}- \bff{\theta}^{\epsilon, (j-1)}}{\bb{L}^2}^2
	+
	\frac{\epsilon}{2} \left(\norm{\nabla \bff{\theta}^{\epsilon, (j)}}{\bb{L}^2}^2 - \norm{\nabla \bff{\theta}^{\epsilon, (j-1)}}{\bb{L}^2}^2\right)
	\nn \\
	&\quad
	+
	\frac{\epsilon}{2} \norm{\nabla \bff{\theta}^{\epsilon, (j)}- \nabla \bff{\theta}^{\epsilon, (j-1)}}{\bb{L}^2}^2
	+
	k\kappa_1 \norm{\nabla \bff{\theta}^{\epsilon, (j)}}{\bb{L}^2}^2
	+
	k\kappa_2 \norm{\bff{\theta}^{\epsilon, (j)}}{\bb{L}^2}^2
	+
	\kappa_2 \mu k \norm{\big|\bff{u}_h^{\epsilon, (j-1)}\big| \big|\bff{\theta}^{\epsilon, (j)}\big|}{\bb{L}^2}^2
	\nn \\
	&=
	-
	k\kappa_1 \inpro{\nabla\bff{\rho}^{\epsilon, (j)}}{\nabla\bff{\theta}^{\epsilon, (j)}}_{\bb{L}^2}
	-
	\epsilon\inpro{\nabla\bff{\rho}^{\epsilon, (j)}-\nabla\bff{\rho}^{\epsilon, (j-1)}}{\nabla\bff{\theta}^{\epsilon, (j)}}_{\bb{L}^2}
	\nn \\
	&\quad
	-
	\int_{t_{j-1}}^{t_j} \kappa_1 \inpro{\nabla\bff{u}^\epsilon(t_j)- \nabla \bff{u}^\epsilon(s)}{\nabla\bff{\theta}^{\epsilon, (j)}}_{\bb{L}^2} \ds
	-
	\int_{t_{j-1}}^{t_j} \kappa_2 \inpro{\bff{u}^\epsilon(t_j)- \bff{u}^\epsilon(s)}{\bff{\theta}^{\epsilon, (j)}}_{\bb{L}^2} \ds
	\nn \\
	&\quad
	-
	\gamma H_1(\bff{u}_h^{\epsilon, (j)},\bff{u}^\epsilon,\bff{\theta}^{\epsilon, (j)}) 
 	+ 
 	\kappa_2 \mu
	H_2(\bff{u}_h^{\epsilon, (j)},\bff{u}^\epsilon,\bff{\theta}^{\epsilon, (j)}) 
 	-
 	k \kappa_2 \mu
 	\inpro{|\bff{u}_h^{\epsilon, (j-1)}|^2 \bff{\rho}^{\epsilon, (j)}}{\bff{\theta}^{\epsilon, (j)}}_{\bb{L}^2}
	\nn \\
	&\le
	\norm{\nabla\bff{\theta}^{\epsilon, (j)}}{\bb{L}^2} 
	\Big(
		c k \norm{\nabla\bff{\rho}^{\epsilon, (j)}}{\bb{L}^2} 
		+
		\epsilon \norm{\nabla\bff{\rho}^{\epsilon, (j)}-\nabla\bff{\rho}^{\epsilon, (j-1)}}{\bb{L}^2} 
		+
		c \int_{t_{j-1}}^{t_j}
		\norm{\nabla\bff{u}^\epsilon(t_j)-\nabla\bff{u}^\epsilon(s)}{\bb{L}^2} \ds
	\Big)
	\nn \\
	&\quad
	+
	c\norm{\bff{\theta}^{\epsilon, (j)}}{\bb{L}^2} 
	\int_{t_{j-1}}^{t_j} 
	\norm{\bff{u}^\epsilon(t_j)-\bff{u}^\epsilon(s)}{\bb{L}^2} \ds
	\nn \\
	&\quad
	+
	c
	\abs{H_1(\bff{u}_h^{\epsilon, (j)},\bff{u}^\epsilon,\bff{\theta}^{\epsilon, (j)})}
 	+ 
	c
	\abs{ H_2(\bff{u}_h^{\epsilon, (j)},\bff{u}^\epsilon,\bff{\theta}^{\epsilon, (j)})}
	+
	c
	k \norm{\bff{\theta}^{\epsilon, (j)}}{\bb{L}^6}
	\norm{\bff{u}_h^{\epsilon, (j-1)}}{\bb{L}^6}^2
	\norm{\bff{\rho}^{\epsilon, (j)}}{\bb{L}^2}
	\nn\\
	&=:
	R_1 + \cdots + R_5.
\end{align}
We now estimate each term on the right-hand side. Let $\delta>0$ be given and let $\alpha$ satisfies~\eqref{equ:condition alpha}. For
the first term, by~\eqref{equ:uj uj1} and Young's inequality (noting the regularity~\eqref{equ:regularity llbe}),
\begin{align}\label{equ:llbe R1}
	R_1
	&=
	\norm{\nabla\bff{\theta}^{\epsilon, (j)}}{\bb{L}^2} 
	\Big(
		c k \norm{\nabla\bff{\rho}^{\epsilon, (j)}}{\bb{L}^2} 
		+
		\epsilon \norm{\nabla\bff{\rho}^{\epsilon, (j)}-\nabla\bff{\rho}^{\epsilon, (j-1)}}{\bb{L}^2} 
		+
		c \int_{t_{j-1}}^{t_j}
		\norm{\nabla\bff{u}^\epsilon(t_j)-\nabla\bff{u}^\epsilon(s)}{\bb{L}^2} \ds
	\Big)
	\nn\\
	&\leq
	\norm{\nabla\bff{\theta}^{\epsilon, (j)}}{\bb{L}^2} 
	\Big(
		c k \norm{\nabla\bff{\rho}^{\epsilon, (j)}}{\bb{L}^2} 
		+
		\epsilon k \norm{\nabla \bff{\rho}^{\epsilon}}{W^{1,\infty}_T(\bb{L}^2)} 
		+
		c k^2 
		\norm{\nabla \bff{u}^\epsilon}{W^{1,\infty}_T(\bb{L}^2)} 
	\Big)
	\nn\\
	&\leq
	\delta k \norm{\nabla\bff{\theta}^{\epsilon, (j)}}{\bb{L}^2}^2
	+
	ck \norm{\nabla \bff{\rho}^{\epsilon, (j)}}{\bb{L}^2}^2
	+
	c k \norm{\nabla\bff{\rho}^{\epsilon}}{W^{1,\infty}_T(\bb{L}^2)}^2
	+ 
	ck^3 \norm{\nabla\bff{u}^\epsilon}{W^{1,\infty}_T(\bb{L}^2)}^2
	\nn\\
	&\leq
    ck(h^{2\alpha}+k^2)
    +
	\delta k \norm{\nabla\bff{\theta}^{\epsilon, (j)}}{\bb{L}^2}^2,
\end{align}
where in the last step we used~\eqref{equ:rho rhot}.
For the second term, by H\"older's and Young's inequalities, and \eqref{equ:uj uj1}, we have
\begin{align*}
	R_2
	&=
	c \norm{\bff{\theta}^{\epsilon, (j)}}{\bb{L}^2} 
	\int_{t_{j-1}}^{t_j} 
	\norm{\bff{u}^\epsilon(t_j)-\bff{u}^\epsilon(s)}{\bb{L}^2} \ds
	\\
	&\leq
	c k^2
	\norm{\bff{\theta}^{\epsilon, (j)}}{\bb{L}^2} 
	\norm{\bff{u}^\epsilon}{W^{1,\infty}_T(\bb{L}^2)}
	\le
	\delta k \norm{\bff{\theta}^{\epsilon, (j)}}{\bb{L}^2}^2
	+
	ck^3.
\end{align*}
For the term $R_3$, we invoke Lemma~\ref{lem:aux lem} (with~$\bff{v}^{\epsilon, (j)}=\bff{u}_h^{\epsilon, (j)}$,
$\bff{w}=\bff{u}^\epsilon$, $\bff{\eta}^{\epsilon, (j)}=\bff{\rho}^{\epsilon, (j)}$,
and~$\bff{\mu}^{\epsilon, (j)}=\bff{\theta}^{\epsilon, (j)}$), use the fact that $\bff{u}^\epsilon$ is a strong solution, and apply~\eqref{equ:rho rhot} to obtain 
\begin{align}\label{equ:R3}
	R_3
	=
	c\, \big| H_1(\bff{u}_h^{\epsilon, (j)},\bff{u}^\epsilon,\bff{\theta}^{\epsilon, (j)}) \big|
	&\le
	c(1+M) k (h^{2\alpha}+k^2) 
	+ c k \norm{\bff{\theta}^{\epsilon, (j-1)}}{\bb{L}^2}^2
	+ c k^3 
	\nn\\
	&\quad
	+ \delta k \norm{\nabla \bff{\theta}^{\epsilon, (j-1)}}{\bb{L}^2}^2
	+ 4\delta k \norm{\nabla \bff{\theta}^{\epsilon, (j)}}{\bb{L}^2}^2,
\end{align}
where we used~\eqref{equ:uhj infty} to
obtain the boundedness of $\bff{u}_h^{\epsilon, (j)}$.
Similarly,
\begin{align*} 
	R_4
	&=
	c\, \big| H_2(\bff{u}_h^{\epsilon, (j)},\bff{u}^\epsilon,\bff{\theta}^{\epsilon, (j)}) \big|
\\
	&\leq
	ck h^{2(1+\alpha)} 
	+
	c k \norm{\bff{\theta}^{\epsilon, (j-1)}}{\bb{L}^2}^2
	+
	ck^3
	+
	2 \delta k \norm{\bff{\theta}^{\epsilon, (j)}}{\bb{H}^1}^2
		+
		\delta k
		\norm{|\bff{u}_h^{\epsilon, (j-1)}||\bff{\theta}^{\epsilon, (j)}|}{\bb{L}^2}^2,
\end{align*}
where we also used Lemma~\ref{lem: stability}. The constant $c$ may depend on $\epsilon^{-1}$.
Finally,
\begin{align*} 
	R_5
	&=
	c
	k \norm{\bff{\theta}^{\epsilon, (j)}}{\bb{L}^6}
	\norm{\bff{u}_h^{\epsilon, (j-1)}}{\bb{L}^6}^2
	\norm{\bff{\rho}^{\epsilon, (j)}}{\bb{L}^2}
	\\
    &\le
	ck 
	\norm{\bff{\rho}^{\epsilon, (j)}}{\bb{L}^2}^2
	+
	\delta k \norm{\bff{\theta}^{\epsilon, (j)}}{\bb{H}^1}^2
	\leq
	ck h^{2(1+\alpha)}
	+
	\delta k \norm{\bff{\theta}^{\epsilon, (j)}}{\bb{H}^1}^2.
\end{align*}
Altogether, we derive from~\eqref{equ:theta sum} that
\begin{align}\label{equ:the 33}
	&\frac{1}{2} 
	\Big(
	\norm{\bff{\theta}^{\epsilon, (j)}}{\bb{L}^2}^2 - \norm{\bff{\theta}^{\epsilon, (j-1)}}{\bb{L}^2}^2
	\Big)
	+
	\frac{1}{2} 
	\norm{\bff{\theta}^{\epsilon, (j)}- \bff{\theta}^{\epsilon, (j-1)}}{\bb{L}^2}^2
	\nn \\
	&\quad
	+
	\frac{\epsilon}{2} 
	\Big(
	\norm{\nabla \bff{\theta}^{\epsilon, (j)}}{\bb{L}^2}^2 - \norm{\nabla \bff{\theta}^{\epsilon, (j-1)}}{\bb{L}^2}^2
	\Big)
	+
	\frac{\epsilon}{2} \norm{\nabla \bff{\theta}^{\epsilon, (j)}- \nabla \bff{\theta}^{\epsilon, (j-1)}}{\bb{L}^2}^2
	\nn\\
	&\quad
	+
	k\kappa_1 \norm{\nabla \bff{\theta}^{\epsilon, (j)}}{\bb{L}^2}^2
	+
	k\kappa_2 \norm{\bff{\theta}^{\epsilon, (j)}}{\bb{L}^2}^2
	+
	\kappa_2 \mu k \norm{\big|\bff{u}_h^{\epsilon, (j-1)}\big| \big|\bff{\theta}^{\epsilon, (j)}\big|}{\bb{L}^2}^2
	\nn\\
	&\leq
	c(1+M) k(h^{2\alpha}+k^2) 
	+ 
	ck \norm{\bff{\theta}^{\epsilon, (j-1)}}{\bb{L}^2}^2
	+
	\delta k \norm{\nabla \bff{\theta}^{\epsilon, (j-1)}}{\bb{L}^2}^2
	\nn\\
	&\quad
	+
	8 \delta k \norm{\nabla \bff{\theta}^{\epsilon, (j)}}{\bb{L}^2}^2
	+
	4 \delta k \norm{\bff{\theta}^{\epsilon, (j)}}{\bb{L}^2}^2
	+
	\delta k
	\norm{|\bff{u}_h^{\epsilon, (j-1)}||\bff{\theta}^{\epsilon, (j)}|}{\bb{L}^2}^2.
\end{align}
We now choose $\delta=\frac{1}{10} \min\{\kappa_1,\kappa_2,\kappa_2 \mu\}$ to absorb the last four terms on the right-hand side to the last three terms on the left-hand side. In this manner we obtain 
\[
	\norm{\bff{\theta}^{\epsilon, (j)}}{\bb{L}^2}^2 
	- 
	\norm{\bff{\theta}^{\epsilon, (j-1)}}{\bb{L}^2}^2
	+
	(\epsilon+\kappa_1) \Big(\norm{\nabla\bff{\theta}^{\epsilon, (j)}}{\bb{L}^2}^2 
	- 
	\norm{\nabla\bff{\theta}^{\epsilon, (j-1)}}{\bb{L}^2}^2 \Big)
	\le
	c(1+M)k(h^{2\alpha}+k^2) 
	+ 
	ck \norm{\bff{\theta}^{\epsilon, (j-1)}}{\bb{L}^2}^2.
\]
Summing this over $j=1,\ldots, n$, we deduce that
\[
	\norm{\bff{\theta}^{\epsilon, (n)}}{\bb{H}^1}^2
	\le
	c(1+M)(h^{2\alpha}+k^2) + ck \sum_{j=0}^{n-1} \norm{\bff{\theta}^{\epsilon, (j)}}{\bb{L}^2}^2,
\]
where $c$ possibly depends on $\epsilon^{-1}$, but is independent of $n$, $h$, and $k$.
Applying the discrete version of the Gronwall inequality~\citep[Lemma~1]{Sug69}, we
deduce that
\begin{equation*}
	\norm{\bff{\theta}^{\epsilon, (n)}}{\bb{H}^1} \leq e^{cT}(1+M)(h^{\alpha}+k),
\end{equation*}
which is the first inequality in~\eqref{equ:induct theta j}. This implies by~\eqref{equ:inverse Vh} and~\eqref{equ:proj stab},
\begin{align*}
	\norm{\bff{u}_h^{\epsilon, (n)}}{\bb{L}^\infty}
	&\leq
	\norm{\bff{u}_h^{\epsilon, (n)}- P_h\bff{u}^\epsilon(t_n)}{\bb{L}^\infty}
	+
	\norm{P_h \bff{u}^\epsilon(t_n)}{\bb{L}^\infty}
	\\
	&\leq
	c_{\mathrm{i}} \ell_h \norm{\bff{\theta}^{\epsilon, (n)}}{\bb{H}^1} + c_\infty \norm{\bff{u}^\epsilon}{L^\infty_T(\bb{L}^\infty)}
	\\
	&\leq
	c_{\mathrm{i}} e^{cT}(1+M) \ell_h (h^\alpha+k)+ c_\infty \norm{\bff{u}^\epsilon}{L^\infty_T(\bb{L}^\infty)}
	\leq
	M,
\end{align*}
where in the last step we used~\eqref{equ:timestep rest}. This proves~\eqref{equ:induct theta j}, thus completing the induction argument.
Using~\eqref{equ:rho rhot}
and the triangle inequality, we obtain the first inequality in~\eqref{equ:induct theta j}. 
\end{proof}

As stated in the following theorem, it is possible to prove an unconditional \emph{a priori} estimate in convex domains, assuming a higher regularity condition on the exact solution. Such regularity can be assured for initial data $\bff{u}_0^\epsilon\in \bb{H}^2_\Delta$ in domains where $\bb{W}^{2,3}$-elliptic regularity result holds, for instance in rectangular parallelepipeds~\citep{Dau92}, or more generally in convex polyhedral domains satisfying certain interior angle restrictions (\citep[Chapter~4]{Gri11}.

{
\begin{theorem}\label{rem:error}
Suppose that $\mathscr{D}$ is a \emph{convex} polytopal domain. Assume that $\bff{u}^\epsilon$ is the strong solution of the $\epsilon$-LLBE with higher regularity:
\begin{align*}
	\bff{u}^\epsilon \in L^\infty_T(\bb{W}^{2,3}) \cap W^{1,\infty}_T(\bb{H}^2).
\end{align*}
Let $\big\{\bff{u}_h^{\epsilon, (j)}\big\}_{j=1}^N$ be the sequence generated by Algorithm~\ref{alg:fem eps llbe}.
Then for any $n=1,\ldots,N$, where $t_n\in [0,T]$, we have, unconditionally,
\begin{equation}\label{equ:err convex}
 \norm{\bff{u}_h^{\epsilon, (n)}-\bff{u}^\epsilon(t_n)}{\bb{H}^1}
 \leq c e^{cT} (h+k),
\end{equation}
where the positive constant~$c$ is independent of $n$, $h$, or $k$, but may depend on $\epsilon^{-1}$.
\end{theorem}

\begin{proof}
The proof proceeds in the same manner as that of Theorem~\ref{the:err}. However, instead of using \eqref{equ:H1 vj} to obtain \eqref{equ:R3}, we now apply \eqref{equ:H1 new} to infer that
\begin{align*}
    R_3 &\leq
    ck \norm{\nabla \bff{\rho}^{\epsilon, (j)}}{\bb{L}^3}^2 + ck \left(\norm{\bff{\rho}^{\epsilon, (j-1)}}{\bb{H}^1}^2 + \norm{\bff{\theta}^{\epsilon, (j-1)}}{\bb{L}^2}^2 \right) \norm{\bff{u}^\epsilon}{L^\infty_T(\bb{H}^2)}^2
    +
    ck^3 
    \\
    &\quad
    + \delta k \norm{\nabla \bff{\theta}^{\epsilon, (j-1)}}{\bb{L}^2}^2
	+ 4\delta k \norm{\nabla \bff{\theta}^{\epsilon, (j)}}{\bb{L}^2}^2
    \\
    &\leq
    ck(h^2+k^2) + ck \norm{\bff{\theta}^{\epsilon, (j-1)}}{\bb{L}^2}^2
    + \delta k \norm{\nabla \bff{\theta}^{\epsilon, (j-1)}}{\bb{L}^2}^2
	+ 4\delta k \norm{\nabla \bff{\theta}^{\epsilon, (j)}}{\bb{L}^2}^2,
\end{align*}
where in the last step we have used the estimate~\citep{CroTho87}:
\begin{align*}
	\norm{\nabla \bff{\rho}^{\epsilon, (j)}}{\bb{L}^3} \leq ch \norm{\bff{u}^\epsilon}{L^\infty_T(\bb{W}^{2,3})}.
\end{align*}
The rest of the proof proceeds in the same manner as before, leading to \eqref{equ:err convex}.
\end{proof}
}

{
\begin{remark}
As a result of Theorem~\ref{the:u eps con u} and either Theorem~\ref{the:err} or Theorem~\ref{rem:error}, an approximate solution $\bff{u}_{h,k}^\epsilon$ can be defined such that, for $\epsilon\in (0,1)$,
\begin{align*}
    \norm{\bff{u}-\bff{u}_{h,k}^\epsilon}{L^\infty_T(\bb{H}^1)}
    &\leq
    \norm{\bff{u}- \bff{u}^\epsilon}{L^\infty_T(\bb{H}^1)} + 
    \norm{\bff{u}^\epsilon - \bff{u}_{h,k}^\epsilon}{L^\infty_T(\bb{H}^1)}
    \\
    &\leq
    \norm{\bff{u}_0- \bff{u}_0^\epsilon}{\bb{H}^1} + c\epsilon + c(\epsilon) e^{c(\epsilon)T} (h+k),
\end{align*}
where $c(\epsilon)$ is a constant depending on $\epsilon^{-1}$.
Therefore, convergence of the numerical scheme is to be understood as
\[
\lim_{\epsilon\to 0^+} \lim_{h,k\to 0^+} \norm{\bff{u}-\bff{u}_{h,k}^\epsilon}{L^\infty_T(\bb{H}^1)} = 0.
\]
\end{remark}
}

\section{Uniform-in-time error estimates for the FEM approximation of the $\epsilon$-LLBE}
\label{sec:longtime}

The result in Theorem~\ref{the:err} suggests that the
approximation error grows exponentially with the {final time $T$}. As such,
estimate~\eqref{equ:err} will not be useful for assessing the approximation
error in the long run.
The aim of this section is to better approximate the long-run trajectory of the
solution and derive error estimates for the approximation of the $\epsilon$-LLBE
that are uniform in time. Similar results are studied for the Navier--Stokes
equation~\citep{HeyRan86}, a semilinear parabolic equation~\citep{Lar89}, and the
parabolic $p$-Laplacian equation~\citep{Ju00}, to name a few. To this end, we
need to derive decay estimates on the solution $\bff{u}^\epsilon$ and its approximation $\bff{u}_h^{\epsilon, (n)}$ for time $t\in [\hat{t},\infty)$. Here,
$\hat{t}$ denotes a sufficiently large time, which depends on the coefficients
of the equation.

Since we are interested in small values of $\epsilon$, we also assume that
$\epsilon<\kappa_1/\kappa_2$ for simplicity of presentation. Similar decay
estimate would still hold for bigger values of $\epsilon$. We assume in this section that a strong solution corresponding to an initial data $\bff{u}_0^\epsilon \in \bb{H}^2_\Delta$ exists, and thus instead of working with the
Faedo--Galerkin solution~$\bff{u}_n^\epsilon$, we can work directly with
$\bff{u}^\epsilon$.

\begin{lemma}\label{lem:dec}
	Suppose $\epsilon<\kappa_1/\kappa_2$ and let $\alpha$ satisfies~\eqref{equ:condition alpha}.
	\begin{enumerate}
	\renewcommand{\labelenumi}{\theenumi}
	\renewcommand{\theenumi}{{\rm (\roman{enumi})}}
		\item 
			For every $t\in [0,\infty)$, 
			\begin{align}
				\label{equ:decay u}
				\norm{\bff{u}^\epsilon(t)}{\bb{H}^1}^2
				+
				\epsilon \norm{\Delta\bff{u}^\epsilon(t)}{\bb{L}^2}^2
				&\leq
				c e^{-2\kappa_2 t},
				\\
				\label{equ:u H alpha decay}
				\epsilon\norm{\bff{u}^\epsilon(t)}{\bb{H}^{1+\alpha}}^2
				&\leq
				ce^{-2\kappa_2 t},
			\end{align}
			where the constant $c$ depends on
			$\norm{\bff{u}_0^\epsilon}{\bb{H}^1_\Delta}$, but is independent
			of~$\epsilon$ and~$t$.
		\item 
			There exists $\hat{t}$ depending on the coefficients
			in~\eqref{equ:LLB eps} such that for $t\in
			[\hat{t},\infty)$,
			\begin{equation}\label{equ:decay dt u}
				\norm{\partial_t \bff{u}^\epsilon(t)}{\bb{L}^2}^2
				+
				\epsilon \norm{\nabla \partial_t \bff{u}^\epsilon(t)}{\bb{L}^2}^2
				\leq
				ce^{-\kappa_2 t},
			\end{equation}
			where $c$ depends on $\norm{\bff{u}_0^\epsilon}{\bb{H}^1_\Delta}$ and
			$\hat{t}$, but is independent of~$t$.

		\item 
			Furthermore, if~$\bff{u}_0^\epsilon\in\bb{H}^2_\Delta$, then there
			exists $\hat{t}$ depending on the coefficients of the
			equation such that for $t\in [\hat{t},\infty)$,
			\begin{align}
				\label{equ:decay delta u}
				\norm{\Delta \bff{u}^\epsilon(t)}{\bb{L}^2}^2
				+
				\epsilon \norm{\nabla \Delta \bff{u}^\epsilon(t)}{\bb{L}^2}^2
				&\leq
				ce^{-\kappa_2 t},
				\\
				\label{equ:decay delta dt u}
				\epsilon \norm{\Delta \partial_t \bff{u}^\epsilon(t)}{\bb{L}^2}^2
				+
				\epsilon \norm{\pa_t \bff{u}^\epsilon (t)}{\bb{H}^{1+\alpha}}^2
				&\leq
				ce^{-\kappa_2 t},
			\end{align}
			where $c$ depends on $\norm{\bff{u}_0}{\bb{H}^2_\Delta}$ and
			$\hat{t}$, but is
			independent of $t$.
	\end{enumerate}

\end{lemma}
\begin{proof}
	Define
	\begin{alignat*}{2}
		\Psi_0(t)
		&:= 
		\norm{\bff{u}^\epsilon(t)}{\bb{L}^2}^2 
		+
		\epsilon \norm{\nabla \bff{u}^\epsilon(t)}{\bb{L}^2}^2,
		\qquad &
		\Psi_1(t)
		&:= 
		\norm{\nabla \bff{u}^\epsilon(t)}{\bb{L}^2}^2 
		+
		\epsilon \norm{\Delta \bff{u}^\epsilon(t)}{\bb{L}^2}^2,
		\\
		\Psi_2(t)
		&:= 
		\norm{\pa_t\bff{u}^\epsilon(t)}{\bb{L}^2}^2
		+
		\epsilon \norm{\nabla\pa_t\bff{u}^\epsilon(t)}{\bb{L}^2}^2,
		\qquad &
		\Psi_3(t)
		&:= 
		\norm{\Delta\bff{u}^\epsilon(t)}{\bb{L}^2}^2
		+
		\epsilon \norm{\nabla\Delta\bff{u}^\epsilon(t)}{\bb{L}^2}^2.
	\end{alignat*}
	\noindent
	\textbf{Proof of (i)}:
	Repeating the arguments in the proof of Lemma~\ref{lem:ene est 1},
	considering~\eqref{equ:weak sol} instead of~\eqref{equ:LLB eps wea equ},
	we obtain exactly the same equations as~\eqref{equ:un L2 H1}
	and~\eqref{equ:un H1} with~$\bff{u}_n^\epsilon$ replaced
	by~$\bff{u}^\epsilon$. Consequently,
	\begin{equation}\label{equ:Psi0 Psi1}
		\ddt \Psi_0(t)
		+
		2\kappa_2 \Psi_0(t)
		\leq 0
		\quad\text{and}\quad 
		\ddt \Psi_1(t)
		+
		2\kappa_2 \Psi_1(t)
		\leq 0,
	\end{equation}
	resulting in
	\begin{align*}
		\Psi_0(t)\leq \Psi_0(0)\cdot e^{-2\kappa_2 t}
		\quad\text{and}\quad 
		\Psi_1(t)\leq \Psi_1(0)\cdot e^{-2\kappa_2 t}.
	\end{align*}
	Estimate~\eqref{equ:decay u} then follows. Estimate~\eqref{equ:u H alpha decay} follows by the same argument as in~\eqref{equ:Lap H1s est}.

	\medskip
	\noindent
	\textbf{Proof of (ii)}: Repeating the arguments in the proof of 
	Lemma~\ref{lem:ene est 2}, we obtain analogously to~\eqref{equ:dt un nab dt}
	\begin{align*}
		\ddt \Psi_2(t)
		&+
		2 \kappa_2 \Psi_2(t)
		+ 
		(2\kappa_1-2\kappa_2 \epsilon) 
		\norm{\nabla \pa_t\bff{u}^\epsilon(t)}{\bb{L}^2}^2
		\nn \\
		&\leq
		2\gamma
		\norm{\partial_t \bff{u}^\epsilon(t)}{\bb{L}^4} 
		\norm{\nabla \bff{u}^\epsilon(t)}{\bb{L}^4} 
		\norm{\nabla \partial_t \bff{u}^\epsilon(t)}{\bb{L}^2}
		\nn \\
		&\leq
		\frac12 \kappa_2 \epsilon 
		\norm{\nabla \partial_t \bff{u}^\epsilon(t)}{\bb{L}^2}^2
		+
		\frac{2\gamma^2}{\kappa_2 \epsilon} 
		\norm{\nabla \bff{u}^\epsilon (t)}{\bb{L}^4}^2 
		\norm{\partial_t \bff{u}^\epsilon(t)}{\bb{L}^4}^2
		\nn \\
		&\leq
		\frac12 \kappa_2 \epsilon 
		\norm{\nabla \partial_t \bff{u}^\epsilon(t)}{\bb{L}^2}^2
		+
		\frac{2\gamma^2}{\kappa_2 \epsilon^2} 
		\epsilon
		\norm{\bff{u}^\epsilon(t)}{\bb{H}^{1+\alpha}}^2 
		\norm{\partial_t \bff{u}^\epsilon(t)}{\bb{L}^4}^2
		\nn \\
		&\leq
		\frac12 \kappa_2 \epsilon 
		\norm{\nabla \partial_t \bff{u}^\epsilon(t)}{\bb{L}^2}^2
		+
		\frac{2c_0 \gamma^2}{\kappa_2 \epsilon^2} e^{-2\kappa_2 t} 
		\norm{\partial_t \bff{u}^\epsilon(t)}{\bb{L}^4}^2,
	\end{align*}
	where in the penultimate step we used the
	embedding~$\bb{H}^{1+\alpha} \hookrightarrow \bb{W}^{1,4}$ and in the last step we
	used~\eqref{equ:u H alpha decay} with the constant~$c$ distinguished by~$c_0$.
	It follows from the Gagliardo--Nirenberg inequality and Young's
	inequality that
	\[
		\norm{\partial_t \bff{u}^\epsilon(t)}{\bb{L}^4}
		\le
		c
		\norm{\partial_t \bff{u}^\epsilon(t)}{\bb{L}^2}^{1-d/4}
		\norm{\nabla \partial_t \bff{u}^\epsilon(t)}{\bb{L}^2}^{d/4} 
		\le
		c\norm{\partial_t \bff{u}^\epsilon(t)}{\bb{L}^2}^2 
		+ 
		\frac{\kappa_2 \epsilon^2}{2c_0 \gamma^2} 
		\norm{\nabla \partial_t \bff{u}^\epsilon(t)}{\bb{L}^2}^2.
	\]
	Therefore,
	\begin{align*} 
		\ddt \Psi_2(t)
		+
		2 \kappa_2 \Psi_2(t)
		&\le
		\frac12 \kappa_2 \epsilon 
		\norm{\nabla \partial_t \bff{u}^\epsilon(t)}{\bb{L}^2}^2
		+
		\frac{2cc_0 \gamma^2}{\kappa_2 \epsilon^2} 
		e^{-2\kappa_2 t} 
		\norm{\partial_t \bff{u}^\epsilon(t)}{\bb{L}^2}^2
		+
		e^{-2\kappa_2 t} 
		\norm{\nabla \partial_t \bff{u}^\epsilon(t)}{\bb{L}^2}^2.
	\end{align*}
	Choose $\hat t$ sufficiently large such that
	\[
	\frac{2c c_0 \gamma^2}{\kappa_2 \epsilon^2} e^{-2\kappa_2 \hat t} 
	\leq 
	\kappa_2
	\quad\text{and}\quad 
	e^{-2\kappa_2 \hat t} 
	\leq 
	\frac12 
	\kappa_2 \epsilon.
	\]
	Then for~$t\ge\hat t$, we deduce
	\[
		\ddt \Psi_2(t)
		+ 
		\kappa_2 \Psi_2(t)
		\le 0,	
	\]
	which then yields~\eqref{equ:decay dt u}.

	\medskip
	\noindent
	\textbf{Proof of (iii)}:
	Similarly to the proof of (i) and (ii), we follow the argument leading
	to~\eqref{equ:ddt Lap une} (or~\eqref{equ:nab Lap une}) to obtain
	\begin{align*}
		\ddt \Psi_3(t)
		&+
		2 \kappa_2 \Psi_3(t)
		+ 
		(2\kappa_1-2\kappa_2 \epsilon) \norm{\nabla \Delta \bff{u}^\epsilon(s)}{\bb{L}^2}^2
		\nn \\
		&\leq
		2\gamma
		\left|\inpro{\nabla\Big(\bff{u}^\epsilon(t)\times\Delta\bff{u}^\epsilon(t)\Big)}
		{\nabla\Delta\bff{u}^\epsilon(t)}_{\bb{L}^2}\right|
		+ 2\kappa_2\mu
		\left|
		\inpro{\nabla\Big(|\bff{u}^\epsilon(t)|^2\bff{u}^\epsilon(t)\Big)}
		{\nabla\Delta\bff{u}^\epsilon(t)}_{\bb{L}^2}
		\right|
		\nn \\
		&=: R_1+R_2.
	\end{align*}
	For the term $R_1$, by using successively H\"older's inequality, the Sobolev
	embedding~$\bb{H}^{1+\alpha} \hookrightarrow \bb{W}^{1,4}$, Young's inequality, and~\eqref{equ:u H alpha decay}, we obtain
	\begin{align*}
		R_1
		&=
		2\gamma 
		\left| 
		\inpro{\nabla \bff{u}^\epsilon(t) \times \Delta \bff{u}^\epsilon(t)}
		{\nabla \Delta \bff{u}^\epsilon(t)}_{\bb{L}^2} 
		\right| 
		\\
		&\leq
		c
		\norm{\nabla \bff{u}^\epsilon(t)}{\bb{L}^4} 
		\norm{\Delta \bff{u}^\epsilon(t)}{\bb{L}^4} 
		\norm{\nabla \Delta \bff{u}^\epsilon(t)}{\bb{L}^2}
		\\
		&\leq
		\frac{c}{\epsilon}e^{-2\kappa_2 t} \big(\norm{\Delta \bff{u}^\epsilon(t)}{\bb{L}^2} + \norm{\nabla\Delta \bff{u}^\epsilon(t)}{\bb{L}^2} \big) \norm{\nabla\Delta \bff{u}^\epsilon(t)}{\bb{L}^2}
		\\
		&\leq
		\frac{c}{\epsilon}e^{-2\kappa_2 t}
		\norm{\Delta \bff{u}^\epsilon(t)}{\bb{L}^2}^2
		+
		\frac{c}{\epsilon}e^{-\kappa_2 t} 
		\norm{\nabla \Delta \bff{u}^\epsilon(t)}{\bb{L}^2}^2.
	\end{align*}
	By similar argument, using the embedding $\bb{H}^{1+\alpha}\hookrightarrow \bb{L}^\infty$,
	\begin{align*}
		R_2
		&\leq
		c 
		\left| 
		\inpro{(\bff{u}^\epsilon(t)\cdot \nabla \bff{u}^\epsilon(t)) \bff{u}^\epsilon(t)}
		{\nabla\Delta \bff{u}^\epsilon(t)}_{\bb{L}^2} 
		\right| 
		+
		c 
		\left| 
		\inpro{|\bff{u}^\epsilon(t)|^2 \nabla \bff{u}^\epsilon(t)}
		{\nabla \Delta \bff{u}^\epsilon(t)}_{\bb{L}^2} 
		\right| 
		\\
		&\leq
		c 
		\norm{\bff{u}^\epsilon(t)}{\bb{L}^\infty}^2 
		\norm{\nabla \bff{u}^\epsilon(t)}{\bb{L}^2} 
		\norm{\nabla \Delta \bff{u}^\epsilon(t)}{\bb{L}^2}
		\\
		&\leq
		c 
		\norm{\bff{u}^\epsilon(t)}{\bb{H}^{1+\alpha}}^2 
		\Big(
		\norm{\bff{u}^\epsilon(t)}{\bb{L}^2} 
		+
		\norm{\Delta \bff{u}^\epsilon(t)}{\bb{L}^2} 
		\Big)
		\norm{\nabla \Delta \bff{u}^\epsilon(t)}{\bb{L}^2}
		\\
		&\leq
		\frac{c}{\epsilon} e^{-2 \kappa_2 t}
		\Big(
			\norm{\bff{u}^\epsilon(t)}{\bb{L}^2} 
			+
			\norm{\Delta \bff{u}^\epsilon(t)}{\bb{L}^2} 
		\Big)
		\norm{\nabla \Delta \bff{u}^\epsilon(t)}{\bb{L}^2}
		\\
		&\leq
		\frac{c}{\epsilon} e^{-2 \kappa_2 t} \norm{\bff{u}^\epsilon(t)}{\bb{L}^2}^2
		+
		\frac{c}{\epsilon} e^{-2 \kappa_2 t} \norm{\Delta \bff{u}^\epsilon(t)}{\bb{L}^2}^2
		+
		\frac{c}{\epsilon} e^{-2 \kappa_2 t}
		\norm{\nabla\Delta \bff{u}^\epsilon(t)}{\bb{L}^2}^2.
	\end{align*}
	Altogether, we have 
	\[
		\ddt \Psi_3(t)
		+
		2 \kappa_2 \Psi_3(t)
		\le
			\frac{c}{\epsilon} e^{-2 \kappa_2 t}
			\left( \norm{\bff{u}^\epsilon(t)}{\bb{L}^2}^2
			+
			\norm{\Delta \bff{u}^\epsilon(t)}{\bb{L}^2}^2
			+
			\norm{\nabla\Delta \bff{u}^\epsilon(t)}{\bb{L}^2}^2 \right).
	\]
	By choosing~$\hat t$ sufficiently large so that
	\[
		\frac{c}{\epsilon} e^{-2 \kappa_2 t}
		\le \kappa_2 \epsilon,
	\]
	upon rearranging the terms we deduce
	\[
		\ddt \Psi_3(t)
		+
		\kappa_2 \Psi_3(t)
		\le 
		\kappa_2
		\norm{\bff{u}^\epsilon(t)}{\bb{L}^2}^2
		\quad\forall t\in[\hat t,\infty),
	\]
	which together with~\eqref{equ:Psi0 Psi1} yields
	\[
		\ddt 
		\Big(\Psi_0(t) + \Psi_3(t) \Big)
		+
		\kappa_2 \Big(\Psi_0(t) + \Psi_3(t) \Big)
		\le 
		0
		\quad\forall t\in[\hat t,\infty).
	\]
	This implies
	\[
		\Psi_0(t) + \Psi_3(t) \le e^{- \kappa_2 t}
		\quad\forall t\in[\hat t,\infty),
	\]
	which yields~\eqref{equ:decay delta u}.

	Finally, it follows from~\eqref{equ:LLB eps} that
	\begin{align*}
		\epsilon \norm{\Delta \partial_t \bff{u}^\epsilon}{\bb{L}^2}^2
		&\leq
		\norm{\partial_t \bff{u}^\epsilon}{\bb{L}^2}^2
		+
		\kappa_1 \norm{\Delta \bff{u}^\epsilon}{\bb{L}^2}^2
		+
		\gamma \norm{\bff{u}^\epsilon}{\bb{L}^\infty}^2
		\norm{\Delta \bff{u}^\epsilon}{\bb{L}^2}^2
		+
		\kappa_2 \norm{\bff{u}^\epsilon}{\bb{L}^2}^2
		+
		\kappa_2 \mu \norm{\bff{u}^\epsilon}{\bb{L}^6}^6
		\leq
		ce^{-\kappa_2 t},
	\end{align*}
	where we used~\eqref{equ:u H alpha decay}, \eqref{equ:decay dt u}, and~\eqref{equ:decay delta u} in the last step. By the same argument as in the proof of~\eqref{equ:u H alpha decay}, we obtain~\eqref{equ:decay delta dt u}.
\end{proof}

In the following lemma, we derive a decay estimate for the finite element
approximation of the solution $\bff{u}^\epsilon$.

\begin{lemma}
	Assume that~$\epsilon<\kappa_1/\kappa_2$. For any $n\in \bb{N}$, we have
	\begin{equation*}
		\norm{\bff{u}_h^{\epsilon, (n)}}{\bb{L}^2}^2
		+
		\epsilon \norm{\nabla \bff{u}_h^{\epsilon, (n)}}{\bb{L}^2}^2
		\leq
		\Big(
			\norm{\bff{u}_h^{(0)}}{\bb{L}^2}^2+ \epsilon \norm{\nabla
			\bff{u}_h^{(0)}}{\bb{L}^2}^2 
		\Big) 
		e^{-\lambda t_n},
	\end{equation*}
	where $\lambda=2\kappa_2(1+2\kappa_2 k)^{-1}$. In particular, $\norm{\bff{u}_h^{\epsilon, (n)}}{\bb{L}^2}^2+ \epsilon \norm{\nabla \bff{u}_h^{\epsilon, (n)}}{\bb{L}^2}^2 \to 0$ as $t_n\to\infty$.
\end{lemma}

\begin{proof}
	From~\eqref{equ:uhj L2}, writing $a_j:= \norm{\bff{u}_h^{\epsilon, (j)}}{\bb{L}^2}^2+ \epsilon \norm{\nabla \bff{u}_h^{\epsilon, (j)}}{\bb{L}^2}^2$, we obtain
	\begin{align*}
		\frac{a_j - a_{j-1}}{k} + 2\kappa_2 a_j + (2\kappa_1-2\kappa_2\epsilon) \norm{\nabla \bff{u}_h^{\epsilon, (j)}}{\bb{L}^2}^2
		\leq
		0.
	\end{align*}
	The result then follows from a discrete version of Gronwall's inequality
	\citep[Proposition~3.1]{Emm99}.
\end{proof}

We will exploit these exponential decay estimates to obtain an error estimate
that does not deteriorate over a long time $t\in [\hat{t},\infty)$. First, note
that as a consequence of~\eqref{equ:rho rhot}, 
\eqref{equ:decay u}--\eqref{equ:decay delta dt u}, and the assumed regularity of $\bff{u}^\epsilon$, we
have, for $t\in [\hat{t},\infty)$,
\begin{equation}
	\label{equ:rho et Wp}
	\begin{aligned}
	\norm{\bff{\rho}(t)}{\bb{L}^2} 
	+
	h^\alpha \norm{\nabla\bff{\rho}(t)}{\bb{L}^2}
	&\le
	c h^{1+\alpha} e^{-\kappa_2 t}, 
	\\
	\norm{\partial_t \bff{\rho}(t)}{\bb{L}^2} 
	+
	h^\alpha \norm{\nabla \partial_t \bff{\rho}(t)}{\bb{L}^2}
	&\le
	c h^{1+\alpha} e^{-\kappa_2 t}.
	\end{aligned}
\end{equation}
Furthermore, noting the embeddings $\bb{H}^{1+\alpha} \hookrightarrow \bb{W}^{1,4}\hookrightarrow \bb{L}^\infty$, inequality~\eqref{equ:uj uj1} becomes 
\begin{equation}\label{equ:u tau Lp}
	\begin{aligned}
	\norm{\bff{u}^\epsilon(s_2)-\bff{u}^\epsilon(s_1)}{\bb{L}^p}
	&\leq
	ck e^{-\frac{\kappa_2 s_1}{2}} \quad \forall p\in [1,\infty],
	\\
	\norm{\nabla\bff{u}^\epsilon(s_2)-\nabla\bff{u}^\epsilon(s_1)}{\bb{L}^{p}}
	&\leq
	ck e^{-\frac{\kappa_2 s_1}{2}} \quad \forall p\in [1,4],
	\end{aligned}
\end{equation}
for any $s_1$ and~$s_2$ satisfying~$\hat t \le t_{j-1} \le s_1\leq s_2 \le t_j$ for 
some~$j=1,\ldots,N$.

We are now ready to prove the main result of this section.
\begin{theorem}\label{the:err long}
	Suppose that $ \epsilon < \kappa_1/ \kappa_2$. Assume that the strong solution $\bff{u}^\epsilon$ of the $\epsilon$-LLBE corresponding to initial data $\bff{u}_0^\epsilon\in \bb{H}^2_\Delta$ exists, satisfying the decay estimates in Lemma~\ref{lem:dec}. Assume that $k\ell_h\leqs 1$, where $\ell_h$ is defined in~\eqref{equ:ell h}. For any $n=1,\ldots,N$, where $t_n\in [0,T]$, and $\alpha$ satisfying~\eqref{equ:condition alpha}, we have
	\[
	\norm{\bff{u}_h^{\epsilon, (n)}-\bff{u}^\epsilon(t_n)}{\bb{H}^1}
	\leq c(h^\alpha+k),
	\]
	where $c$ is a constant independent of $n$, $h$, $k$, and $T$ (but may depend on $\epsilon$).
\end{theorem}

\begin{proof}
	Repeating the arguments in the proof of Theorem~\ref{the:err}, but
	replacing~\eqref{equ:rho rhot} and~\eqref{equ:uj uj1} by~\eqref{equ:rho
	et Wp} and~\eqref{equ:u tau Lp}, respectively, we obtain, similarly
	to~\eqref{equ:the 33},
	\begin{align*}
		&\frac{1}{2} \left(\norm{\bff{\theta}^{\epsilon, (j)}}{\bb{L}^2}^2 - \norm{\bff{\theta}^{\epsilon, (j-1)}}{\bb{L}^2}^2\right)
		+
		\frac{1}{2} \norm{\bff{\theta}^{\epsilon, (j)}- \bff{\theta}^{\epsilon, (j-1)}}{\bb{L}^2}^2
			+
		\frac{\epsilon}{2} \left(\norm{\nabla \bff{\theta}^{\epsilon, (j)}}{\bb{L}^2}^2 - \norm{\nabla \bff{\theta}^{\epsilon, (j-1)}}{\bb{L}^2}^2\right)
		\nn \\
		&\quad
		+
		\frac{\epsilon}{2} \norm{\nabla \bff{\theta}^{\epsilon, (j)}- \nabla \bff{\theta}^{\epsilon, (j-1)}}{\bb{L}^2}^2
		+
		k\kappa_1 \norm{\nabla \bff{\theta}^{\epsilon, (j)}}{\bb{L}^2}^2
		+
		k\kappa_2 \norm{\bff{\theta}^{\epsilon, (j)}}{\bb{L}^2}^2
		+
		k\kappa_2\mu \norm{\abs{\bff{u}_h^{(j-1)}} \abs{\bff{\theta}^{\epsilon, (j)}}}{\bb{L}^2}^2
		\\
		&\leq
		c(1+M)k(h^{2\alpha}+k^2) e^{-c t_{j-1}}
		+ 
		cke^{-c t_{j-1}} \norm{\bff{\theta}^{\epsilon, (j-1)}}{\bb{L}^2}^2
		\\
		&\quad
		+
		\delta k \norm{\bff{\theta}^{\epsilon, (j)}}{\bb{L}^2}^2
		+
		\delta k \norm{\nabla \bff{\theta}^{\epsilon, (j-1)}}{\bb{L}^2}^2
		+
		\delta k \norm{\nabla \bff{\theta}^{\epsilon, (j)}}{\bb{L}^2}^2
		+
		\delta k \norm{\abs{\bff{u}_h^{\epsilon, (j-1)}} \abs{\bff{\theta}^{\epsilon, (j)}}}{\bb{L}^2}^2,
	\end{align*}
	where $M$ is defined in~\eqref{equ:uhj infty}.
	We now choose $\delta< \frac{1}{8} \min\{\kappa_1,\kappa_2, \kappa_2\mu\}$, and let $j\geq j_0$ with
	$j_0$ sufficiently large so that $ce^{-ct_{j-1}}<\kappa_2/4$. In this manner, the last three terms on the right-hand
	side can be absorbed to corresponding terms on the left-hand side.
	Putting
	\[
	a_j:= 
	(1+k\kappa_2) \norm{\bff{\theta}^{\epsilon, (j)}}{\bb{L}^2}^2
	+ 
	(\epsilon+k\kappa_1) \norm{\nabla \bff{\theta}^{\epsilon, (j)}}{\bb{L}^2}^2,
	\]
	we deduce
	\[
		a_j - a_{j-1}
		\le
		c(1+M) k (h^{2\alpha}+k^2) e^{-ct_{j-1}},
		\quad \forall j \ge j_0.
	\]
	Summing over~$j$ from~$j_0$ to~$n$ yields
	\begin{align*}
		a_n
		&\le
		a_{j_0-1} 
		+ 
		c(1+M) k(h^{2\alpha}+k^2) \sum_{j=j_0}^{\infty} e^{-ct_{j-1}}
		\\
		&\le 
		ce^{cj_0 k} (h^{2\alpha}+k^2)
		+ 
		ck(h^{2\alpha}+k^2) \sum_{j=j_0}^{\infty} e^{-c(j-1)k},
	\end{align*}
	where in the last step we used Theorem~\ref{the:err}.
	The constant $c$ depends only on
	the coefficients of the equation. Since~$k\sum_{j=j_0}^{\infty}
	e^{-c(j-1)k}$ is bounded by a constant depending on~$j_0$, we obtain the
	required estimate for~$\bff{\theta}^{\epsilon, (n)}$. The required result then
	follows by induction and the triangle inequality, similar to the proof of Theorem~\ref{the:err}.
\end{proof}

\section{Numerical simulations} \label{sec:num sim}

To assess convergence, we perform several simulations using the open-source package~\textsc{FEniCS}~\citep{AlnaesEtal15}. The results are presented in this section. Recall that $\bff{u}_h^{(n)}$ and $\bff{u}_h^{\epsilon, (n)}$ are computed by Algorithm~\ref{alg:fem llbe} and Algorithm~\ref{alg:fem eps llbe}, respectively. We carry out four numerical observations to assess the convergence with respect to $h$, $k$, and $\epsilon$ separately:
{
\begin{enumerate}

\item[(O1)] Snapshots at selected times: We run the simulation with a fixed $\epsilon=0.001$, $h=1/8$, and $k=2.5\times 10^{-3}$ to obtain snapshots of the magnetic spin field $\bff{u}_h^{\epsilon, (n)}$ at selected times (i.e., selected values of $n$).

\item[(O2)] Convergence with respect to $h$:
Since the exact solution of the equation is not known, we use extrapolation to verify the order of convergence experimentally. We fix $\epsilon=0.001$ and $k=2.5\times 10^{-3}$, then vary $h$. For $s=0$ or $1$, the extrapolated spatial orders of convergence are computed by
\begin{equation}\label{equ:hrate}
	h\text{-rate}_s :=  
    \log_2 \left[\frac{\max_n \norm{\bff{e}_{2h}^{(n)}}{\bb{H}^s}}{\max_n \norm{\bff{e}_{h}^{(n)}}{\bb{H}^s}}\right]
    \quad\text{and}\quad
    h\text{-rate}_s^\epsilon :=
    \log_2 \left[\frac{\max_n \norm{\bff{e}_{2h}^{\epsilon, (n)}}{\bb{H}^s}}{\max_n \norm{\bff{e}_{h}^{\epsilon, (n)}}{\bb{H}^s}}\right],
\end{equation}
where
$\bff{e}_h^{(n)} := \bff{u}_{h}^{(n)}-\bff{u}_{h/2}^{(n)}$ and $\bff{e}_h^{\epsilon, (n)} := \bff{u}_{h}^{\epsilon, (n)}-\bff{u}_{h/2}^{\epsilon, (n)}$, depending on the context. We perform log-log plots of $\norm{\bff{e}_h^{(n)}}{\bb{H}^s}$ (or $\norm{\bff{e}_h^{\epsilon, (n)}}{\bb{H}^s}$) against $1/h$.  It is expected that $h\text{-rate}_s \approx 2-s$ and $h\text{-rate}_s^\epsilon \approx 2-s$ (Theorems~\ref{the:err} and~\ref{rem:error}).

\item[(O3)] Convergence with respect to $k$:
The temporal order of convergence is verified by comparing numerical solutions against a reference solution $\bff{u}_\mathrm{ref}(T)$ of the LLBE (or $\bff{u}^\epsilon_\mathrm{ref}(T)$ of the $\epsilon$-LLBE), computed with a significantly refined time step. We fix the spatial resolution, while varying $k=T/(10\times 2^j)$, where $j=0,1,\ldots,4$. Let $\bff{f}_k^{(N)}= \bff{u}_\mathrm{ref}(T)- \bff{u}_{h,k}^{(N)}$ and $\bff{f}_k^{\epsilon, (N)}= \bff{u}_\mathrm{ref}^\epsilon(T)- \bff{u}_{h,k}^{\epsilon, (N)}$ be the errors at the fixed final time $T$. We perform log-log plots of $\norm{\bff{f}_k^{(N)}}{\bb{H}^s}$ (or $\norm{\bff{f}_k^{\epsilon, (N)}}{\bb{H}^s}$) against $1/k$. The temporal convergence rate is observed by computing
\begin{equation}\label{equ:krate}
    k\text{-rate}_s:= \log_2 \left[\frac{\norm{\bff{f}_{2k}^{(N)}}{\bb{H}^s}}{\norm{\bff{f}_k^{(N)}}{\bb{H}^s}} \right]
    \quad \text{and} \quad
    k\text{-rate}^\epsilon_s:= \log_2 \left[\frac{\norm{\bff{f}_{2k}^{\epsilon,(N)}}{\bb{H}^s}}{\norm{\bff{f}_k^{\epsilon,(N)}}{\bb{H}^s}} \right],
    \quad
    s=0,1.
\end{equation}
It is expected that $k\text{-rate}_s\approx 1$ (Theorem~\ref{the:err llb}) and $k\text{-rate}^\epsilon_s\approx 1$ (Theorems~\ref{the:err} and~\ref{rem:error}) for $s=0,1$.

\item[(O4)] Convergence with respect to $\epsilon$: The convergence with respect to $\epsilon$ is assessed by comparing numerical solutions $\bff{u}_h^{\epsilon, (N)}$ against a reference solution $\bff{u}_\mathrm{ref}(T):= \bff{u}_{h_0}^{(N)}$, where $h_0$ and the time step are chosen to be sufficiently small, the values of which will be specified in each simulation. We vary the parameter $\epsilon$ through successive halving, while keeping $h$ and $k$ fixed. Let $\bff{g}^{\epsilon, (N)}:= \bff{u}_\mathrm{ref}(T)- \bff{u}_h^{\epsilon, (N)}$. A log-log plot of $\norm{\bff{g}^{\epsilon, (N)}}{\bb{H}^s}$ against $1/\epsilon$ is performed. The convergence rate with respect to $\epsilon$ is assessed by computing
\begin{equation}\label{equ:epsrate}
    \epsilon\text{-rate}_s:= \log_2 \left[\frac{\norm{\bff{g}^{2\epsilon, (N)}}{\bb{H}^s}}{\norm{\bff{g}^{\epsilon, (N)}}{\bb{H}^s}} \right],
    \quad
    s=0,1.
\end{equation}
It is expected that $\epsilon\text{-rate}_s\approx 1$ (Theorem~\ref{the:u eps con u}) for $s=0,1$.
\end{enumerate}

The details of observations (O1)--(O4) are provided within the description of each simulation below.
}

\subsection{Simulation 1: Square domain}
We fix the domain $\mathscr{D}:= [0,1]^2\subset \bb{R}^2$. The coefficients in~\eqref{equ:LLB pro} are taken to be $\kappa_1=5.0, \kappa_2=2.0, \mu=1.0$, and $\gamma=50.0$. The initial data $\bff{u}_0$ is given by
\begin{equation*}
	\bff{u}_0(x,y)= \big(\cos(2\pi x),\, \sin(2\pi y),\, 2\cos(2\pi x) \sin(2\pi y) \big).
\end{equation*}

{
We carry out four numerical observations as explained above:

\begin{enumerate}
\item[(O1)] First, we run the simulation using Algorithm~\ref{alg:fem eps llbe} as shown in Figure~\ref{fig:snapshots field 2d}. In this configuration, one could see the formation of a Bloch wall around time $t=0.2$, which is dissipating as time progresses. Eventually, the magnetisation vectors will decay to $\bff{0}$ as $t\to\infty$. The plots showing the decay of $\bff{u}_h^{\epsilon,(n)}$ in the $\bb{L}^2$, $\bb{L}^\infty$, and $\bb{H}^1$-norms are shown in Figures~\ref{fig:norms L} and~\ref{fig:norms H}.

\item[(O2)] We perform several numerical simulations with various values of $h$ to compute $\bff{e}_h^{\epsilon, (n)}$. The plot of $\max_n \norm{\bff{e}_h^{\epsilon, (n)}}{\bb{H}^s}$ against $1/h$ (with $h=2^{-j}$, where $j=2,3,\ldots,7$) to observe $h\text{-rate}_s^\epsilon$ (defined in~\eqref{equ:hrate}) is shown in Figure~\ref{fig:order}, confirming the spatial order of convergence.

\item[(O3)] To assess the temporal order of convergence, we set $T=0.2$ and compute the reference solution $\bff{u}^\epsilon_{\mathrm{ref}}(T)$ with $h=1/64$ and $k=T/640$. The plot of $\norm{\bff{f}_k^{\epsilon, (N)}}{\bb{H}^s}$ against $1/k$ (with $k=T/(10\times 2^j)$, where $j=0,1,\ldots,4$) to observe $k\text{-rate}_s^\epsilon$ (defined in~\eqref{equ:krate}) is shown in Figure~\ref{fig:order exp 1a k}.

The same experiments as above are carried out with $\epsilon=0$, i.e., we use Algorithm~\ref{alg:fem llbe} to compute $\bff{u}_h^{(n)}$. The corresponding results are presented in Figures~\ref{fig:snapshots field 2d zero}, \ref{fig:order sim4}, and~\ref{fig:order exp 1b k}. We see that the order of convergence does not deteriorate compared to the previous simulation with $\epsilon>0$, indicating that the order of convergence in Theorem~\ref{the:err llb} may be sub-optimal. Qualitatively, the solution $\bff{u}^\epsilon$ in Figure~\ref{fig:snapshots field 2d} and the solution $\bff{u}$ in Figure~\ref{fig:snapshots field 2d zero} look similar.

\item[(O4)] Finally, to observe $\epsilon\text{-rate}_s$, we compute the reference solution $\bff{u}_\mathrm{ref}(T)$ at $T=0.2$ with $\epsilon=0$, $h=1/64$, and $k=2\times 10^{-3}$. Subsequently, with $h=1/16$ and $k=2\times 10^{-2}$, we vary $\epsilon=2^{-j}$, for $j=3,4,\ldots, 9$, to compute the corresponding solutions $\bff{u}_{h}^{\epsilon, (N)}$, and plot $\norm{\bff{g}^{\epsilon, (N)}}{\bb{H}^s}$ against $1/\epsilon$. This last procedure is repeated with $h=1/64$ and $k=2\times 10^{-3}$ to demonstrate how the mesh size and time step influence the $\epsilon$-convergence. The resulting convergence behaviour as $\epsilon\to 0^+$ is displayed in Figure~\ref{fig:order exp 1 eps}.
\end{enumerate}
}

\begin{figure}[hbt!]
	\centering
	\begin{subfigure}[b]{0.25\textwidth}
		\centering
		\includegraphics[width=\textwidth]{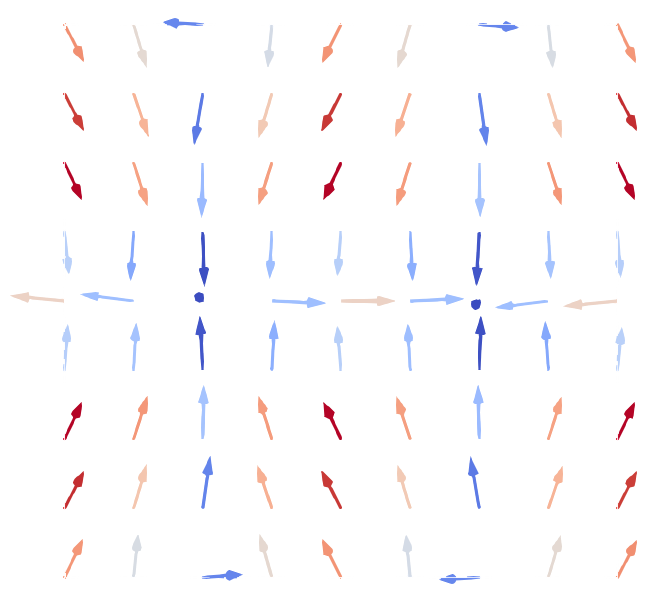}
		\caption{$t=0$}
	\end{subfigure}
	\begin{subfigure}[b]{0.25\textwidth}
		\centering
		\includegraphics[width=\textwidth]{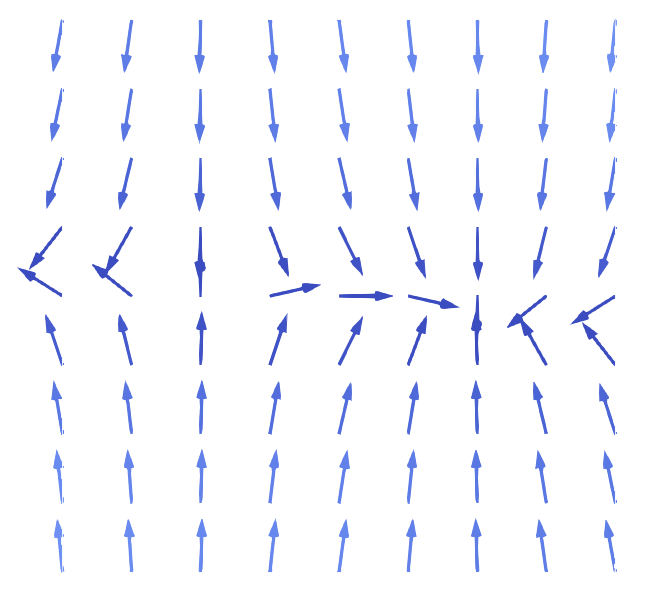}
		\caption{$t=0.025$}
	\end{subfigure}
	\begin{subfigure}[b]{0.25\textwidth}
		\centering
		\includegraphics[width=\textwidth]{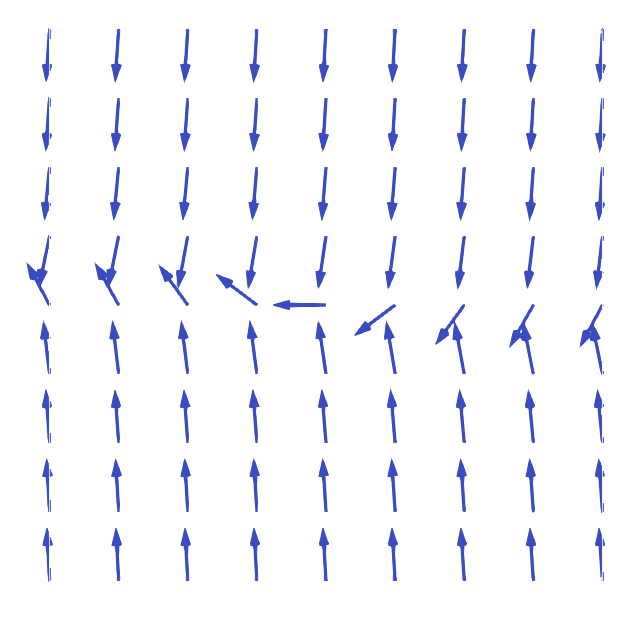}
		\caption{$t=0.2$}
	\end{subfigure}
	\begin{subfigure}[b]{0.08\textwidth}
		\centering
		\includegraphics[width=\textwidth]{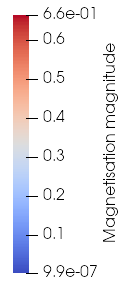}
	\end{subfigure}
	\begin{subfigure}[b]{0.25\textwidth}
		\centering
		\includegraphics[width=\textwidth]{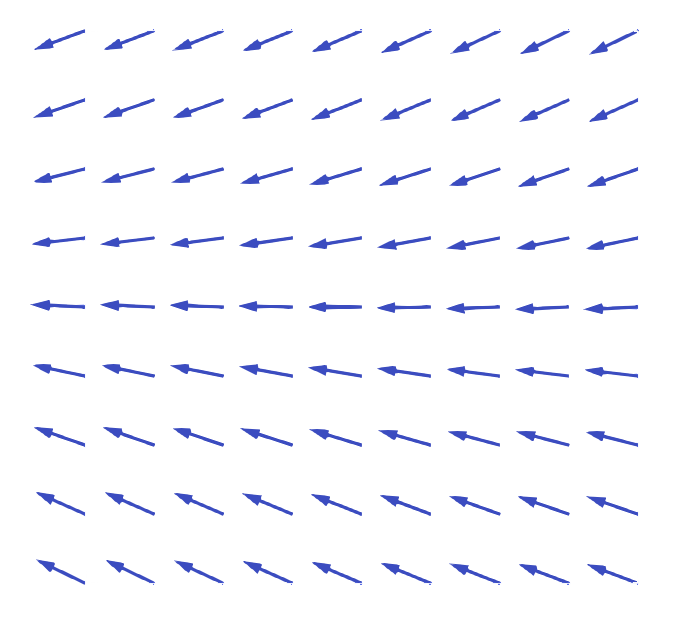}
		\caption{$t=0.3$}
	\end{subfigure}
	\begin{subfigure}[b]{0.25\textwidth}
		\centering
		\includegraphics[width=\textwidth]{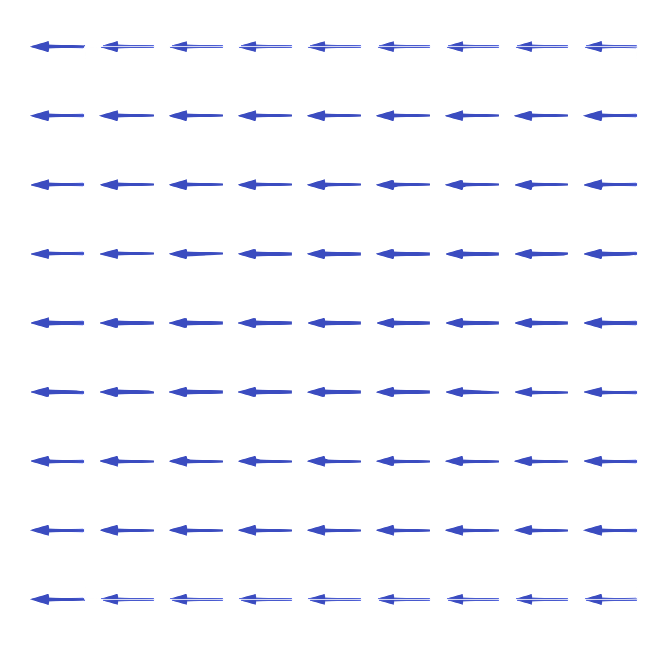}
		\caption{$t=0.4$}
	\end{subfigure}
	\begin{subfigure}[b]{0.25\textwidth}
		\centering
		\includegraphics[width=\textwidth]{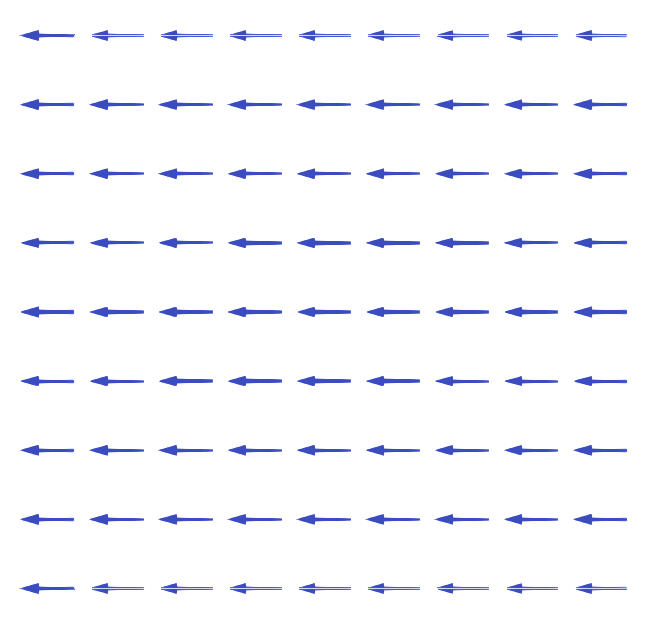}
		\caption{$t=0.5$}
	\end{subfigure}
	\begin{subfigure}[b]{0.08\textwidth}
		\centering
		\includegraphics[width=\textwidth]{legend_2d.png}
	\end{subfigure}
	\caption{Snapshots of the magnetic spin field $\bff{u}_h^{\epsilon, (n)}$ (projected onto $\bb{R}^2$) at given times using Algorithm~\ref{alg:fem eps llbe} in Simulation 1. The colours indicate relative magnitude of the vectors.}
	\label{fig:snapshots field 2d}
\end{figure}

\begin{figure}[!hbt]
	\centering
	\begin{subfigure}[b]{0.49\textwidth}
		\centering
		\includegraphics[width=\textwidth]{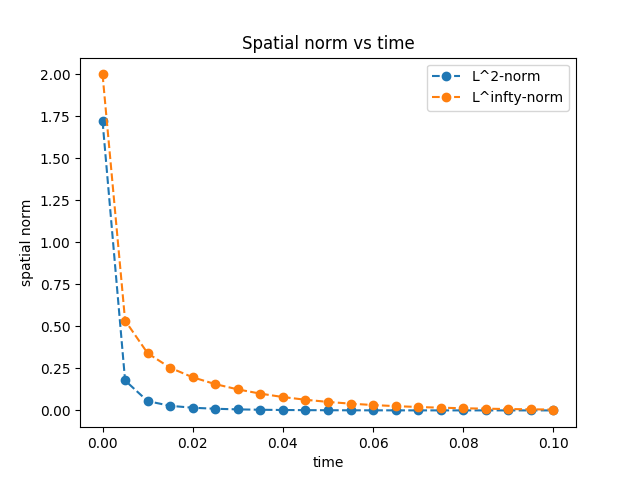}
		\caption{Plot of $\norm{\bff{u}_h^{\epsilon,(n)}}{\bb{L}^2}$ and
			$\norm{\bff{u}_h^{\epsilon,(n)}}{\bb{L}^\infty}$ against $n$.
			Theoretical result: Theorem~\ref{the:ut inf} and Lemma~\ref{lem:dec}.}
		\label{fig:norms L}
	\end{subfigure}
	\begin{subfigure}[b]{0.49\textwidth}
		\centering
		\includegraphics[width=\textwidth]{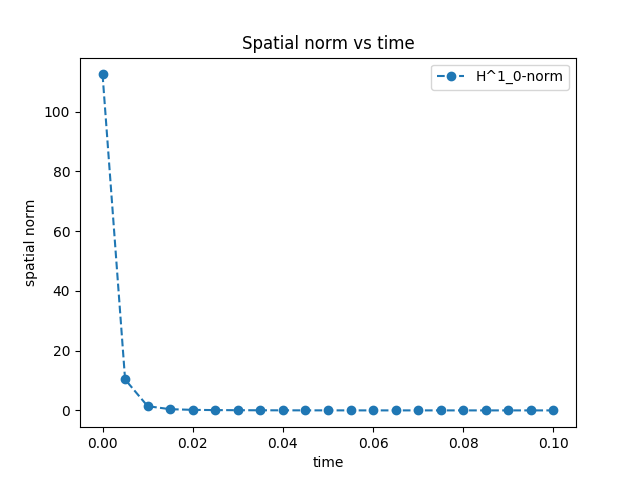}
		\caption{Plot of $\norm{\nabla \bff{u}_h^{\epsilon,(n)}}{\bb{L}^2}$ against $n$. Theoretical result: Lemma~\ref{lem:dec}.}
		\label{fig:norms H}
	\end{subfigure}
	\caption{Plots of various norms of $\bff{u}_h^{\epsilon,(n)}$ against $n$ in Simulation 1}
\end{figure}

\begin{figure}[hbt!]
	\begin{subfigure}[b]{0.42\textwidth}
		\centering
		\begin{tikzpicture}
			\begin{axis}[
				title=Log-log plot of $\max_n \norm{\bff{e}_h^{\epsilon, (n)}}{\bb{H}^s}$ against $1/h$,
				height=1.4\textwidth,
				width=1\textwidth,
				xlabel= $1/h$,
				ylabel= $\max_n \norm{\bff{e}_h^{\epsilon, (n)}}{\bb{H}^s}$,
				xmode=log,
				ymode=log,
				legend pos=south west,
				legend cell align=left,
				]
				\addplot+[mark=*,red] coordinates {(4,0.741)(8,0.393)(16,0.180)(32,0.0884)(64,0.0438)(128,0.022)};
				\addplot+[mark=*,blue] coordinates {(4,0.072)(8,0.0199)(16,0.00579)(32,0.00143)(64,0.000369)(128,0.0000923)};
				\addplot+[dashed,no marks,red,domain=40:128]{5.5/x};
				\addplot+[dashed,no marks,blue,domain=40:128]{4.5/x^2};
				\legend{\small{$s=1$}, \small{$s=0$}, \small{order 1 line}, \small{order 2 line}}
			\end{axis}
		\end{tikzpicture}
		\caption{Spatial order of convergence of Algorithm~\ref{alg:fem eps llbe} in Simulation 1.}
		\label{fig:order}
	\end{subfigure}
	\hspace{0.5cm}
	\begin{subfigure}[b]{0.42\textwidth}
		\centering
		\begin{tikzpicture}
			\begin{axis}[
				title=Log-log plot of $\norm{\bff{f}_k^{\epsilon, (N)}}{\bb{H}^s}$ against $N$,
				height=1.4\textwidth,
				width=1\textwidth,
				xlabel= $N$,
				ylabel= $\norm{\bff{f}_k^{\epsilon, (N)}}{\bb{H}^s}$,
				xmode=log,
				ymode=log,
				legend pos=south west,
				legend cell align=left,
				]
				\addplot+[mark=*,red] coordinates {(10,0.001)(20,0.000302)(40,0.000108)(80,0.0000426)(160,0.0000163)};
				\addplot+[mark=*,blue] coordinates {(10,0.00032)(20,0.000096)(40,0.0000345)(80,0.0000135)(160,0.00000516)};
				\addplot+[dashed,no marks,black,domain=60:160]{0.007/x};
				\legend{\small{$s=1$}, \small{$s=0$}, \small{order 1 line}}
			\end{axis}
		\end{tikzpicture}
		\caption{Temporal order of convergence of Algorithm~\ref{alg:fem eps llbe} in Simulation 1.}
		\label{fig:order exp 1a k}
	\end{subfigure}
	\caption{Spatial and temporal orders of convergence of Algorithm~\ref{alg:fem eps llbe} in Simulation 1. Theoretical result: Theorems~\ref{the:err} and~\ref{rem:error}.}
\end{figure}

\begin{figure}[hbt!]
	\centering
	\begin{subfigure}[b]{0.25\textwidth}
		\centering
		\includegraphics[width=\textwidth]{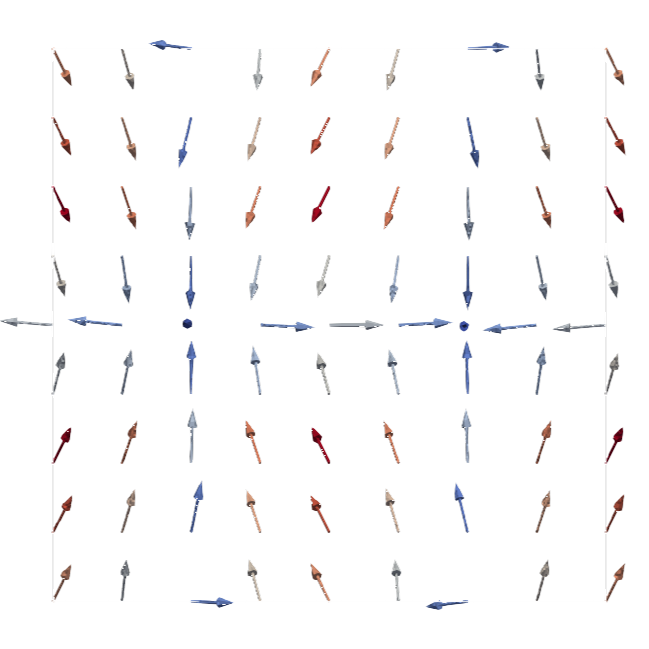}
		\caption{$t=0$}
	\end{subfigure}
	\begin{subfigure}[b]{0.25\textwidth}
		\centering
		\includegraphics[width=\textwidth]{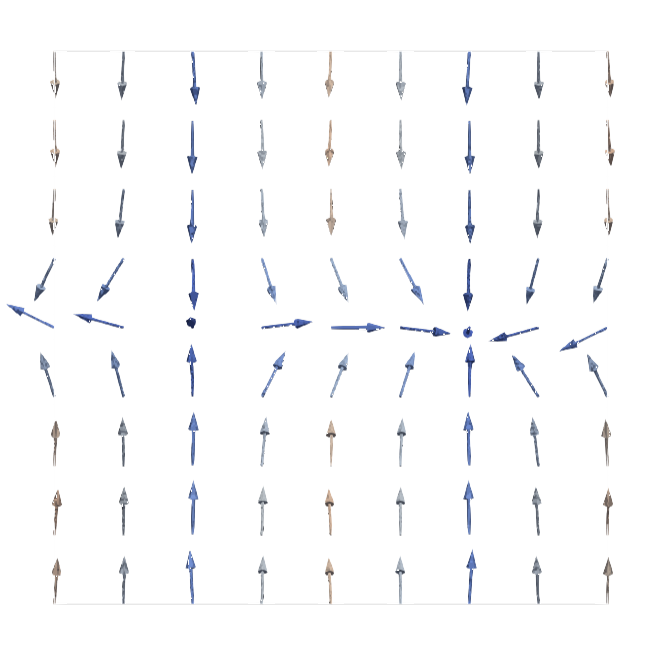}
		\caption{$t=0.025$}
	\end{subfigure}
	\begin{subfigure}[b]{0.25\textwidth}
		\centering
		\includegraphics[width=\textwidth]{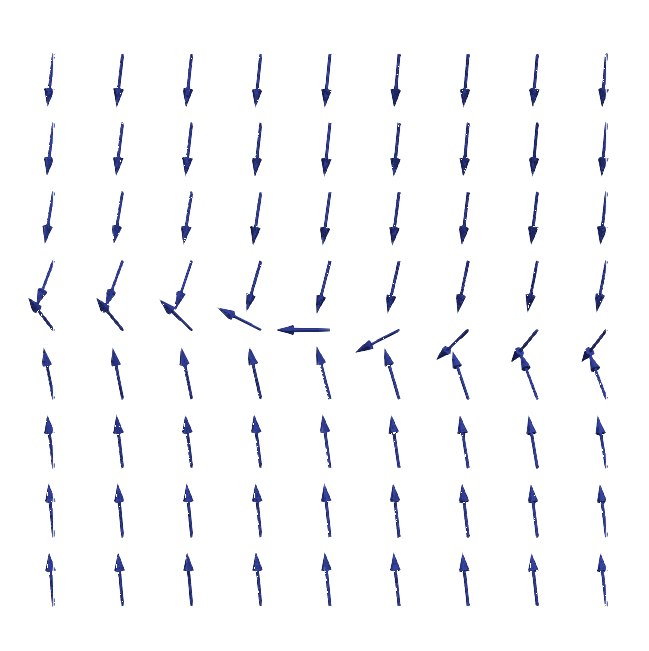}
		\caption{$t=0.2$}
	\end{subfigure}
	\begin{subfigure}[b]{0.08\textwidth}
		\centering
		\includegraphics[width=\textwidth]{legend_2d.png}
	\end{subfigure}
	\begin{subfigure}[b]{0.25\textwidth}
		\centering
		\includegraphics[width=\textwidth]{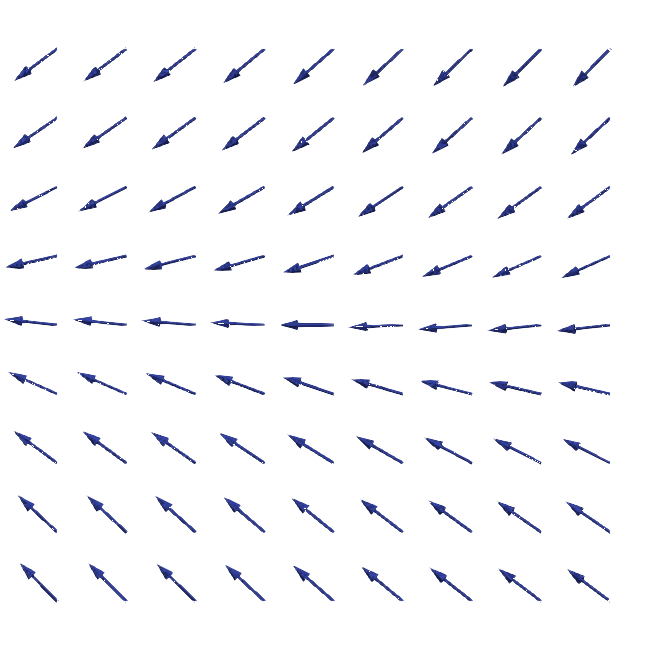}
		\caption{$t=0.3$}
	\end{subfigure}
	\begin{subfigure}[b]{0.25\textwidth}
		\centering
		\includegraphics[width=\textwidth]{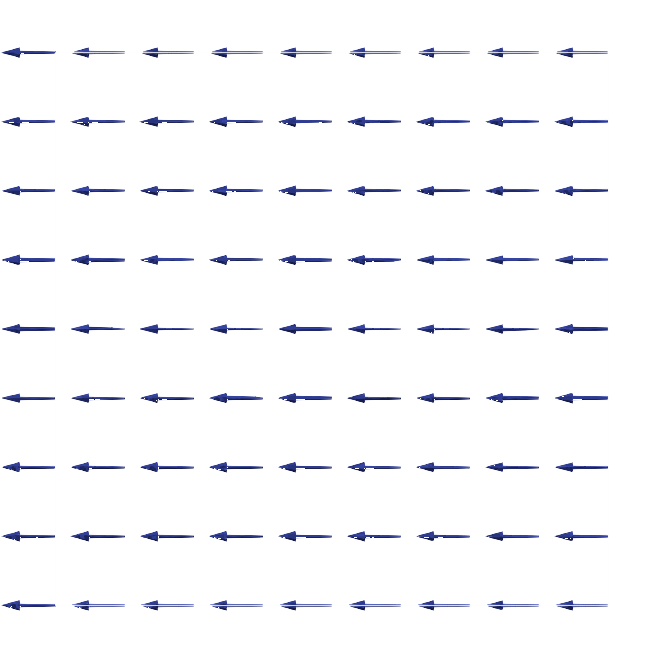}
		\caption{$t=0.4$}
	\end{subfigure}
	\begin{subfigure}[b]{0.25\textwidth}
		\centering
		\includegraphics[width=\textwidth]{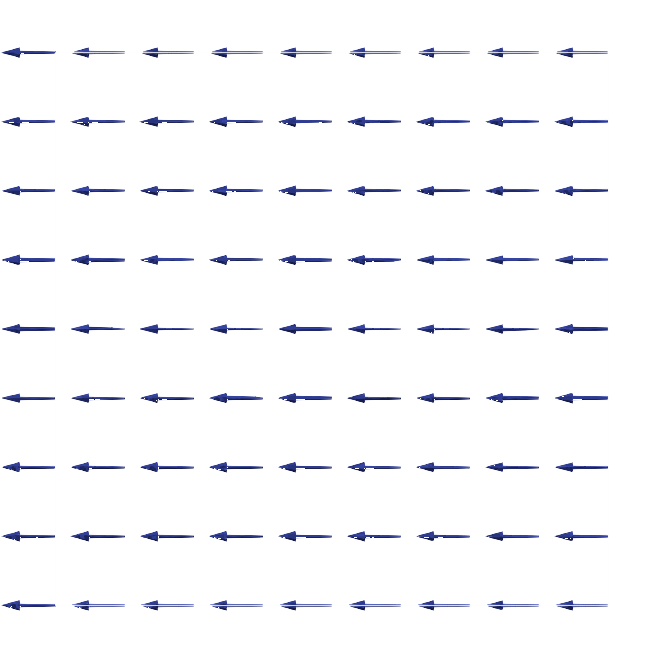}
		\caption{$t=0.5$}
	\end{subfigure}
	\begin{subfigure}[b]{0.08\textwidth}
		\centering
		\includegraphics[width=\textwidth]{legend_2d.png}
	\end{subfigure}
	\caption{Snapshots of the magnetic spin field $\bff{u}_h^{(n)}$ (projected onto $\bb{R}^2$) at given times using Algorithm~\ref{alg:fem llbe} in Simulation 1. The colours indicate relative magnitude of the vectors.}
	\label{fig:snapshots field 2d zero}
\end{figure}

\begin{figure}[hbt!]
	\begin{subfigure}[b]{0.42\textwidth}
		\centering
		\begin{tikzpicture}
			\begin{axis}[
				title=Log-log plot of $\max_n \norm{\bff{e}_h^{(n)}}{\bb{H}^s}$ against $1/h$,
				height=1.4\textwidth,
				width=1\textwidth,
				xlabel= $1/h$,
				ylabel= $\max_n \norm{\bff{e}_h^{(n)}}{\bb{H}^s}$,
				xmode=log,
				ymode=log,
				legend pos=south west,
				legend cell align=left,
				]
				\addplot+[mark=*,red] coordinates {(4,0.691)(8,0.36)(16,0.167)(32,0.0829)(64,0.0412)(128,0.021)};
				\addplot+[mark=*,blue] coordinates {(4,0.11)(8,0.04)(16,0.0095)(32,0.0027)(64,0.00071)(128,0.000175)};
				\addplot+[dashed,no marks,red,domain=40:128]{7/x};
				\addplot+[dashed,no marks,blue,domain=40:128]{9/x^2};
				\legend{\small{$s=1$}, \small{$s=0$}, \small{order 1 line}, \small{order 2 line}}
			\end{axis}
		\end{tikzpicture}
		\caption{Spatial order of convergence of Algorithm~\ref{alg:fem llbe} in Simulation 1.}
		\label{fig:order sim4}
	\end{subfigure}
	\hspace{0.5cm}
	\begin{subfigure}[b]{0.42\textwidth}
		\centering
		\begin{tikzpicture}
			\begin{axis}[
				title=Log-log plot of $\norm{\bff{f}_k^{(N)}}{\bb{H}^s}$ against $N$,
				height=1.4\textwidth,
				width=1\textwidth,
				xlabel= $N$,
				ylabel= $\norm{\bff{f}_k^{(N)}}{\bb{H}^s}$,
				xmode=log,
				ymode=log,
				legend pos=south west,
				legend cell align=left,
				]
				\addplot+[mark=*,red] coordinates {(10,0.00094)(20,0.00028)(40,0.000099)(80,0.0000388)(160,0.0000148)};
				\addplot+[mark=*,blue] coordinates {(10,0.00030)(20,0.000089)(40,0.0000317)(80,0.0000123)(160,0.00000468)};
				\addplot+[dashed,no marks,black,domain=60:160]{0.007/x};
				\legend{\small{$s=1$}, \small{$s=0$}, \small{order 1 line}}
			\end{axis}
		\end{tikzpicture}
		\caption{Temporal order of convergence of Algorithm~\ref{alg:fem llbe} in Simulation 1.}
		\label{fig:order exp 1b k}
	\end{subfigure}
	\caption{Spatial and temporal orders of convergence of Algorithm~\ref{alg:fem llbe} in Simulation 1. Theoretical result: Theorem~\ref{the:err llb}.}
\end{figure}

\begin{figure}[hbt!]
	\begin{center}
	\begin{tikzpicture}
		\begin{axis}[
			title=Log-log plot of $\norm{\bff{g}^{\epsilon, (N)}}{\bb{H}^s}$ against $1/\epsilon$,
			height=0.5\textwidth,
			width=0.5\textwidth,
			xlabel= $1/\epsilon$,
			ylabel= $\norm{\bff{g}^{\epsilon, (N)}}{\bb{H}^s}$,
			xmode=log,
			ymode=log,
			legend pos=outer north east,
			legend cell align=left,
			]
			\addplot+[dashed,mark=*,mark options={fill=red},red] coordinates {(16,0.00433)(32,0.00171)(64,0.000888)(128,0.000583)(256,0.000457)(512, 0.000400)};
			\addplot+[dashed,mark=*,mark options={fill=blue},blue] coordinates {(16,0.00128)(32,0.000542)(64,0.000282)(128,0.000185)(256,0.000145)(512,0.000127)};
			\addplot+[mark=triangle*,mark options={fill=red},red] coordinates {(16,0.00359)(32,0.000743)(64,0.00022)(128,0.0000847)(256,0.0000369)(512,0.0000172)};
			\addplot+[mark=triangle*,mark options={fill=blue},blue] coordinates {(16,0.00114)(32,0.000236)(64,0.0000707)(128,0.000027)(256,0.0000117)(512,0.00000548)};
			\addplot+[dotted,no marks,black,domain=128:512]{0.0016/x};
			\legend{\small{$s=1$ $(h=1/16, k=2\times 10^{-2})$}, \small{$s=0$ $(h=1/16, k=2\times 10^{-2})$}, \small{$s=1$ $(h=1/64, k= 2\times 10^{-3})$}, \small{$s=0$ $(h=1/64, k=2\times 10^{-3})$}, \small{order 1 line}}
		\end{axis}
	\end{tikzpicture}
\end{center}
	\caption{Order of convergence with respect to $\epsilon$ in Simulation 1. Theoretical result: Theorem~\ref{the:u eps con u}.}
	\label{fig:order exp 1 eps}
\end{figure}

\subsection{Simulation 2: Cube domain}
We fix the domain $\mathscr{D}:= [0,1]^3\subset \bb{R}^3$. The coefficients in~\eqref{equ:LLB pro} are taken to be $\kappa_1=5.0, \kappa_2=2.0, \mu=1.0, \gamma=50.0$. The initial data $\bff{u}_0$ is given by
\begin{equation*}
	\bff{u}_0(x,y,z)= \big(2\cos(2\pi x),\, \sin(2\pi y),\, 2\cos(2\pi y) \sin(2\pi z) \big).
\end{equation*}

{
As before, we carry out four numerical observations:
\begin{enumerate}
\item[(O1)] The result of this simulation is displayed in Figure~\ref{fig:snapshots field 3d}. 

\item[(O2)] The log-log plot with $h=2^{-j}$, where $j=1,2,3,4$, is shown in Figure~\ref{fig:order 3d} to observe $h\text{-rate}_s^\epsilon$, verifying the spatial order of convergence.

\item[(O3)] To assess $k\text{-rate}_s^\epsilon$, we set $T=0.05$ and compute the reference solution $\bff{u}^\epsilon_{\mathrm{ref}}(T)$ with $h=1/32$ and $k=T/640$. The log-log plot with $k=T/(10\times 2^j)$, where $j=0,1,\ldots,4$, is shown in Figure~\ref{fig:order 3d temp}.

\item[(O4)] To evaluate $\epsilon\text{-rate}_s$, we first compute the reference solution $\bff{u}_\mathrm{ref}(T)$ at $T=0.1$ by setting $\epsilon=0$, $h=1/8$, and $k=5\times 10^{-3}$. We then take $h=1/4$, $k=10^{-2}$, and compute $\norm{\bff{g}^{\epsilon, (N)}}{\bb{H}^s}$ for $\epsilon=2^{-j}$ ($j=4,5,\ldots, 9$). This procedure is repeated with $h=1/8$ and $k=5\times 10^{-3}$. The resulting log-log plot is displayed in Figure~\ref{fig:order 3d eps}.

\end{enumerate}
}

\begin{figure}[hbt!]
	\centering
	\begin{subfigure}[b]{0.25\textwidth}
		\centering
		\includegraphics[width=1.01\textwidth]{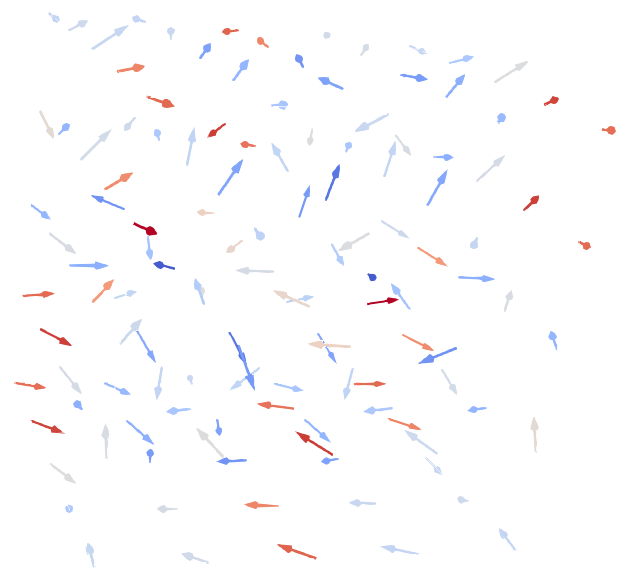}
		\caption{$t=0$}
	\end{subfigure}
	\begin{subfigure}[b]{0.25\textwidth}
		\centering
		\includegraphics[width=0.96\textwidth]{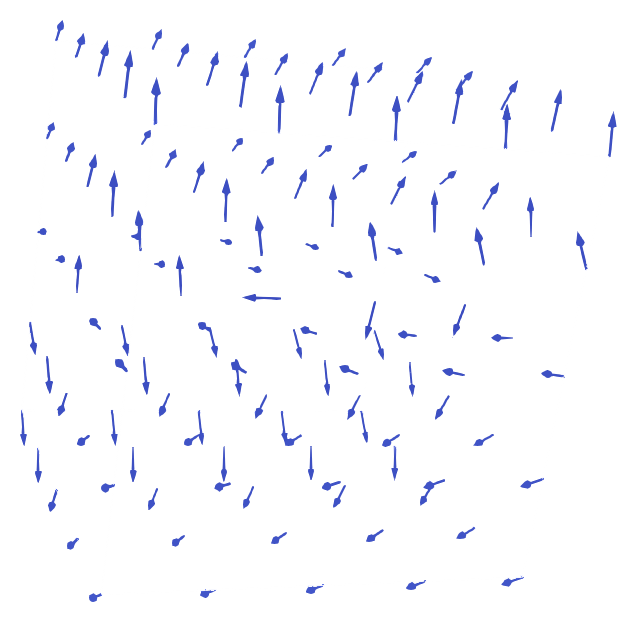}
		\caption{$t=0.075$}
	\end{subfigure}
	\begin{subfigure}[b]{0.25\textwidth}
		\centering
		\includegraphics[width=0.95\textwidth]{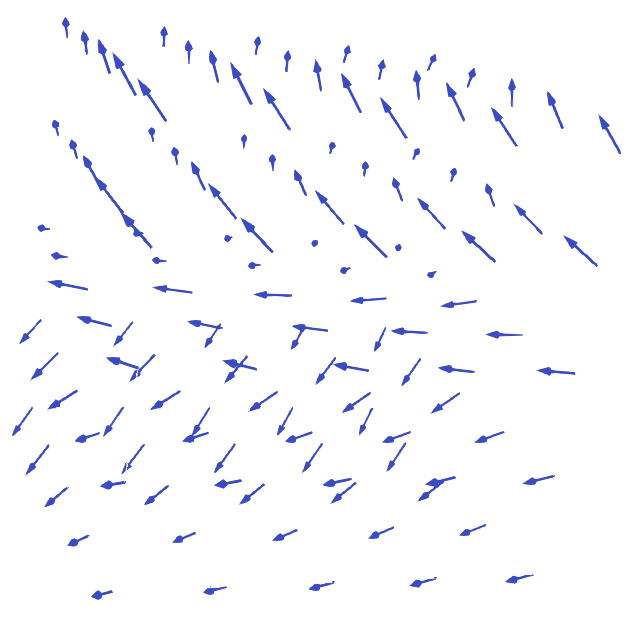}
		\caption{$t=0.2$}
	\end{subfigure}
	\begin{subfigure}[b]{0.08\textwidth}
		\centering
		\includegraphics[width=\textwidth]{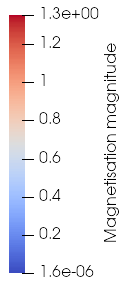}
	\end{subfigure}
	\begin{subfigure}[b]{0.25\textwidth}
		\centering
		\includegraphics[width=\textwidth]{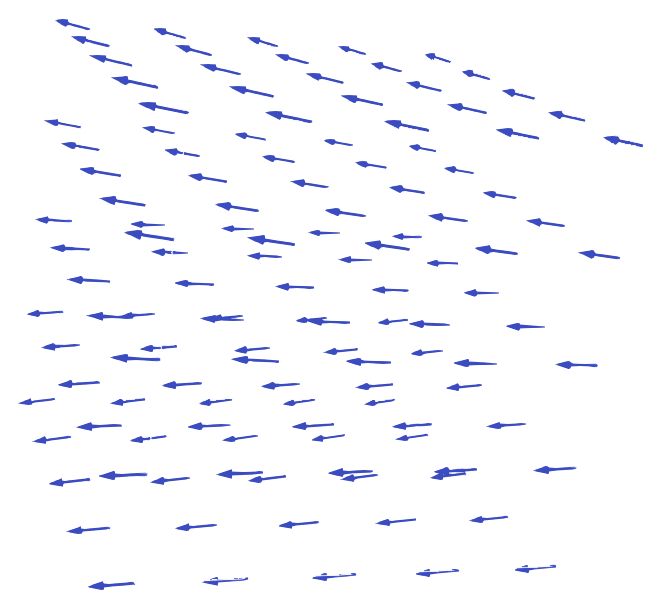}
		\caption{$t=0.25$}
	\end{subfigure}
	\begin{subfigure}[b]{0.25\textwidth}
		\centering
		\includegraphics[width=0.98\textwidth]{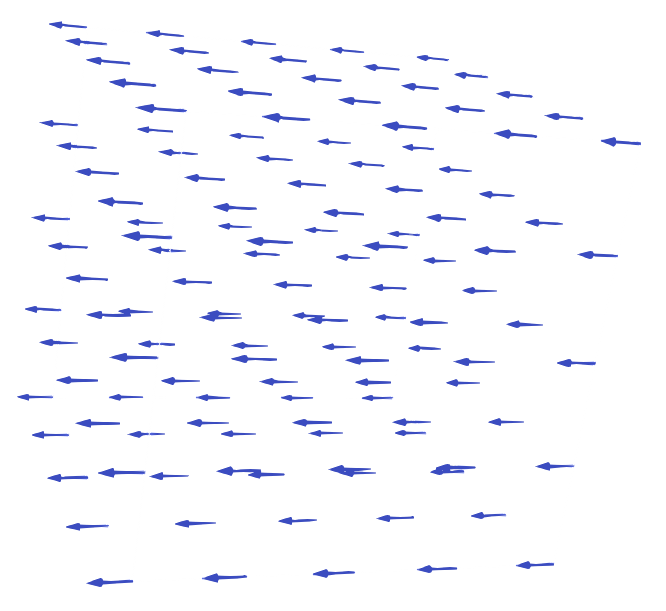}
		\caption{$t=0.4$}
	\end{subfigure}
	\begin{subfigure}[b]{0.25\textwidth}
		\centering
		\includegraphics[width=0.95\textwidth]{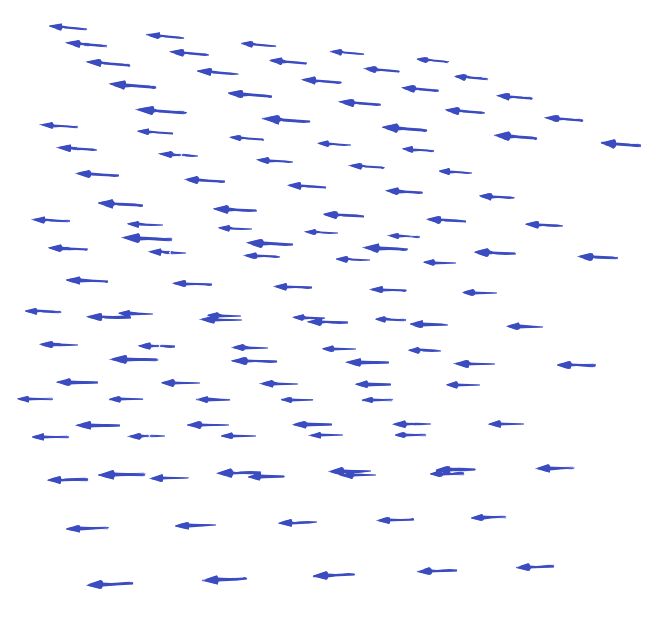}
		\caption{$t=0.5$}
	\end{subfigure}
	\begin{subfigure}[b]{0.08\textwidth}
		\centering
		\includegraphics[width=\textwidth]{legend_3d.png}
	\end{subfigure}
	\caption{Snapshots of the magnetic spin field $\bff{u}_h^{\epsilon, (n)}$ at given times in Simulation~2. The colours indicate relative magnitude of the vectors.}
	\label{fig:snapshots field 3d}
\end{figure}

\begin{figure}[hbt!]
	\begin{subfigure}[b]{0.42\textwidth}
		\centering
		\begin{tikzpicture}
			\begin{axis}[
				title=Log-log plot of $\max_n \norm{\bff{e}_h^{\epsilon, (n)}}{\bb{H}^s}$ against $1/h$,
				height=1.4\textwidth,
				width=1\textwidth,
				xlabel= $1/h$,
				ylabel= $\max_n \norm{\bff{e}_h^{\epsilon, (n)}}{\bb{H}^s}$,
				xmode=log,
				ymode=log,
				legend pos=south west,
				legend cell align=left,
				]
				\addplot+[mark=*,red] coordinates {(2,0.883)(4,0.470)(8,0.231)(16,0.115)};
				\addplot+[mark=*,blue] coordinates {(2,0.185)(4,0.096)(8,0.03)(16,0.00745)};
				\addplot+[dashed,no marks,red,domain=6:16]{3.2/x};
				\addplot+[dashed,no marks,blue,domain=6:16]{4/x^2};
				\legend{\small{$s=1$}, \small{$s=0$}, \small{order 1 line}, \small{order 2 line}}
			\end{axis}
		\end{tikzpicture}
		\caption{Spatial order of convergence of Algorithm~\ref{alg:fem eps llbe} in Simulation 2.}
		\label{fig:order 3d}
	\end{subfigure}
	\hspace{0.5cm}
	\begin{subfigure}[b]{0.42\textwidth}
	\centering
	\begin{tikzpicture}
		\begin{axis}[
			title=Log-log plot of $\norm{\bff{f}_k^{\epsilon, (N)}}{\bb{H}^s}$ against $N$,
			height=1.4\textwidth,
			width=1\textwidth,
			xlabel= $N$,
			ylabel= $\norm{\bff{f}_k^{\epsilon, (N)}}{\bb{H}^s}$,
			xmode=log,
			ymode=log,
			legend pos=south west,
			legend cell align=left,
			]
			\addplot+[mark=*,red] coordinates {(10,1.95)(20,1.71)(40,1.43)(80,1.046)(160,0.625)};
			\addplot+[mark=*,blue] coordinates {(10,0.54)(20,0.49)(40,0.41)(80,0.29)(160,0.16)};
			\addplot+[dashed,no marks,black,domain=60:160]{50/x};
			\legend{\small{$s=1$}, \small{$s=0$}, \small{order 1 line}}
		\end{axis}
	\end{tikzpicture}
	\caption{Temporal order of convergence of Algorithm~\ref{alg:fem eps llbe} in Simulation 2.}
	\label{fig:order 3d temp}
	\end{subfigure}
	\caption{Spatial and temporal orders of convergence of Algorithm~\ref{alg:fem eps llbe} in Simulation 2. Theoretical result: Theorems~\ref{the:err} and~\ref{rem:error}.}
\end{figure}

\begin{figure}[hbt!]
\begin{center}
	\begin{tikzpicture}
		\begin{axis}[
			title=Log-log plot of $\norm{\bff{g}^{\epsilon, (N)}}{\bb{H}^s}$ against $1/\epsilon$,
			height=0.5\textwidth,
			width=0.5\textwidth,
			xlabel= $1/\epsilon$,
			ylabel= $\norm{\bff{g}^{\epsilon, (N)}}{\bb{H}^s}$,
			xmode=log,
			ymode=log,
			legend pos=outer north east,
			legend cell align=left,
			]
			\addplot+[dashed,mark=*,mark options={fill=red},red] coordinates {(16,0.0814)(32,0.015)(64,0.00468)(128,0.00347)(256,0.00335)(512, 0.0033)};
			\addplot+[dashed,mark=*,mark options={fill=blue},blue] coordinates {(16,0.0163)(32,0.00328)(64,0.0011)(128,0.000964)(256,0.000961)(512,0.00096)};
			\addplot+[mark=triangle*,mark options={fill=red},red] coordinates {(16,0.081)(32,0.015)(64,0.00435)(128,0.0013)(256,0.00054)(512,0.000234)};
			\addplot+[mark=triangle*,mark options={fill=blue},blue] coordinates {(16,0.016)(32,0.0032)(64,0.0011)(128,0.000379)(256,0.000145)(512,0.000063)};
			\addplot+[dotted,no marks,black,domain=110:512]{0.08/x};
			\legend{\small{$s=1$ $(h=1/4, k=10^{-2})$}, \small{$s=0$ $(h=1/4, k=10^{-2})$}, \small{$s=1$ $(h=1/8, k= 5\times 10^{-3})$}, \small{$s=0$ $(h=1/8, k=5\times 10^{-3})$}, \small{order 1 line}}
		\end{axis}
	\end{tikzpicture}
\end{center}
	\caption{Order of convergence with respect to $\epsilon$ in Simulation 2. Theoretical result: Theorem~\ref{the:u eps con u}.}
	\label{fig:order 3d eps}
\end{figure}

%

\subsection{Simulation 3: L-shaped domain}
We set $\mathscr{D}:= [-1,1]^2 \setminus [0,1]^2 \subset \bb{R}^2$, which is a non-convex domain. The coefficients in~\eqref{equ:LLB pro} are taken to be $\kappa_1=0.5, \kappa_2=2.0, \mu=1.0$, and $\gamma=50.0$. The initial data $\bff{u}_0$ is
\begin{equation*}
	\bff{u}_0(x,y)= \big(2x^2,\, 2y,\, x^2-2y^2 \big).
\end{equation*}

{
Four numerical observations are carried out as before:
\begin{enumerate}
    \item[(O1)] The result of this simulation is displayed in Figure~\ref{fig:snapshots field 2d sim5}.

    \item[(O2)] The log-log plot with $h=2^{-j}$, where $j=2,3,\ldots,7$, is shown in Figure~\ref{fig:order sim5} to observe $h\text{-rate}_s^\epsilon$. As seen here, the singularity at the non-convex corner pollutes the finite element solution everywhere in $\mathscr{D}$, leading to $O(h^\alpha)$ and $O(h^{2\alpha})$ order of convergence in $\bb{H}^1$ and $\bb{L}^2$ norms, respectively, for every $\alpha<\alpha_0$. For an L-shaped domain, $\alpha_0=2/3$; see~\eqref{equ:gamma 0}.

    \item[(O3)] To verify $k\text{-rate}^\epsilon_s$, we set $T=0.1$ and compute the reference solution $\bff{u}^\epsilon_{\mathrm{ref}}(T)$ with $h=1/64$ and $k=T/640$. The log-log plot with $k=T/(10\times 2^j)$, where $j=0,1,\ldots,4$, is shown in Figure~\ref{fig:order 2d L temp}.

    \item[(O4)] To observe $\epsilon\text{-rate}_s$, we compute $\bff{u}_\mathrm{ref}(T)$ at $T=0.1$ with $\epsilon=0$, $h=1/128$, and $k=10^{-3}$. Next, with $h=1/64$, $k=2.5\times 10^{-3}$, we compute $\norm{\bff{g}^{\epsilon, (N)}}{\bb{H}^s}$ with $\epsilon=2^{-j}$, for $j=3,4,\ldots, 8$. This procedure is repeated with $h=1/128$ and $k=10^{-3}$. The log-log plot is displayed in Figure~\ref{fig:order L eps}.
\end{enumerate}
}

\begin{figure}[hbt!]
	\centering
	\begin{subfigure}[b]{0.27\textwidth}
		\centering
		\includegraphics[width=\textwidth]{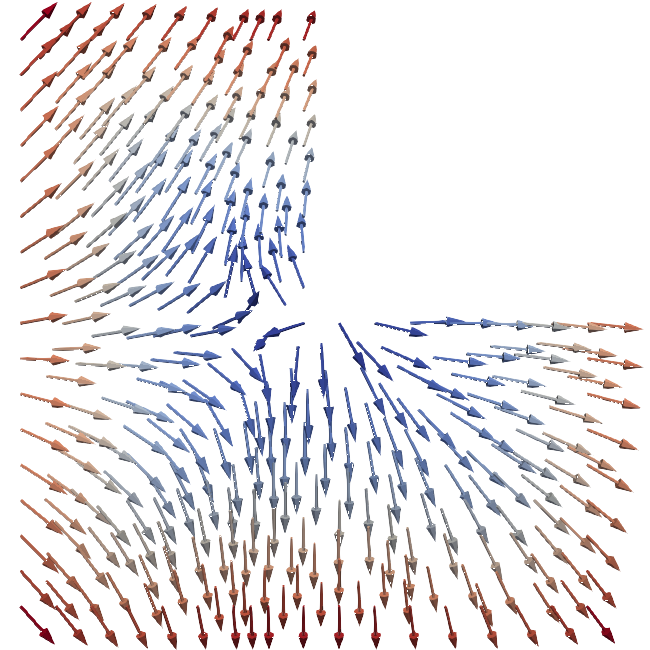}
		\caption{$t=0$}
	\end{subfigure}
	\begin{subfigure}[b]{0.27\textwidth}
		\centering
		\includegraphics[width=\textwidth]{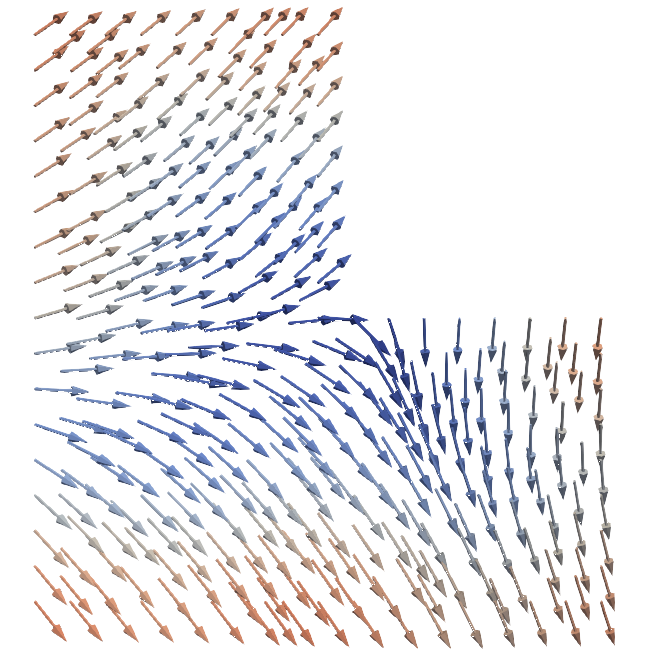}
		\caption{$t=0.025$}
	\end{subfigure}
	\begin{subfigure}[b]{0.27\textwidth}
		\centering
		\includegraphics[width=\textwidth]{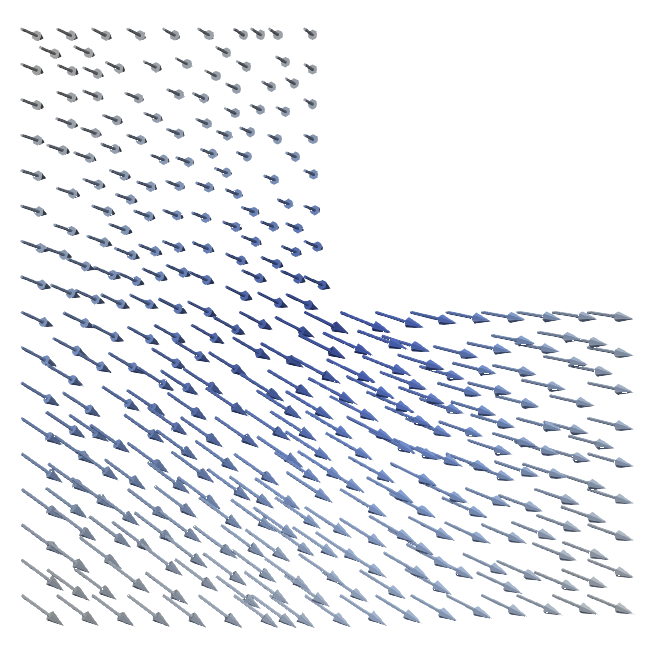}
		\caption{$t=0.1$}
	\end{subfigure}
	\begin{subfigure}[b]{0.1\textwidth}
		\centering
		\includegraphics[width=\textwidth]{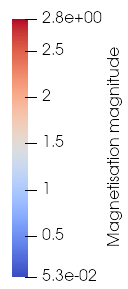}
	\end{subfigure}
	\begin{subfigure}[b]{0.27\textwidth}
		\centering
		\includegraphics[width=\textwidth]{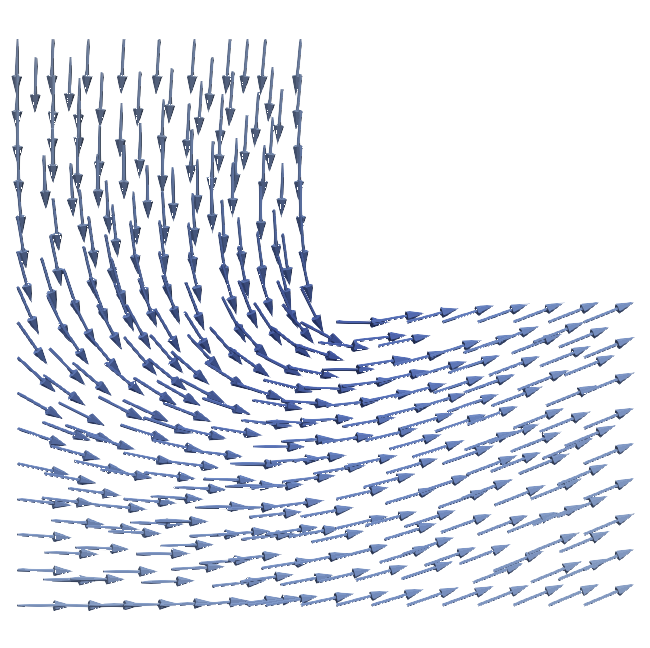}
		\caption{$t=0.2$}
	\end{subfigure}
	\begin{subfigure}[b]{0.27\textwidth}
		\centering
		\includegraphics[width=\textwidth]{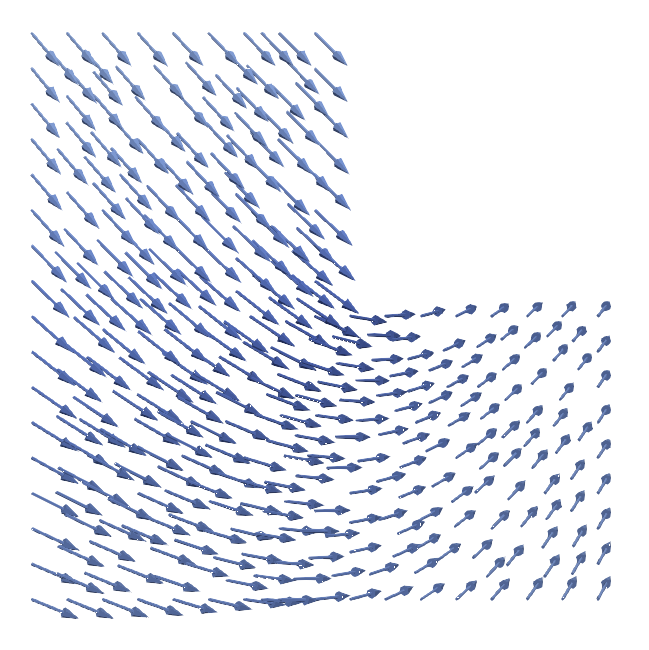}
		\caption{$t=0.3$}
	\end{subfigure}
	\begin{subfigure}[b]{0.27\textwidth}
		\centering
		\includegraphics[width=\textwidth]{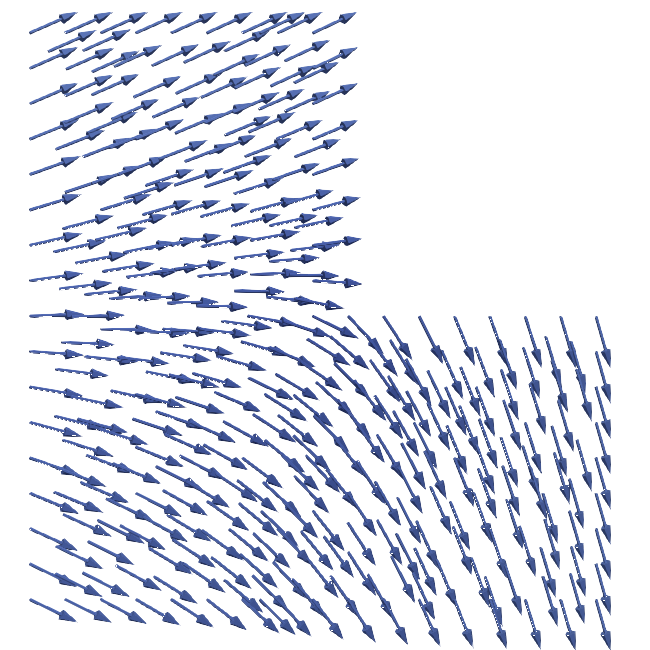}
		\caption{$t=0.5$}
	\end{subfigure}
	\begin{subfigure}[b]{0.1\textwidth}
		\centering
		\includegraphics[width=\textwidth]{sim5_legend.png}
	\end{subfigure}
	\caption{Snapshots of the magnetic spin field $\bff{u}_h^{\epsilon, (n)}$ (projected onto $\bb{R}^2$) at given times in Simulation 3. The colours indicate relative magnitude of the vectors.}
	\label{fig:snapshots field 2d sim5}
\end{figure}

\begin{figure}[hbt!]
	\begin{subfigure}[b]{0.42\textwidth}
		\centering
		\begin{tikzpicture}
			\begin{axis}[
				title=Log-log plot of $\max_n \norm{\bff{e}_h^{\epsilon, (n)}}{\bb{H}^s}$ against $1/h$,
				height=1.45\textwidth,
				width=1\textwidth,
				xlabel= $1/h$,
				ylabel= $\max_n \norm{\bff{e}_h^{\epsilon, (n)}}{\bb{H}^s}$,
				xmode=log,
				ymode=log,
				legend pos=south west,
				legend cell align=left,
				]
				\addplot+[mark=*,red] coordinates {(4,1.95)(8,1.2)(16,0.7)(32,0.4)(64,0.22)(128,0.14)};
				\addplot+[mark=*,blue] coordinates {(4,0.25)(8,0.07)(16,0.022)(32,0.0058)(64,0.0022)(128,0.0008)};
				\addplot+[dashed,no marks,red,domain=30:140]{6.2/(x^2)^(1/3)};
				\addplot+[dashed,no marks,blue,domain=30:140]{1.2/(x^2)^(2/3)};
				\legend{\small{$s=1$}, \small{$s=0$}, \small{order 2/3 line}, \small{order 4/3 line}}
			\end{axis}
		\end{tikzpicture}
		\caption{Spatial order of convergence of Algorithm~\ref{alg:fem eps llbe} in Simulation 3.}
		\label{fig:order sim5}
	\end{subfigure}
	\hspace{0.5cm}
	\begin{subfigure}[b]{0.42\textwidth}
		\centering
		\begin{tikzpicture}
			\begin{axis}[
				title=Log-log plot of $\norm{\bff{f}_k^{\epsilon, (N)}}{\bb{H}^s}$ against $N$,
				height=1.45\textwidth,
				width=1\textwidth,
				xlabel= $N$,
				ylabel= $\norm{\bff{f}_k^{\epsilon, (N)}}{\bb{H}^s}$,
				xmode=log,
				ymode=log,
				legend pos=south west,
				legend cell align=left,
				]
				\addplot+[mark=*,red] coordinates {(10,2.44)(20,2.39)(40,2.00)(80,1.48)(160,0.86)};
				\addplot+[mark=*,blue] coordinates {(10,1.78)(20,1.76)(40,1.48)(80,1.07)(160,0.60)};
				\addplot+[dashed,no marks,black,domain=80:160]{180/x};
				\legend{\small{$s=1$}, \small{$s=0$}, \small{order 1 line}}
			\end{axis}
		\end{tikzpicture}
		\caption{Temporal order of convergence of Algorithm~\ref{alg:fem eps llbe} in Simulation 3.}
		\label{fig:order 2d L temp}
	\end{subfigure}
	\caption{Spatial and temporal orders of convergence of Algorithm~\ref{alg:fem eps llbe} in Simulation 3. Theoretical results: Theorems~\ref{the:err} and~\ref{rem:error}.}
\end{figure}

\begin{figure}[hbt!]
	\begin{center}
		\begin{tikzpicture}
			\begin{axis}[
				title=Log-log plot of $\norm{\bff{g}^{\epsilon, (N)}}{\bb{H}^s}$ against $1/\epsilon$,
				height=0.5\textwidth,
				width=0.5\textwidth,
				xlabel= $1/\epsilon$,
				ylabel= $\norm{\bff{g}^{\epsilon, (N)}}{\bb{H}^s}$,
				xmode=log,
				ymode=log,
				legend pos=outer north east,
				legend cell align=left,
				]
				\addplot+[dashed,mark=*,mark options={fill=red},red] coordinates {(4,3.0)(8,2.09)(16,1.65)(32,1.3)(64,1.14)(128,1.12)(256,1.12)};
				\addplot+[dashed,mark=*,mark options={fill=blue},blue] coordinates {(4,1.0)(8,0.702)(16,0.693)(32,0.664)(64,0.702)(128,0.720)(256,0.750)};
				\addplot+[mark=triangle*,mark options={fill=red},red] coordinates {(4,3.6)(8,1.9)(16,1.3)(32,0.76)(64,0.39)(128,0.2)(256,0.11)};
				\addplot+[mark=triangle*,mark options={fill=blue},blue] coordinates {(4,0.9)(8,0.52)(16,0.31)(32,0.19)(64,0.11)(128,0.069)(256,0.038)};
				\addplot+[dotted,no marks,black,domain=64:256]{14/x};
				\legend{\small{$s=1$ $(h=1/64, k=2.5\times 10^{-3})$}, \small{$s=0$ $(h=1/64, k=2.5\times 10^{-3})$}, \small{$s=1$ $(h=1/128, k= 10^{-3})$}, \small{$s=0$ $(h=1/128, k=10^{-3})$}, \small{order 1 line}}
			\end{axis}
		\end{tikzpicture}
	\end{center}
	\caption{Order of convergence with respect to $\epsilon$ in Simulation 3. Theoretical result: Theorem~\ref{the:u eps con u}.}
	\label{fig:order L eps}
\end{figure}

\subsection{Simulation 4: Domain with Fichera corner}
We take $\mathscr{D}:= [-1,1]^3 \setminus [0,1]^3 \subset \bb{R}^3$, which is a non-convex domain with Fichera corner. The coefficients in~\eqref{equ:LLB pro} are taken to be $\kappa_1=0.5, \kappa_2=2.0, \mu=1.0$, and $\gamma=50.0$. The initial data $\bff{u}_0$ is
\begin{equation*}
	\bff{u}_0(x,y)= \big(2x^2,\, 2z,\, x^2-2y^2 \big).
\end{equation*}

{
Four numerical observations are carried out as before:
\begin{enumerate}
    \item[(O1)] The result of this simulation is displayed in Figure~\ref{fig:snapshots field 3d sim6}.

    \item[(O2)] The log-log plot with $h=2^{-j}$, where $j=2,3,4,5$, is shown in Figure~\ref{fig:order sim6} to observe $h\text{-rate}_s^\epsilon$. While non-convex three-dimensional polyhedral domain is not covered under our error analysis, it appears that an order of convergence similar to that in Simulation 3 is attained here.

    \item[(O3)] To verify $k\text{-rate}^\epsilon_s$, we set $T=0.1$ and compute the reference solution $\bff{u}^\epsilon_{\mathrm{ref}}(T)$ with $h=1/16$ and $k=T/320$. The log-log plot with $k=T/(10\times 2^j)$, where $j=0,1,\ldots,4$, is shown in Figure~\ref{fig:order 3d L temp}.

    \item[(O4)] To observe $\epsilon\text{-rate}_s$, we compute $\bff{u}_\mathrm{ref}(T)$ at $T=0.1$ with $\epsilon=0$, $h=1/25$, and $k=10^{-3}$. Next, with $h=1/8$, $k=2.5\times 10^{-3}$, we compute $\norm{\bff{g}^{\epsilon, (N)}}{\bb{H}^s}$ with $\epsilon=2^{-j}$, for $j=3,4,\ldots, 7$. This is repeated with $h=1/25$ and $k=10^{-3}$. The resulting log-log plot is displayed in Figure~\ref{fig:order 3d L eps}.
\end{enumerate}
}

\begin{figure}[hbt!]
	\centering
	\begin{subfigure}[b]{0.27\textwidth}
		\centering
		\includegraphics[width=\textwidth]{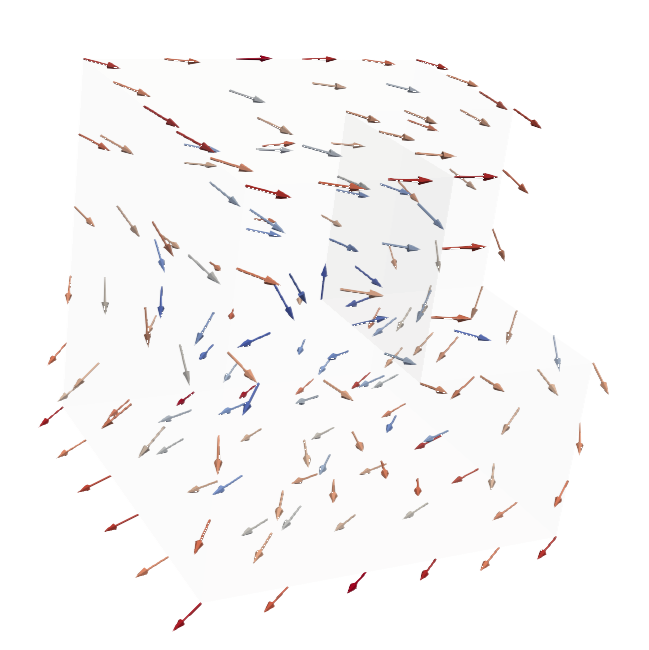}
		\caption{$t=0$}
	\end{subfigure}
	\begin{subfigure}[b]{0.27\textwidth}
		\centering
		\includegraphics[width=\textwidth]{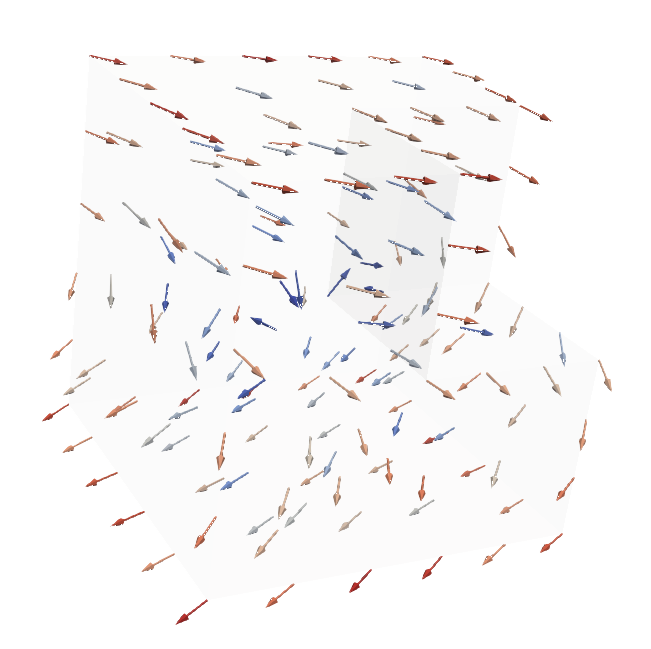}
		\caption{$t=0.025$}
	\end{subfigure}
	\begin{subfigure}[b]{0.27\textwidth}
		\centering
		\includegraphics[width=\textwidth]{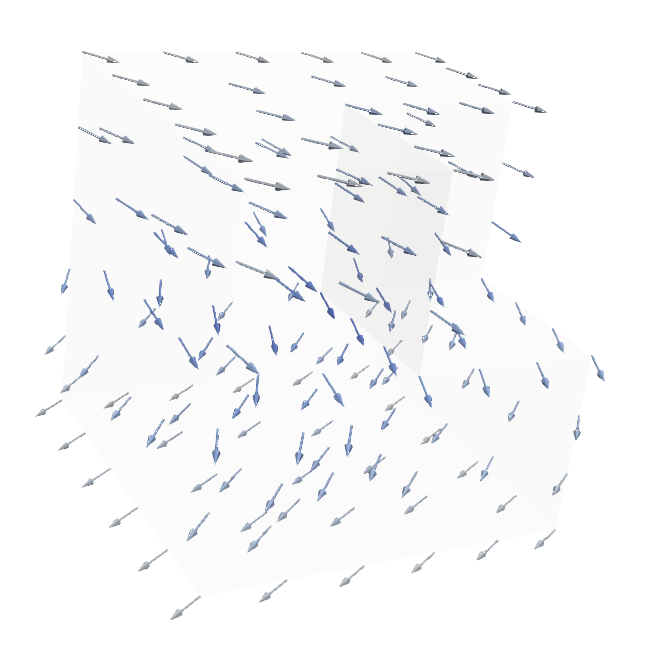}
		\caption{$t=0.1$}
	\end{subfigure}
	\begin{subfigure}[b]{0.1\textwidth}
		\centering
		\includegraphics[width=\textwidth]{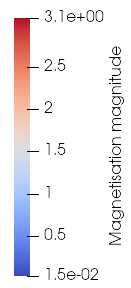}
	\end{subfigure}
	\begin{subfigure}[b]{0.27\textwidth}
		\centering
		\includegraphics[width=\textwidth]{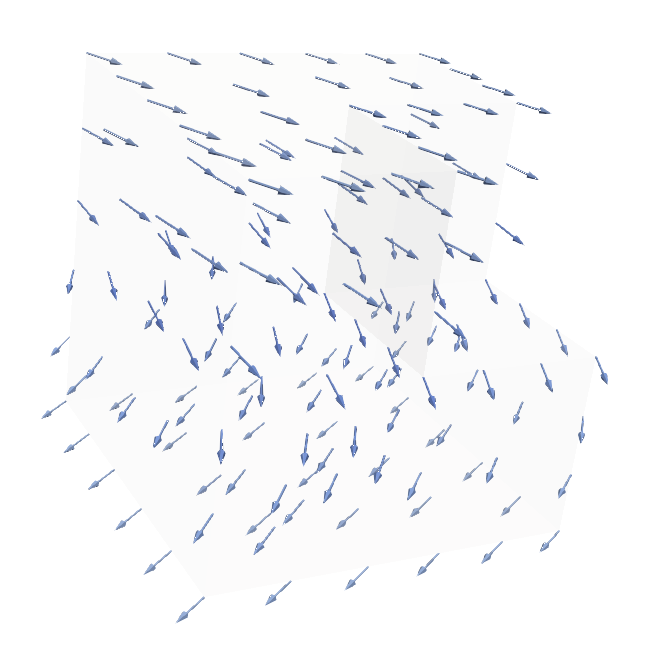}
		\caption{$t=0.2$}
	\end{subfigure}
	\begin{subfigure}[b]{0.27\textwidth}
		\centering
		\includegraphics[width=\textwidth]{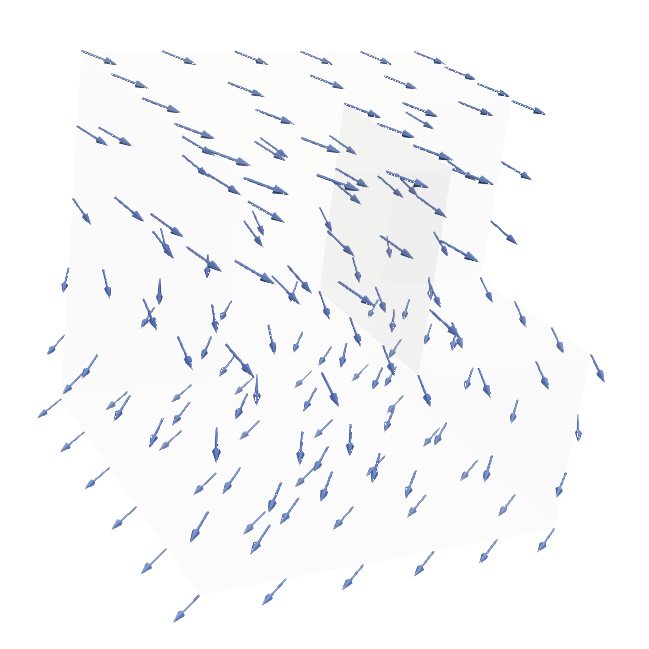}
		\caption{$t=0.3$}
	\end{subfigure}
	\begin{subfigure}[b]{0.27\textwidth}
		\centering
		\includegraphics[width=\textwidth]{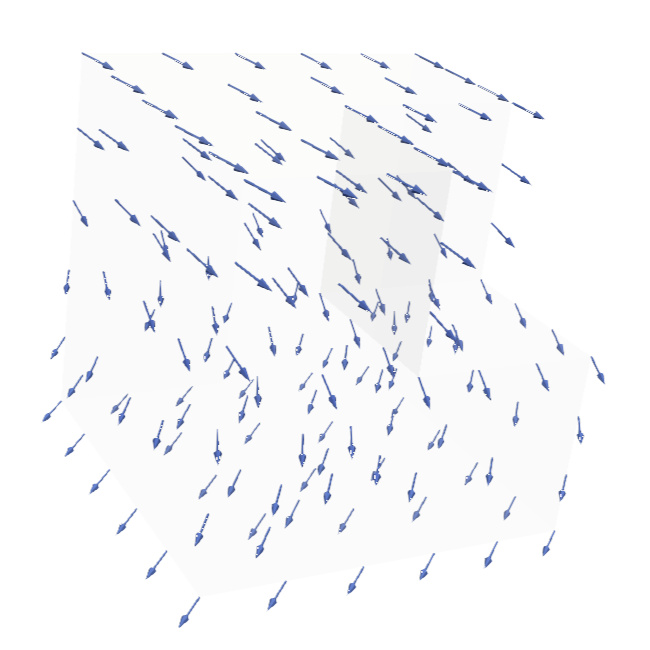}
		\caption{$t=0.5$}
	\end{subfigure}
	\begin{subfigure}[b]{0.1\textwidth}
		\centering
		\includegraphics[width=\textwidth]{sim6_legende.png}
	\end{subfigure}
	\caption{Snapshots of the magnetic spin field $\bff{u}_h^{\epsilon, (n)}$ at given times in Simulation 4. The colours indicate relative magnitude of the vectors.}
	\label{fig:snapshots field 3d sim6}
\end{figure}

\begin{figure}[hbt!]
	\begin{subfigure}[b]{0.42\textwidth}
		\centering
		\begin{tikzpicture}
			\begin{axis}[
				title=Log-log plot of $\max_n \norm{\bff{e}_h^{\epsilon, (n)}}{\bb{H}^s}$,
				height=1.4\textwidth,
				width=1\textwidth,
				xlabel= $1/h$,
				ylabel= $\max_n \norm{\bff{e}_h^{\epsilon, (n)}}{\bb{H}^s}$,
				xmode=log,
				ymode=log,
				legend pos=south west,
				legend cell align=left,
				]
				\addplot+[mark=*,red] coordinates {(4,2.23)(8,1.76)(16,1.19)(32,0.7)};
				\addplot+[mark=*,blue] coordinates {(4,0.34)(8,0.22)(16,0.1)(32,0.04)};
				\addplot+[dashed,no marks,red,domain=10:32]{5/(x^2)^(1/3)};
				\addplot+[dashed,no marks,blue,domain=10:32]{6/(x^2)^(2/3)};
				\legend{\small{$s=1$}, \small{$s=0$}, \small{order 2/3 line}, \small{order 4/3 line}}
			\end{axis}
		\end{tikzpicture}
		\caption{Spatial order of convergence of Algorithm~\ref{alg:fem eps llbe} in Simulation 4.}
		\label{fig:order sim6}
	\end{subfigure}
	\hspace{0.5cm}
	\begin{subfigure}[b]{0.42\textwidth}
		\centering
		\begin{tikzpicture}
			\begin{axis}[
				title=Log-log plot of $\norm{\bff{f}_k^{\epsilon, (N)}}{\bb{H}^s}$ against $N$,
				height=1.4\textwidth,
				width=1\textwidth,
				xlabel= number of time steps $N$,
				ylabel= $\norm{\bff{f}_k^{\epsilon, (N)}}{\bb{H}^s}$ ,
				xmode=log,
				ymode=log,
				legend pos=south west,
				legend cell align=left,
				]
				\addplot+[mark=*,red] coordinates {(10,0.38)(20,0.28)(40,0.22)(80,0.17)(160,0.12)};
				\addplot+[mark=*,blue] coordinates {(10,0.085)(20,0.061)(40,0.048)(80,0.038)(160,0.025)};
				\addplot+[dashed,no marks,black,domain=60:160]{7/x};
				\legend{\small{$s=1$}, \small{$s=0$}, \small{order 1 line}}
			\end{axis}
		\end{tikzpicture}
		\caption{Temporal order of convergence of Algorithm~\ref{alg:fem eps llbe} in Simulation 4.}
		\label{fig:order 3d L temp}
	\end{subfigure}
	\caption{Spatial and temporal orders of convergence of Algorithm~\ref{alg:fem eps llbe} in Simulation 4. Theoretical results: Theorems~\ref{the:err} and~\ref{rem:error}.}
\end{figure}

\begin{figure}[hbt!]
	\begin{center}
		\begin{tikzpicture}
			\begin{axis}[
				title=Log-log plot of $\norm{\bff{g}^{\epsilon, (N)}}{\bb{H}^s}$ against $1/\epsilon$,
				height=0.5\textwidth,
				width=0.5\textwidth,
				xlabel= $1/\epsilon$,
				ylabel= $\norm{\bff{g}^{\epsilon, (N)}}{\bb{H}^s}$,
				xmode=log,
				ymode=log,
				legend pos=outer north east,
				legend cell align=left,
				]
				\addplot+[dashed,mark=*,mark options={fill=red},red] coordinates {(8,3.19)(16,2.51)(32,1.8)(64,1.26)(128,1.02)};
				\addplot+[dashed,mark=*,mark options={fill=blue},blue] coordinates {(8,0.52)(16,0.37)(32,0.26)(64,0.19)(128,0.17)};
				\addplot+[mark=triangle*,mark options={fill=red},red] coordinates {(8,3.1)(16,2.22)(32,1.34)(64,0.69)(128,0.34)};
				\addplot+[mark=triangle*,mark options={fill=blue},blue] coordinates {(8,0.49)(16,0.32)(32,0.2)(64,0.1)(128,0.05)};
				\addplot+[dotted,no marks,black,domain=30:128]{16/x};
				\legend{\small{$s=1$ $(h=1/8, k=2.5\times 10^{-3})$}, \small{$s=0$ $(h=1/8, k=2.5\times 10^{-3})$}, \small{$s=1$ $(h=1/25, k=10^{-3})$}, \small{$s=0$ $(h=1/25, k=10^{-3})$}, \small{order 1 line}}
			\end{axis}
		\end{tikzpicture}
	\end{center}
	\caption{Order of convergence with respect to $\epsilon$ in Simulation 4. Theoretical result: Theorem~\ref{the:u eps con u}.}
	\label{fig:order 3d L eps}
\end{figure}

\section*{Acknowledgements}
The authors are partially supported by the Australian Research Council under grant
number DP190101197, DP200101866, and DP220100937.
Agus L. Soenjaya is also supported by the Australian Government Research
Training Program (RTP) Scholarship awarded at the University of New South Wales, Sydney.
{We would like to thank the referees for their valuable feedback which improves the quality of the manuscript.}

\appendix
\section{Auxiliary results}\label{sec:aux}

Some auxiliary results which are needed in the analysis are collected in this section.
Firstly, if~$\bff{u}\in \bb{H}^1_\Delta \cap \bb{H}^{1+\alpha}$ and $\bff{v}\in \bb{H}^{\alpha}$ (where $\alpha$ satisfies \eqref{equ:condition alpha}),
and~$\pa \bff{u}/\pa\bff{\nu}=0$ on $\pa\mathscr{D}$ in the sense of trace (where~$\bff{\nu}$ is the outward normal vector
to~$\pa \mathscr{D}$), then we can use integration-by-parts formula to obtain
\begin{align}\label{equ:vec Gre 1}
		-\inpro{\Delta\bff{u}}{\bff{v}}_{\bb{L}^2}
		&=
		-\sum_{i=1}^3
		\inpro{\Delta u_i}{v_i}_{L^2}
		=
		\sum_{i=1}^3
		\inpro{\nabla u_i}{\nabla v_i}_{\bb{L}^2}
		=
		\inpro{\nabla\bff{u}}{\nabla\bff{v}}_{\bb{L}^2}.
\end{align}

Moreover, if $\bff{v}\in\bb{H}^{1+\alpha}$, then by using the well-known fact~\citep{BreMir01} that $\bb{H}^s(\mathscr{D})$ is an algebra under multiplication if~$s>d/2$, we have $\bff{v}\times\bff{u}\in \bb{H}^{1+\alpha}$. Hence from \eqref{equ:vec Gre 1}, we infer
\begin{align}
        \label{equ:vec Gre 2}
		\inpro{\bff{u}\times\Delta\bff{u}}{\bff{v}}_{\bb{L}^2}
		&=
		\inpro{\Delta\bff{u}}{\bff{v}\times\bff{u}}_{\bb{L}^2}
		\stackrel{\eqref{equ:vec Gre 1}}{=}
		-
		\inpro{\nabla\bff{u}}{\nabla(\bff{v}\times\bff{u})}_{\bb{L}^2}
		=
		-
		\inpro{\bff{u}\times\nabla\bff{u}}{\nabla\bff{v}}_{\bb{L}^2}.
\end{align}

For a bounded Lipschitz domain, certain embedding theorems and the Gagliardo--Nirenberg inequalities hold (see~\citep{BreMir18, BreMir19, DiPalVal12}), which we will use extensively in the analysis.

Next, we gather some elliptic regularity (and the associated regularity shift) theorems for bounded polytopal domains. For a polygonal domain, refer to~\citep{Dau92, Li22}. For a polyhedral domain, see~\citep{Dau92, JerKen89}.

\begin{theorem}[two-dimensional polygonal domains]\label{the:ell reg polyg}
Suppose that $\mathscr{D}$ is a bounded polygonal domain which has $m$ re-entrant corners with opening angles $\omega_j \in (\pi,2\pi)$, for $j=1,2,\ldots,m$. Let $\eta_j=\pi/\omega_j$ so that $\eta_j\in (\frac12,1)$, and let $\eta_0 := \min_j \eta_j$. Let $\bff{f}\in \widetilde{\bb{H}}^{-1+\eta_0}$ be given. Then for any $\eta \in (0, \eta_0)$, there exists a unique weak solution $\bff{v}\in \bb{H}^{1+\eta}$ to the elliptic boundary value problem
\begin{subequations}\label{equ:elliptic}
	\begin{alignat}{2}
		-\Delta \bff{v}+ \bff{v} &=\bff{f}
		\qquad&
		&\text{in } \mathscr{D},
		\label{equ:elliptic equ}
		\\
		\frac{\pa\bff{v}}{\pa\bff{\nu}} &= 0
		\qquad&
		&\text{on } \pa \mathscr{D},
		\label{equ:elliptic equ bou con}
	\end{alignat}
\end{subequations}
satisfying the regularity shift estimate
\begin{align}\label{equ:shift polyg H1s}
	\norm{\bff{v}}{\bb{H}^{1+\eta}} \leq c \left( \norm{\bff{v}}{\widetilde{\bb{H}}^{-1+\eta}} + \norm{\Delta \bff{v}}{\widetilde{\bb{H}}^{-1+\eta}} \right). 
\end{align}
\end{theorem}

\begin{theorem}[convex polyhedral domain]\label{the:ell reg polyhed}
Suppose that $\mathscr{D}$ is a \emph{convex} polyhedral domain. 
If $\bff{f}\in \bb{L}^p$ for some $p\in (1,2]$, then there exists a unique weak solution $\bff{v}\in \bb{W}^{2,p}$ to the same problem, satisfying
\begin{align}\label{equ:shift polyh W2p}
	\norm{\bff{v}}{\bb{W}^{2,p}} \leq c \left( \norm{\bff{v}}{\bb{L}^{p}} + \norm{\Delta \bff{v}}{\bb{L}^{p}} \right).
\end{align}
Here, the constant $c$ depends only on the domain $\mathscr{D}$.
\end{theorem}

We also need the following technical lemma.

\begin{lemma}\label{lem:tec lem}
	Let $d$ be the dimension of the domain $\mathscr{D}$. For any $\bff{v}\in \bb{H}^1$, $\bff{w}\in \bb{W}^{1,2d}$, and $\delta>0$, the
	following inequality holds:
	\[
	\Big|
	\inpro{\bff{v}\times\nabla\bff{w}}{\nabla\bff{v}}_{\bb{L}^2}
	\Big|
	\le
	\Phi(\bff{w})
	\norm{\bff{v}}{\bb{L}^2}^2
	+
	\delta \norm{\nabla\bff{v}}{\bb{L}^2}^2
	\]
	where
	\begin{equation}\label{equ:Phi}
		\Phi(\bff{w})
		:=
		\begin{cases}
			\frac{c^2}{\delta}
			\norm{\nabla\bff{w}}{\bb{L}^2}^2
			+
			\frac{c^2}{4\delta^3}
			\norm{\nabla\bff{w}}{\bb{L}^2}^4,
			\quad & d=1,
			\\[1ex]
			\frac{3^3 c^4}{2^8\delta^3}
			\norm{\nabla\bff{w}}{\bb{L}^4}^4 
			+ \delta,
			\quad & d=2,
			\\[1ex]
			\frac{3^3 c^4}{2^8\delta^3}
			\norm{\nabla\bff{w}}{\bb{L}^6}^4 
			+ \delta,
			\quad & d=3.
		\end{cases}
	\end{equation}
	The positive constant~$c$ (given by the Gagliardo--Nirenberg inequalities) depends
	only on the spatial domain $\mathscr{D}$.
\end{lemma}
\begin{proof}
	We first prove the result for~$d=1$. By using H\"older's inequality, we have
	\[
	\left|\inpro{\bff{v}\times\nabla\bff{w}}{\nabla\bff{v}}\right|
	\le
	\norm{\bff{v}}{\bb{L}^\infty}
	\norm{\nabla\bff{w}}{\bb{L}^2} 
	\norm{\nabla\bff{v}}{\bb{L}^2}.
	\]
	The following Gagliardo--Nirenberg interpolation inequality
	\[
	\norm{\bff{v}}{\bb{L}^\infty}
	\le
	c
	\norm{\bff{v}}{\bb{L}^2}^{1/2}
	\norm{\bff{v}}{\bb{H}^1}^{1/2}
	\]
	implies
	\begin{align}\label{equ:v nab w}
		\left|\inpro{\bff{v}\times\nabla\bff{w}}{\nabla\bff{v}}_{\bb{L}^2}\right|
		&\le
		c
		\norm{\bff{v}}{\bb{L}^2}^{1/2}
		\left(
		\norm{\bff{v}}{\bb{L}^2}^{1/2}
		+
		\norm{\nabla\bff{v}}{\bb{L}^2}^{1/2}
		\right)
		\norm{\nabla\bff{w}}{\bb{L}^2}
		\norm{\nabla\bff{v}}{\bb{L}^2}
		\nn\\
		&=
		c
		\norm{\bff{v}}{\bb{L}^2}
		\norm{\nabla\bff{w}}{\bb{L}^2}
		\norm{\nabla\bff{v}}{\bb{L}^2}
		+
		c
		\norm{\bff{v}}{\bb{L}^2}^{1/2}
		\norm{\nabla\bff{w}}{\bb{L}^2}
		\norm{\nabla\bff{v}}{\bb{L}^2}^{3/2}.
	\end{align}
	The first term on the right-hand side of~\eqref{equ:v nab w} can be estimated by
	\[
	c\norm{\bff{v}}{\bb{L}^2}
	\norm{\nabla\bff{w}}{\bb{L}^2}
	\norm{\nabla\bff{v}}{\bb{L}^2}
	\le
	\frac{c^2}{\delta}
	\norm{\bff{v}}{\bb{L}^2}^2
	\norm{\nabla\bff{w}}{\bb{L}^2}^2
	+
	\frac{\delta}{4}
	\norm{\nabla\bff{v}}{\bb{L}^2}^2.
	\]
	For the second term on the right-hand side of~\eqref{equ:v nab w} we use 
	Young's inequality 
	\[
	ab \le \frac{a^p}{p} + \frac{b^q}{q} \quad \text{where} \quad\frac{1}{p}+\frac{1}{q} = 1,
	\]
	with
	\[
	a=\delta^{-3/4} c \norm{\bff{v}}{\bb{L}^2}^{1/2}\norm{\nabla\bff{w}}{\bb{L}^2}, \quad 
	b=\delta^{3/4}\norm{\nabla\bff{v}}{\bb{L}^2}^{3/2}, \quad
	p=4,
	\quad\text{and}\quad
	q = 4/3,
	\]
	to obtain
	\[
	c \norm{\bff{v}}{\bb{L}^2}^{1/2}
	\norm{\nabla\bff{w}}{\bb{L}^2}
	\norm{\nabla\bff{v}}{\bb{L}^2}^{3/2}
	\le
	\frac{c^2}{4\delta^3}
	\norm{\bff{v}}{\bb{L}^2}^2
	\norm{\nabla\bff{w}}{\bb{L}^2}^4
	+
	\frac{3\delta}{4}
	\norm{\nabla\bff{v}}{\bb{L}^2}^2.
	\]
	Thus,~\eqref{equ:v nab w} gives
	\begin{align*}
		\left|\inpro{\bff{v}\times\nabla\bff{w}}{\nabla\bff{v}}_{\bb{L}^2}\right|
		&\le
		c^2 \left(
		\frac{1}{\delta}
		\norm{\nabla\bff{w}}{\bb{L}^2}^2
		+
		\frac{1}{4\delta^3}
		\norm{\nabla\bff{w}}{\bb{L}^2}^4
		\right)
		\norm{\bff{v}}{\bb{L}^2}^2
		+
		\delta
		\norm{\nabla\bff{v}}{\bb{L}^2}^2.
	\end{align*}

	Next we consider the case when~$d=2$. Using H\"older's inequality again
	yields
	\[
	\left|\inpro{\bff{v}\times\nabla\bff{w}}{\nabla\bff{v}}\right|
	\le
	\norm{\bff{v}}{\bb{L}^4}
	\norm{\nabla\bff{w}}{\bb{L}^4} 
	\norm{\nabla\bff{v}}{\bb{L}^2}.
	\]
	The Gagliardo--Nirenberg interpolation inequality
	\[
	\norm{\bff{v}}{\bb{L}^4}
	\le
	c\norm{\bff{v}}{\bb{L}^2}^{1/2}\norm{\bff{v}}{\bb{H}^1}^{1/2}
	\]
	yields
	\begin{align*}
		\left|\inpro{\bff{v}\times\nabla\bff{w}}{\nabla\bff{v}}\right|
		&\le
		c
		\norm{\nabla\bff{w}}{\bb{L}^4}
		\norm{\bff{v}}{\bb{L}^2}^{1/2}
		\norm{\bff{v}}{\bb{H}^1}^{1/2}
		\norm{\nabla\bff{v}}{\bb{L}^2}
		\\
		&\le
		c
		\norm{\nabla\bff{w}}{\bb{L}^4}
		\norm{\bff{v}}{\bb{L}^2}^{1/2}
		\norm{\bff{v}}{\bb{H}^1}^{3/2}.
	\end{align*}
	Using Young's inequality again with
	\begin{alignat*}{2}
		a
		&=
		\left(\frac{4\delta}{3}\right)^{-3/4}
		c
		\norm{\nabla\bff{w}}{\bb{L}^4}
		\norm{\bff{v}}{\bb{L}^2}^{1/2}, \qquad
		&
		p &= 4,
		\\
		b
		&=
		\left(\frac{4\delta}{3}\right)^{3/4}
		\norm{\bff{v}}{\bb{H}^1}^{3/2}, \quad
		&
		q &= 4/3,
	\end{alignat*}
	we deduce 
	\begin{align*}
		\left|\inpro{\bff{v}\times\nabla\bff{w}}{\nabla\bff{v}}\right|
		&\le
		\frac{3^3}{4^4 \delta^3} c^4
		\norm{\nabla\bff{w}}{\bb{L}^4}^4
		\norm{\bff{v}}{\bb{L}^2}^{2}
		+
		\delta
		\norm{\bff{v}}{\bb{H}^1}^{2}
		\\
		&=
		\left(
		\frac{3^3}{4^4 \delta^3} c^4
		\norm{\nabla\bff{w}}{\bb{L}^4}^4
		+ \delta
		\right)
		\norm{\bff{v}}{\bb{L}^2}^{2}
		+
		\delta
		\norm{\nabla\bff{v}}{\bb{L}^2}^{2}.
	\end{align*}
	
	Finally, we consider the case~$d=3$. By H\"older's inequality, we have 
	\[
	\left|\inpro{\bff{v}\times\nabla\bff{w}}{\nabla\bff{v}}\right|
	\le
	\norm{\bff{v}}{\bb{L}^3}
	\norm{\nabla\bff{w}}{\bb{L}^6} 
	\norm{\nabla\bff{v}}{\bb{L}^2}.
	\]
	The Gagliardo--Nirenberg inequality
	\[
	\norm{\bff{v}}{\bb{L}^3}
	\le
	c
	\norm{\bff{v}}{\bb{L}^2}^{1/2} 
	\norm{\bff{v}}{\bb{H}^1}^{1/2}
	\]
	and H\"older's inequality yield
	\begin{align*} 
		\left|\inpro{\bff{v}\times\nabla\bff{w}}{\nabla\bff{v}}\right|
		&\le
		c
		\norm{\nabla\bff{w}}{\bb{L}^6}^{1/2}
		\norm{\bff{v}}{\bb{L}^2}^{1/2}
		\norm{\bff{v}}{\bb{H}^1}^{3/2}.
	\end{align*}
	Using Young's inequality, by the same argument as in the case $d=2$, we obtain
	\begin{align*}
		\left|\inpro{\bff{v}\times\nabla\bff{w}}{\nabla\bff{v}}\right|
		\leq
		\left(
		\frac{3^3}{4^4 \delta^3} c^4
		\norm{\nabla\bff{w}}{\bb{L}^6}^4
		+ \delta
		\right)
		\norm{\bff{v}}{\bb{L}^2}^{2}
		+
		\delta
		\norm{\nabla\bff{v}}{\bb{L}^2}^{2}.
	\end{align*}
	completing the proof of the lemma.
\end{proof}

\section{Further regularity of the solution}\label{sec:further regular proof}

In this section, we develop further a priori estimates on the approximate solution $\bff{u}_n^\epsilon$ to the $\epsilon$-LLBE defined in \eqref{equ:Gal LLB eps}. Our aim is to derive estimates for $\bff{u}_n^\epsilon$ which are uniform in $\epsilon\in [0,1)$, but which possibly hold only locally in time, i.e. on $[0,T^\ast]$, where~$T^\ast$ is given by
    \begin{equation}\label{equ:T ast llbe}
		\begin{cases}
			T^\ast = T, \quad & \text{if } d=1,2,
			\\
			T^\ast \le T, \quad & \text{if } d=3.
		\end{cases}
	\end{equation}
In particular, these estimates also imply further regularity properties for the solution of the $\epsilon$-LLBE and the LLBE.

\begin{lemma}\label{lem:ene est 2}
	Assume that~\eqref{equ:u0 eps con} and~\eqref{equ:u0n H2} hold for $s=\max\{1, d-1\}$.
	For any~$t\in[0,T^\ast]$, where $T^\ast$ is given by \eqref{equ:T ast llbe}, there exists a constant $c$ such that for every $n\in\bb{N}$ and $\epsilon\in [0,1)$, 
	\begin{align*}
		\norm{\pa_t\bff{u}_n^\epsilon(t)}{\bb{L}^2}^2
		&+
		\epsilon \norm{\nabla \pa_t\bff{u}_n^\epsilon(t)}{\bb{L}^2}^2
		+
		\int_0^t
		\norm{\pa_t\bff{u}_n^\epsilon(s)}{\bb{H}^1}^2 \ds
		\le c.
	\end{align*}
	The constant~$c$ is independent of~$n$ and~$\epsilon$, but may depend on~$T^\ast$.
\end{lemma}
\begin{proof}
	Differentiating both sides of~\eqref{equ:LLB eps wea equ} with respect 
	to~$t$ we obtain, for all~$\bff{v}\in\mathcal{S}_n$,
	\begin{align}\label{equ:dt2 une}
		\inpro{\pa^2_t\bff{u}_n^\epsilon(t)}{\bff{v}}_{\bb{L}^2}
		&+
		\epsilon \inpro{\nabla\pa^2_t\bff{u}_n^\epsilon(t)}{\nabla\bff{v}}_{\bb{L}^2}
		+
		\kappa_1 \inpro{\nabla\pa_t\bff{u}_n^\epsilon(t)}{\nabla\bff{v}}_{\bb{L}^2}
		+
		\kappa_2
		\inpro{\pa_t\bff{u}_n^\epsilon(t)}{\bff{v}}_{\bb{L}^2}
		\nn\\
		&=
		\gamma
		\inpro{\pa_t\bff{u}_n^\epsilon(t)\times\Delta\bff{u}_n^\epsilon(t)}{\bff{v}}_{\bb{L}^2}
		+
		\gamma
		\inpro{\bff{u}_n^\epsilon(t)\times\Delta\pa_t\bff{u}_n^\epsilon(t)}{\bff{v}}_{\bb{L}^2}
		\nn\\
		&\quad
		-
		\kappa_2\mu
		\inpro{\pa_t(|\bff{u}_n^\epsilon(t)|^2)\bff{u}_n^\epsilon(t)}{\bff{v}}_{\bb{L}^2}
		-
		\kappa_2\mu
		\inpro{|\bff{u}_n^\epsilon(t)|^2\pa_t\bff{u}_n^\epsilon(t)}{\bff{v}}_{\bb{L}^2}
	\end{align}
	where~$\pa_t^2:=\pa^2/\pa t^2$.
	Choosing~$\bff{v}=2\pa_t\bff{u}^\epsilon_n$ and integrating over~$(0,t)$, we
	deduce, after rearranging the terms,
	\begin{align}\label{equ:dt un nab dt}
		\nn
		&\norm{\pa_t\bff{u}^\epsilon_n(t)}{\bb{L}^2}^2
		+
		\epsilon \norm{\nabla\pa_t\bff{u}^\epsilon_n(t)}{\bb{L}^2}^2
		+ 2\kappa_1 \int_0^t \norm{\nabla \pa_t\bff{u}^\epsilon_n(s)}{\bb{L}^2}^2\ds
		+ 2\kappa_2 \int_0^t \norm{\pa_t\bff{u}^\epsilon_n(s)}{\bb{L}^2}^2\ds
		\\
		\nn
		&\quad
		+ \kappa_2\mu \int_0^t \norm{\pa_t(|\bff{u}^\epsilon_n(s)|^2)}{\bb{L}^2}^2\ds
		+2\kappa_2\mu \int_0^t \norm{|\bff{u}^\epsilon_n(s)||\pa_t\bff{u}^\epsilon_n(s)|}{\bb{L}^2}^2\ds
		\\
		&=
		\norm{\pa_t\bff{u}^\epsilon_n(0)}{\bb{L}^2}^2
		+
		\epsilon \norm{\nabla\pa_t\bff{u}^\epsilon_n(0)}{\bb{L}^2}^2
		-
		2\gamma\int_0^t
		\inpro{\pa_t\bff{u}^\epsilon_n(s)\times\nabla\bff{u}^\epsilon_n(s)}%
		{\nabla\pa_t\bff{u}^\epsilon_n(s)}\ds.
	\end{align}
	We use Lemma~\ref{lem:tec
		lem} with~$\bff{v}=\pa_t \bff{u}_n^\epsilon$,
	$\bff{w}=\bff{u}_n^\epsilon$, and $\delta=\kappa_1/4\gamma$ to estimate the last term on
	the right-hand side, and thus obtain
	\begin{align}\label{equ:dtune 2}
		\norm{\pa_t\bff{u}^\epsilon_n(t)}{\bb{L}^2}^2
		&+ \epsilon \norm{\nabla\pa_t\bff{u}^\epsilon_n(t)}{\bb{L}^2}^2
		+ \int_0^t \norm{\pa_t\bff{u}^\epsilon_n(s)}{\bb{H}^1}^2\ds
		\nn\\
		&\quad
		+ \int_0^t \norm{\pa_t(|\bff{u}^\epsilon_n(s)|^2)}{\bb{L}^2}^2\ds
		+ \int_0^t \norm{|\bff{u}^\epsilon_n(s)||\pa_t\bff{u}^\epsilon_n(s)|}{\bb{L}^2}^2\ds
		\nn\\
		&\lesssim 
		\norm{\pa_t\bff{u}^\epsilon_n(0)}{\bb{L}^2}^2
		+
		\epsilon \norm{\nabla\pa_t\bff{u}^\epsilon_n(0)}{\bb{L}^2}^2
		+
		\int_0^t
		\Phi(\bff{u}_n^\epsilon(s))
		\norm{\pa_t\bff{u}_n^\epsilon(s)}{\bb{L}^2}^{2}
		\ds,
	\end{align}
	where the constant is independent of~$\epsilon$ and~$n$. 
	
	Now, for $d=1,2$, we repeat the arguments leading to \eqref{equ:Phi un less}. For $d=3$, we use the embedding $\bb{H}^2\hookrightarrow \bb{W}^{1,6}$ and the $\bb{H}^2$ elliptic regularity to obtain
	\begin{align*}
		\Phi(\bff{u}_n^\epsilon(t)) 
		\leqs 1+ \norm{\bff{u}_n^\epsilon(t)}{\bb{H}^2}^4
		&\leqs
		1+ \norm{\bff{u}_n^\epsilon(t)}{\bb{L}^2}^4 + \norm{\Delta \bff{u}_n^\epsilon(t)}{\bb{L}^2}^4
		\\
		&\leq 
		1+ \norm{\bff{u}_n^\epsilon(t)}{\bb{L}^2}^4 + \norm{\nabla \bff{u}_n^\epsilon(t)}{\bb{L}^2}^2 \norm{\nabla\Delta \bff{u}_n^\epsilon(t)}{\bb{L}^2}^2
		\\
		&\leqs
		1+ \norm{\nabla\Delta \bff{u}_n^\epsilon(t)}{\bb{L}^2}^2,
	\end{align*}	
	where in the last step we used Lemma~\ref{lem:ene est 1}.
	Altogether, we have
	\begin{equation}\label{equ:Phi une less 123}
		\Phi(\bff{u}_n^\epsilon(t))
		\lesssim
		\begin{cases}
			1, \quad & d=1,
			\\
			1 + \norm{\Delta\bff{u}_n^\epsilon(t)}{\bb{L}^2}^2,
			\quad & d=2,
			\\
			1 + \norm{\nabla\Delta\bff{u}_n^\epsilon(t)}{\bb{L}^2}^2,
			\quad & d=3.
		\end{cases}
	\end{equation}
	Noting \eqref{equ:Phi une less 123}, we continue from~\eqref{equ:dtune 2} and apply the Gronwall inequality. For $d=1,2$ we use Lemma~\ref{lem:ene est 1}, while for $d=3$ we use Lemma~\ref{lem:nab Lap une} (which needs \eqref{equ:u0 eps con} and~\eqref{equ:u0n H2} to hold for $s=2$) to deduce that
	\begin{align}\label{equ:pat une}
		\norm{\pa_t\bff{u}^\epsilon_n(t)}{\bb{L}^2}^2
		+ \epsilon \norm{\nabla\pa_t\bff{u}^\epsilon_n(t)}{\bb{L}^2}^2
		+ \int_0^t \norm{\pa_t\bff{u}^\epsilon_n(s)}{\bb{H}^1}^2\ds
		\lesssim
		\norm{\pa_t\bff{u}^\epsilon_n(0)}{\bb{L}^2}^2
		+
		\epsilon \norm{\nabla\pa_t\bff{u}^\epsilon_n(0)}{\bb{L}^2}^2,
	\end{align}
	for~$t\in[0,T^\ast]$.
	It remains to show the boundedness of~$\norm{\pa_t\bff{u}^\epsilon_n(0)}{\bb{L}^2}^2
	+ \epsilon \norm{\nabla\pa_t\bff{u}^\epsilon_n(0)}{\bb{L}^2}^2$ in~\eqref{equ:pat une}.
	Letting~$t\to0^+$ in~\eqref{equ:LLB eps wea equ},
	choosing~$\bff{v}=\pa_t\bff{u}_n^\epsilon(0)$, and rearranging the resulting equation, we
	deduce after using H\" older's inequality
	\begin{align}\label{equ:dt une 0}
		&\norm{\pa_t\bff{u}_n^\epsilon(0)}{\bb{L}^2}^2
		+
		\epsilon
		\norm{\nabla\pa_t\bff{u}_n^\epsilon(0)}{\bb{L}^2}^2
		\nn \\
		&=
		\inpro{\kappa_1 \Delta\bff{u}_{0,n}^\epsilon
			+
			\gamma
			\bff{u}_{0,n}^\epsilon\times\Delta\bff{u}_{0,n}^\epsilon
			-\kappa_2
			(1+\mu|\bff{u}_{0,n}^\epsilon|^2)\bff{u}_{0,n}^\epsilon}
		{\pa_t\bff{u}_n^\epsilon(0)}_{\bb{L}^2}
		\nn \\
		&\lesssim
		\left(\norm{\Delta\bff{u}_{0,n}^\epsilon}{\bb{L}^2}
		+
		\norm{\bff{u}_{0,n}^\epsilon}{\bb{L}^\infty}
		\norm{\Delta\bff{u}_{0,n}^\epsilon}{\bb{L}^2}
		+
		\norm{\bff{u}_{0,n}^\epsilon}{\bb{L}^2}^2
		+
		\norm{\bff{u}_{0,n}^\epsilon}{\bb{L}^\infty}^2
		\norm{\bff{u}_{0,n}^\epsilon}{\bb{L}^2}
		\right)
		\norm{\pa_t\bff{u}_n^\epsilon(0)}{\bb{L}^2}.
	\end{align}
	The embedding~$\bb{H}^1_\Delta \hookrightarrow\bb{L}^\infty$ and~\eqref{equ:u0n eps con} yield
	\[
	\norm{\pa_t\bff{u}_n^\epsilon(0)}{\bb{L}^2}^2
	+
	\epsilon
	\norm{\nabla\pa_t\bff{u}_n^\epsilon(0)}{\bb{L}^2}^2
	\lesssim
	\norm{\pa_t\bff{u}_n^\epsilon(0)}{\bb{L}^2},
	\]
	which implies
	\[
	\norm{\pa_t\bff{u}_n^\epsilon(0)}{\bb{L}^2}^2
	+
	\epsilon
	\norm{\nabla\pa_t\bff{u}_n^\epsilon(0)}{\bb{L}^2}^2
	\lesssim
	1.
	\]
	This, together with~\eqref{equ:pat une}, proves the lemma.
\end{proof}

\begin{lemma}\label{lem:nab Lap une}
	Assume that \eqref{equ:u0 eps con} and~\eqref{equ:u0n H2} hold for $s=2$. For any~$t\in[0,T^\ast]$, there exists a constant $c$ such that for every $n\in\bb{N}$ and $\epsilon\in [0,1)$, 
	\begin{align}
		\label{equ:nab Lap une}
		\norm{\Delta\bff{u}_n^\epsilon(t)}{\bb{L}^2}^2
		+
		\epsilon \norm{\nabla\Delta\bff{u}_n^\epsilon(t)}{\bb{L}^2}^2
		+
		\int_0^t \norm{\nabla \Delta \bff{u}_n^\epsilon(s)}{\bb{L}^2}^2 \ds 
		&\leq 
		c.
	\end{align}
	The constant $c$ in \eqref{equ:nab Lap une} is independent of~$n$ and~$\epsilon$, but may depend
	on~$T^\ast$.
\end{lemma}

\begin{proof}
	To prove~\eqref{equ:nab Lap une}, we
	set~$\bff{v}=2\Delta^2\bff{u}_n^\epsilon(t)\in\mathcal{S}_n$
	in~\eqref{equ:LLB eps wea equ} and obtain
	\begin{align*}
		&\ddt\norm{\Delta\bff{u}_n^\epsilon(t)}{\bb{L}^2}^2
		+
		\epsilon\ddt\norm{\nabla\Delta\bff{u}_n^\epsilon(t)}{\bb{L}^2}^2
		+
		2\kappa_1\norm{\nabla\Delta\bff{u}_n^\epsilon(t)}{\bb{L}^2}^2
		+
		2\kappa_2 \norm{\Delta \bff{u}_n^\epsilon(t)}{\bb{L}^2}^2
		\nn\\
		&=
		-2\gamma
		\inpro{\nabla\Big(\bff{u}_n^\epsilon(t)\times\Delta\bff{u}_n^\epsilon(t)\Big)}
		{\nabla\Delta\bff{u}_n^\epsilon(t)}_{\bb{L}^2}
		+2\kappa_2\mu
		\inpro{\nabla\Big(|\bff{u}_n^\epsilon(t)|^2\bff{u}_n^\epsilon(t)\Big)}
		{\nabla\Delta\bff{u}_n^\epsilon(t)}_{\bb{L}^2}.
	\end{align*}
	The rest of the proof follows the argument in Lemma~\ref{lem:H2 un} almost verbatim, replacing $\bff{u}_n$ by $\bff{u}_n^\epsilon$, thus leading to the required estimate.
\end{proof}

\begin{lemma}\label{lem:ene est nab dt u}
	Assume that~\eqref{equ:u0 eps con} and~\eqref{equ:u0n H2} hold for $s=2$. For any~$t\in[0,T^\ast]$ and $\alpha$ satisfying~\eqref{equ:condition alpha}, there exists a constant $c$ such that for every $n\in\bb{N}$ and $\epsilon\in [0,1)$, 
	\begin{align*}
		\norm{\nabla \pa_t \bff{u}_n^\epsilon(t)}{\bb{L}^2}^2
		&+
		\epsilon \norm{\Delta \pa_t\bff{u}_n^\epsilon(t)}{\bb{L}^2}^2
		+
		\int_0^t
		\norm{\Delta \pa_t\bff{u}_n^\epsilon(s)}{\bb{L}^2}^2 \ds
		+
		\int_0^t
		\norm{\pa_t\bff{u}_n^\epsilon(s)}{\bb{H}^{1+\alpha}}^2 \ds
		\le c,
	\end{align*}
	where the constant~$c$ is independent of~$n$ and~$\epsilon$, but may depend on $T^\ast$.
\end{lemma}

\begin{proof}
	Putting $\bff{v}=-2\Delta \pa_t \bff{u}_n^\epsilon$ in~\eqref{equ:dt2 une} and integrating over $(0,t)$, we obtain
	\begin{align}\label{equ:nab dt un L2}
		&\norm{\nabla \pa_t \bff{u}_n^\epsilon(t)}{\bb{L}^2}^2
		+
		\epsilon \norm{\Delta \pa_t \bff{u}_n^\epsilon(t)}{\bb{L}^2}^2
		+
		2\kappa_1 \int_0^t \norm{\Delta \pa_t \bff{u}_n^\epsilon(s)}{\bb{L}^2}^2 \ds 
		+
		2\kappa_2 \int_0^t \norm{\nabla \pa_t \bff{u}_n^\epsilon(s)}{\bb{L}^2}^2 \ds 
		\nn \\
		&=
		\norm{\nabla \pa_t \bff{u}_n^\epsilon(0)}{\bb{L}^2}^2
		+
		\epsilon \norm{\Delta \pa_t \bff{u}_n^\epsilon(0)}{\bb{L}^2}^2
		-
		2\gamma \int_0^t \inpro{\pa_t \bff{u}_n^\epsilon(s) \times \Delta \bff{u}_n^\epsilon(s)}{\Delta \pa_t \bff{u}_n^\epsilon(s)}_{\bb{L}^2} \ds 
		\nn \\
		&\quad
		+
		2\kappa_2\mu \int_0^t \inpro{\pa_t \big(|\bff{u}_n^\epsilon(s)|^2\big) \bff{u}_n^\epsilon(s)
			+
			|\bff{u}_n^\epsilon(s)|^2 \pa_t\bff{u}_n^\epsilon(s)}{\Delta \pa_t \bff{u}_n^\epsilon(s)}_{\bb{L}^2} \ds
		\nn\\
		&=
		\norm{\nabla \pa_t \bff{u}_n^\epsilon(0)}{\bb{L}^2}^2
		+
		\epsilon \norm{\Delta \pa_t \bff{u}_n^\epsilon(0)}{\bb{L}^2}^2
		+ T_1 + T_2.
	\end{align}
	We will estimate the last two terms on the right-hand side. Firstly, by H\"older's inequality we have
	\begin{align*}
		T_1
		&\leq
		2\gamma
		\int_0^t \norm{\pa_t \bff{u}_n^\epsilon(s)}{\bb{L}^4} \norm{\Delta \bff{u}_n^\epsilon(s)}{\bb{L}^4} 
		\norm{\Delta \pa_t \bff{u}_n^\epsilon(s)}{\bb{L}^2} \ds 
		\\
		&\leq
		c \int_0^t \norm{\pa_t \bff{u}_n^\epsilon(s)}{\bb{H}^1}^2 \norm{\Delta \bff{u}_n^\epsilon(s)}{\bb{H}^1}^2 
		+
		\frac{\kappa_1}{2} \int_0^t \norm{\Delta \pa_t \bff{u}_n^\epsilon(s)}{\bb{L}^2}^2 \ds,
	\end{align*}
	where in the last step we used Young's inequality and Sobolev embedding $\bb{H}^1\hookrightarrow \bb{L}^4$. The constant $c$ is independent of $\epsilon$. For the last term, by similar argument we have
	\begin{align*}
		T_2
		&\leq
		2\kappa_2 \mu
		\int_0^t \norm{\pa_t \bff{u}_n^\epsilon(s)}{\bb{L}^6}
		\norm{\bff{u}_n^\epsilon(s)}{\bb{L}^6}^2 
		\norm{\Delta \pa_t \bff{u}_n^\epsilon(s)}{\bb{L}^2} \ds 
		\\
		&\leq
		c\int_0^t \norm{\pa_t \bff{u}_n^\epsilon(s)}{\bb{H}^1}^2 \norm{ \bff{u}_n^\epsilon(s)}{\bb{H}^1}^4 
		+
		\frac{\kappa_1}{2} \int_0^t \norm{\Delta \pa_t \bff{u}_n^\epsilon(s)}{\bb{L}^2}^2 \ds.
	\end{align*}
	Substituting these into~\eqref{equ:nab dt un L2} and adding the term $\norm{\pa_t \bff{u}_n^\epsilon(t)}{\bb{L}^2}^2$ on both sides give
	\begin{align*}
		&\norm{\pa_t \bff{u}_n^\epsilon(t)}{\bb{H}^1}^2
		+
		\epsilon \norm{\Delta \pa_t \bff{u}_n^\epsilon(t)}{\bb{L}^2}^2
		+
		\kappa_1 \int_0^t \norm{\Delta \pa_t \bff{u}_n^\epsilon(s)}{\bb{L}^2}^2 \ds 
		+
		2\kappa_2 \int_0^t \norm{\nabla \pa_t \bff{u}_n^\epsilon(s)}{\bb{L}^2}^2 \ds 
		\\
		&\leq
		\norm{\nabla \pa_t \bff{u}_n^\epsilon(0)}{\bb{L}^2}^2
		+
		\epsilon \norm{\Delta \pa_t \bff{u}_n^\epsilon(0)}{\bb{L}^2}^2
		+
		\norm{\pa_t \bff{u}_n^\epsilon(t)}{\bb{L}^2}^2
		+
		c\int_0^t \left(1+\norm{\Delta \bff{u}_n^\epsilon(s)}{\bb{H}^1}^2\right) \norm{\pa_t \bff{u}_n^\epsilon(s)}{\bb{H}^1}^2 \ds,
	\end{align*}
	where in the last step we used~\eqref{equ:un H1 H2}. By Gronwall's inequality (noting Lemma~\ref{lem:ene est 2}), we obtain
	\begin{align*}
		&\norm{\pa_t \bff{u}_n^\epsilon(t)}{\bb{H}^1}^2
		+
		\epsilon \norm{\Delta \pa_t \bff{u}_n^\epsilon(t)}{\bb{L}^2}^2
		+
		\int_0^t \norm{\Delta \pa_t \bff{u}_n^\epsilon(s)}{\bb{L}^2}^2 \ds 
		+
		\int_0^t \norm{\nabla \pa_t \bff{u}_n^\epsilon(s)}{\bb{L}^2}^2 \ds 
		\\
		&\leq
		c\left(1+\norm{\nabla \pa_t \bff{u}_n^\epsilon(0)}{\bb{L}^2}^2
		+
		\epsilon \norm{\Delta \pa_t \bff{u}_n^\epsilon(0)}{\bb{L}^2}^2\right) 
		\exp \left(\int_0^t 1+\norm{\Delta \bff{u}_n^\epsilon(s)}{\bb{H}^1}^2 \ds \right)
		\\
		&\leq
		c\left(1+\norm{\nabla \pa_t \bff{u}_n^\epsilon(0)}{\bb{L}^2}^2
		+
		\epsilon \norm{\Delta \pa_t \bff{u}_n^\epsilon(0)}{\bb{L}^2}^2\right),
	\end{align*}
	where $c$ is independent of $\epsilon$, and in the last step we used~\eqref{equ:nab Lap une}. To show the boundedness of $\norm{\nabla \pa_t \bff{u}_n^\epsilon(0)}{\bb{L}^2}^2 +
	\epsilon \norm{\Delta \pa_t \bff{u}_n^\epsilon(0)}{\bb{L}^2}^2$, we can follow the same argument as in~\eqref{equ:dt une 0} with~$\bff{v}=\Delta \pa_t \bff{u}_n^\epsilon(0)$ instead. This, together with an argument similar to \eqref{equ:unt H1a}, completes the proof.
\end{proof}

To summarise, we have the following estimates for the approximate solution $\bff{u}_n^\epsilon$ of the $\epsilon$-LLBE which are uniform in $\epsilon\in [0,1)$:
\begin{itemize}
	\item 
	If~\eqref{equ:u0 eps con} holds for $s=\max\{1,d-1\}$, then
	\begin{equation}\label{equ:une dt est local}
		\norm{\pa_t\bff{u}_n^\epsilon}{L^\infty_{T^\ast}(\bb{L}^2)}
		+
        \sqrt{\epsilon} \norm{\partial_t \bff{u}_n^\epsilon}{L^\infty_{T^\ast}(\bb{H}^1)}
        +
		\norm{\pa_t\bff{u}_n^\epsilon}{L^2_{T^\ast}(\bb{H}^1)}
		\lesssim 1.
	\end{equation}
	\item 
	If, in addition, \eqref{equ:u0 eps con} holds for $s=2$, then we also have
    \begin{equation}\label{equ:une est local}
	\begin{aligned}
		\norm{\Delta \bff{u}_n^\epsilon}{L^\infty_{T^\ast}(\bb{L}^2)}
        +
        \sqrt{\epsilon} \norm{\Delta \bff{u}_n^\epsilon}{L^\infty_{T^\ast}(\bb{H}^1)}
		+
		\norm{\Delta \bff{u}_n^\epsilon}{L^2_{T^\ast}(\bb{H}^1)}
		&\lesssim 1,
		\\
		\norm{\pa_t\bff{u}_n^\epsilon}{L^\infty_{T^\ast}(\bb{H}^1)}
		+
		\norm{\pa_t\bff{u}_n^\epsilon}{L^2_{T^\ast}(\bb{H}^1_\Delta)}
		+
		\norm{\pa_t\bff{u}_n^\epsilon}{L^2_{T^\ast}(\bb{H}^{1+\alpha})}
        &\lesssim 1,
        \\
        \sqrt{\epsilon} \norm{\pa_t\bff{u}_n^\epsilon}{L^\infty_{T^\ast}(\bb{H}^1_\Delta)}
        +
        \sqrt{\epsilon} \norm{\pa_t \bff{u}_n^\epsilon}{L^2_{T^\ast}(\bb{H}^{1+\alpha})}
		&\lesssim 1.
	\end{aligned}
\end{equation}
\end{itemize}
All constants on the right-hand sides of~\eqref{equ:une dt est local} and \eqref{equ:une est local} are independent of $\epsilon$ and $n$. Consequently, the same estimates (with $\epsilon=0$) also hold for the approximate solution $\bff{u}_n$ of the LLBE.


\begin{thebibliography}{48}
	\providecommand{\natexlab}[1]{#1}
	\providecommand{\url}[1]{\texttt{#1}}
	\expandafter\ifx\csname urlstyle\endcsname\relax
	\providecommand{\doi}[1]{doi: #1}\else
	\providecommand{\doi}{doi: \begingroup \urlstyle{rm}\Url}\fi
	
	\bibitem[Alnaes et~al.(2015)Alnaes, Blechta, Hake, Johansson, Kehlet, Logg,
	Richardson, Ring, Rognes, and Wells]{AlnaesEtal15}
	M.~S. Alnaes, J.~Blechta, J.~Hake, A.~Johansson, B.~Kehlet, A.~Logg, C.~N.
	Richardson, J.~Ring, M.~E. Rognes, and G.~N. Wells.
	\newblock The {FEniCS} project version 1.5.
	\newblock \emph{Archive of Numerical Software}, 3, 2015.
	
	\bibitem[Ang and Tran(1990)]{AngTra90}
	D.~D. Ang and T.~Tran.
	\newblock A nonlinear pseudoparabolic equation.
	\newblock \emph{Proc. Roy. Soc. Edinburgh Sect. A}, 114\penalty0
	(1-2):\penalty0 119--133, 1990.
	
	\bibitem[Arendt and Kreuter(2018)]{AreKre18}
	W.~Arendt and M.~Kreuter.
	\newblock Mapping theorems for {S}obolev spaces of vector-valued functions.
	\newblock \emph{Studia Math.}, 240\penalty0 (3):\penalty0 275--299, 2018.
	
	\bibitem[Atxitia et~al.(2010)Atxitia, Chubykalo-Fesenko, Walowski, Mann, and
	M\"unzenberg]{AtxChuWalMan10}
	U.~Atxitia, O.~Chubykalo-Fesenko, J.~Walowski, A.~Mann, and M.~M\"unzenberg.
	\newblock Evidence for thermal mechanisms in laser-induced femtosecond spin
	dynamics.
	\newblock \emph{Phys. Rev. B}, 81:\penalty0 174401, 2010.
	
	\bibitem[Atxitia et~al.(2016)Atxitia, Hinzke, and Nowak]{AtxHinNow16}
	U.~Atxitia, D.~Hinzke, and U.~Nowak.
	\newblock Fundamentals and applications of the {L}andau–{L}ifshitz–{B}loch
	equation.
	\newblock \emph{Journal of Physics D: Applied Physics}, 50\penalty0
	(3):\penalty0 033003, 2016.
	
	\bibitem[Benjamin et~al.(1972)Benjamin, Bona, and Mahony]{BBM72}
	T.~Benjamin, J.~Bona, and J.~Mahony.
	\newblock Model equations for long waves in nonlinear dispersive systems.
	\newblock \emph{Philos. Trans. Roy. Soc. London, Ser. A}, 272:\penalty0 47--78,
	1972.
	
	\bibitem[Benmouane et~al.(2024)Benmouane, Essoufi, and Ayouch]{BenEssAyo24}
	M.~Benmouane, E.-H. Essoufi, and C.~Ayouch.
	\newblock A finite element scheme for the {L}andau-{L}ifshitz-{B}loch equation.
	\newblock \emph{Comput. Appl. Math.}, 43\penalty0 (7):\penalty0 Paper No. 394,
	30, 2024.
	
	\bibitem[Bihari(1956)]{Bih56}
	I.~Bihari.
	\newblock A generalization of a lemma of {B}ellman and its application to
	uniqueness problems of differential equations.
	\newblock \emph{Acta Math. Acad. Sci. Hungar.}, 7:\penalty0 81--94, 1956.
	
	\bibitem[Brezis and Mironescu(2001)]{BreMir01}
	H.~Brezis and P.~Mironescu.
	\newblock Gagliardo-{N}irenberg, composition and products in fractional
	{S}obolev spaces.
	\newblock \emph{J. Evol. Equ.}, 1\penalty0 (4):\penalty0 387--404, 2001.
	
	\bibitem[Brezis and Mironescu(2018)]{BreMir18}
	H.~Brezis and P.~Mironescu.
	\newblock Gagliardo-{N}irenberg inequalities and non-inequalities: the full
	story.
	\newblock \emph{Ann. Inst. H. Poincar\'e{} C, Anal. Non Lin\'eaire},
	35\penalty0 (5):\penalty0 1355--1376, 2018.
	
	\bibitem[Brezis and Mironescu(2019)]{BreMir19}
	H.~Brezis and P.~Mironescu.
	\newblock Where {S}obolev interacts with {G}agliardo--{N}irenberg.
	\newblock \emph{Journal of Functional Analysis}, 277:\penalty0 2839--2864,
	2019.
	
	\bibitem[Chatzipantelidis et~al.(2006)Chatzipantelidis, Lazarov, Thom\'ee, and
	Wahlbin]{ChaLazThoWah06}
	P.~Chatzipantelidis, R.~D. Lazarov, V.~Thom\'ee, and L.~B. Wahlbin.
	\newblock Parabolic finite element equations in nonconvex polygonal domains.
	\newblock \emph{BIT}, 46:\penalty0 S113--S143, 2006.
	
	\bibitem[Chubykalo-Fesenko and Nieves(2020)]{ChuNie20}
	O.~Chubykalo-Fesenko and P.~Nieves.
	\newblock \emph{Landau-Lifshitz-Bloch Approach for Magnetization Dynamics Close
		to Phase Transition}, pages 867--893.
	\newblock Springer International Publishing, Cham, 2020.
	
	\bibitem[Chubykalo-Fesenko et~al.(2006)Chubykalo-Fesenko, Nowak, Chantrell, and
	Garanin]{ChuNowChaGar06}
	O.~Chubykalo-Fesenko, U.~Nowak, R.~W. Chantrell, and D.~Garanin.
	\newblock Dynamic approach for micromagnetics close to the {C}urie temperature.
	\newblock \emph{Phys. Rev. B}, 74:\penalty0 094436, 2006.
	
	\bibitem[Crouzeix and Thom\'{e}e(1987)]{CroTho87}
	M.~Crouzeix and V.~Thom\'{e}e.
	\newblock The stability in {$L_p$} and {$W^1_p$} of the {$L_2$}-projection onto
	finite element function spaces.
	\newblock \emph{Math. Comp.}, 48\penalty0 (178):\penalty0 521--532, 1987.
	
	\bibitem[Dauge(1992)]{Dau92}
	M.~Dauge.
	\newblock Neumann and mixed problems on curvilinear polyhedra.
	\newblock \emph{Integral Equations Operator Theory}, 15\penalty0 (2):\penalty0
	227--261, 1992.
	
	\bibitem[Di~Nezza et~al.(2012)Di~Nezza, Palatucci, and Valdinoci]{DiPalVal12}
	E.~Di~Nezza, G.~Palatucci, and E.~Valdinoci.
	\newblock Hitchhiker's guide to the fractional {S}obolev spaces.
	\newblock \emph{Bull. Sci. Math.}, 136\penalty0 (5):\penalty0 521--573, 2012.
	
	\bibitem[Emmrich(1999)]{Emm99}
	E.~Emmrich.
	\newblock \emph{Discrete versions of Gronwall's lemma and their application to
		the numerical analysis of parabolic problems}.
	\newblock Issue 637 of Preprint Reihe Mathematik. TU, Fachbereich 3, 1999.
	
	\bibitem[Ern and Guermond(2021)]{ErnGue21}
	A.~Ern and J.-L. Guermond.
	\newblock \emph{Finite elements {I}--{A}pproximation and interpolation},
	volume~72 of \emph{Texts in Applied Mathematics}.
	\newblock Springer, Cham, 2021.
	
	\bibitem[Garanin(1991)]{Gar91}
	D.~Garanin.
	\newblock Generalized equation of motion for a ferromagnet.
	\newblock \emph{Physica A: Statistical Mechanics and its Applications},
	172\penalty0 (3):\penalty0 470 -- 491, 1991.
	
	\bibitem[Garanin(1997)]{Gar97}
	D.~A. Garanin.
	\newblock {F}okker-{P}lanck and {L}andau-{L}ifshitz-{B}loch equations for
	classical ferromagnets.
	\newblock \emph{Phys. Rev. B}, 55:\penalty0 3050--3057, 1997.
	
	\bibitem[Gilbert(1955)]{Gil55}
	T.~Gilbert.
	\newblock A {L}agrangian formulation of the gyromagnetic equation of the
	magnetic field.
	\newblock \emph{Phys Rev}, 100:\penalty0 1243--1255, 1955.
	
	\bibitem[Goldys et~al.(2025)Goldys, Jiao, and Le]{GolJiaLe22}
	B.~Goldys, C.~Jiao, and K.-N. Le.
	\newblock Numerical method and error estimate for stochastic
	{L}andau-{L}ifshitz-{B}loch equation.
	\newblock \emph{IMA J. Numer. Anal.}, 45\penalty0 (3):\penalty0 1821--1867,
	2025.
	
	\bibitem[Grisvard(2011)]{Gri11}
	P.~Grisvard.
	\newblock \emph{Elliptic problems in nonsmooth domains}, volume~69 of
	\emph{Classics in Applied Mathematics}.
	\newblock Society for Industrial and Applied Mathematics (SIAM), Philadelphia,
	PA, 2011.
	
	\bibitem[Heywood and Rannacher(1986)]{HeyRan86}
	J.~G. Heywood and R.~Rannacher.
	\newblock Finite element approximation of the nonstationary {N}avier-{S}tokes
	problem. {II}. {S}tability of solutions and error estimates uniform in time.
	\newblock \emph{SIAM J. Numer. Anal.}, 23\penalty0 (4):\penalty0 750--777,
	1986.
	
	\bibitem[Jerison and Kenig(1989)]{JerKen89}
	D.~Jerison and C.~E. Kenig.
	\newblock The functional calculus for the {L}aplacian on {L}ipschitz domains.
	\newblock In \emph{Journ\'ees ``\'Equations aux {D}\'eriv\'ees {P}artielles''
		({S}aint {J}ean de {M}onts, 1989)}, pages Exp. No. IV, 10. \'Ecole Polytech.,
	Palaiseau, 1989.
	
	\bibitem[Ju(2000)]{Ju00}
	N.~Ju.
	\newblock Numerical analysis of parabolic {$p$}-{L}aplacian: approximation of
	trajectories.
	\newblock \emph{SIAM J. Numer. Anal.}, 37\penalty0 (6):\penalty0 1861--1884,
	2000.
	
	\bibitem[Kazantseva et~al.(2008)Kazantseva, Hinzke, Nowak, Chantrell, Atxitia,
	and Chubykalo-Fesenko]{KazEtal08}
	N.~Kazantseva, D.~Hinzke, U.~Nowak, R.~W. Chantrell, U.~Atxitia, and
	O.~Chubykalo-Fesenko.
	\newblock Towards multiscale modeling of magnetic materials: Simulations of
	{F}e{P}t.
	\newblock \emph{Phys. Rev. B}, 77:\penalty0 184428, 2008.
	
	\bibitem[Landau and Lifshitz(1935)]{LL35}
	L.~Landau and E.~Lifshitz.
	\newblock On the theory of the dispersion of magnetic permeability in
	ferromagnetic bodies.
	\newblock \emph{Phys Z Sowjetunion}, 8:\penalty0 153--168, 1935.
	
	\bibitem[Larsson(1989)]{Lar89}
	S.~Larsson.
	\newblock The long-time behavior of finite-element approximations of solutions
	to semilinear parabolic problems.
	\newblock \emph{SIAM J. Numer. Anal.}, 26\penalty0 (2):\penalty0 348--365,
	1989.
	
	\bibitem[Le(2016)]{Le16}
	K.~N. Le.
	\newblock Weak solutions of the {L}andau-{L}ifshitz-{B}loch equation.
	\newblock \emph{J. Differential Equations}, 261\penalty0 (12):\penalty0
	6699--6717, 2016.
	
	\bibitem[Li(2022)]{Li22}
	B.~Li.
	\newblock Maximum-norm stability of the finite element method for the {N}eumann
	problem in nonconvex polygons with locally refined mesh.
	\newblock \emph{Math. Comp.}, 91\penalty0 (336):\penalty0 1533--1585, 2022.
	
	\bibitem[Li and Zhang(2017)]{LiZha17}
	B.~Li and Z.~Zhang.
	\newblock Mathematical and numerical analysis of the time-dependent
	{G}inzburg-{L}andau equations in nonconvex polygons based on {H}odge
	decomposition.
	\newblock \emph{Math. Comp.}, 86\penalty0 (306):\penalty0 1579--1608, 2017.
	
	\bibitem[Li et~al.(2021)Li, Guo, Liu, and Liu]{LiGuoLiu21}
	Q.~Li, B.~Guo, F.~Liu, and W.~Liu.
	\newblock Weak and strong solutions to {L}andau-{L}ifshitz-{B}loch-{M}axwell
	equations with polarization.
	\newblock \emph{J. Differential Equations}, 286:\penalty0 47--83, 2021.
	
	\bibitem[Lin et~al.(1991)Lin, Thom\'{e}e, and Wahlbin]{LinThoWah91}
	Y.~P. Lin, V.~Thom\'{e}e, and L.~B. Wahlbin.
	\newblock Ritz-{V}olterra projections to finite-element spaces and applications
	to integrodifferential and related equations.
	\newblock \emph{SIAM J. Numer. Anal.}, 28\penalty0 (4):\penalty0 1047--1070,
	1991.
	
	\bibitem[Lions(1969)]{Lio69}
	J.~L. Lions.
	\newblock \emph{Quelques M\'ethodes de R\'esolution des Probl\`emes aux Limites
		Non Lin\'eaires}.
	\newblock Dunod Gauthier-Villars, Paris, 1969.
	
	\bibitem[Meo et~al.(2020)Meo, Pantasri, Daeng-am, Rannala, Ruta, Chantrell,
	Chureemart, and Chureemart]{MeoPan20}
	A.~Meo, W.~Pantasri, W.~Daeng-am, S.~E. Rannala, S.~I. Ruta, R.~W. Chantrell,
	P.~Chureemart, and J.~Chureemart.
	\newblock Magnetization dynamics of granular heat-assisted magnetic recording
	media by means of a multiscale model.
	\newblock \emph{Phys. Rev. B}, 102:\penalty0 174419, 2020.
	
	\bibitem[Nieves and Chubykalo-Fesenko(2016)]{NieChu16}
	P.~Nieves and O.~Chubykalo-Fesenko.
	\newblock Modeling of ultrafast heat- and field-assisted magnetization dynamics
	in {F}e{P}t.
	\newblock \emph{Phys. Rev. Appl.}, 5:\penalty0 014006, 2016.
	
	\bibitem[Novick-Cohen and Pego(1991)]{NovPeg91}
	A.~Novick-Cohen and R.~L. Pego.
	\newblock Stable patterns in a viscous diffusion equation.
	\newblock \emph{Trans. Amer. Math. Soc.}, 324\penalty0 (1):\penalty0 331--351,
	1991.
	
	\bibitem[Schatz et~al.(1998)Schatz, Thom\'ee, and Wahlbin]{SchThoWah98}
	A.~H. Schatz, V.~Thom\'ee, and L.~B. Wahlbin.
	\newblock Stability, analyticity, and almost best approximation in maximum norm
	for parabolic finite element equations.
	\newblock \emph{Comm. Pure Appl. Math.}, 51\penalty0 (11-12):\penalty0
	1349--1385, 1998.
	
	\bibitem[Showalter and Ting(1970)]{ShoTin70}
	R.~Showalter and T.~Ting.
	\newblock Pseudo-parabolic partial differential equations.
	\newblock \emph{SIAM J. Math. Anal.}, 1:\penalty0 1--26, 1970.
	
	\bibitem[Soenjaya(2025)]{Soe25}
	A.~L. Soenjaya.
	\newblock Mixed finite element methods for the
	{L}andau--{L}ifshitz--{B}aryakhtar and the regularised
	{L}andau--{L}ifshitz--{B}loch equations in micromagnetics.
	\newblock \emph{J. Sci. Comput.}, 103\penalty0 (2):\penalty0 Paper No. 65,
	2025.
	
	\bibitem[Soenjaya and Tran(2023)]{SoeTra23}
	A.~L. Soenjaya and T.~Tran.
	\newblock Global solutions of the {L}andau--{L}ifshitz--{B}aryakhtar equation.
	\newblock \emph{J. Differential Equations}, 371:\penalty0 191--230, 2023.
	
	\bibitem[Stephan and Tran(2021)]{SteTra21}
	E.~P. Stephan and T.~Tran.
	\newblock \emph{Schwarz methods and multilevel preconditioners for boundary
		element methods}.
	\newblock Springer, Cham, 2021.
	
	\bibitem[Sugiyama(1969)]{Sug69}
	S.~Sugiyama.
	\newblock On the stability problems of difference equations.
	\newblock \emph{Bull. Sci. Engrg. Res. Lab. Waseda Univ.}, 45:\penalty0
	140--144, 1969.
	
	\bibitem[Thom\'ee(1973)]{Tho73}
	V.~Thom\'ee.
	\newblock Polygonal domain approximation in {D}irichlet's problem.
	\newblock \emph{J. Inst. Math. Appl.}, 11:\penalty0 33--44, 1973.
	
	\bibitem[Vogler et~al.(2014)Vogler, Abert, Bruckner, and Suess]{VogAbe14}
	C.~Vogler, C.~Abert, F.~Bruckner, and D.~Suess.
	\newblock {L}andau-{L}ifshitz-{B}loch equation for exchange-coupled grains.
	\newblock \emph{Phys. Rev. B}, 90:\penalty0 214431, 2014.
	
	\bibitem[Zhu and Li(2013)]{ZhuLi13}
	J.-G. Zhu and H.~Li.
	\newblock Understanding signal and noise in heat assisted magnetic recording.
	\newblock \emph{IEEE Transactions on Magnetics}, 49\penalty0 (2):\penalty0
	765--772, 2013.
	
\end{thebibliography}

\end{document}